\documentclass{amsart}

\usepackage{amsmath}
\usepackage{amscd}
\usepackage{amsthm}
\usepackage{amsfonts}
\usepackage{amssymb}

\usepackage{eucal}		

\usepackage{lscape}
\usepackage{longtable}
\usepackage{afterpage}

\usepackage{graphicx}

\renewcommand\ge\geqslant
\renewcommand\geq\geqslant
\renewcommand\le\leqslant
\renewcommand\leq\leqslant
\renewcommand\propto\varpropto

\newcommand{\sn}{\mathop{{}\rm sn}}
\newcommand{\cn}{\mathop{{}\rm cn}}

\newcommand{\slemn}{\mathop{{}\rm sl}}
\newcommand{\clemn}{\mathop{{}\rm cl}}
\newcommand{\sm}{\mathop{{}\rm sm}}
\newcommand{\cm}{\mathop{{}\rm cm}}

\newcommand{\defeq}{\mathrel{:=}}
\newcommand{\eqdef}{\mathrel{=:}}

\newcommand{\notdivides}{\nmid}

\theoremstyle{plain}
\newtheorem{theorem}{Theorem}[section]
\newtheorem{proposition}[theorem]{Proposition}
\newtheorem{lemma}[theorem]{Lemma}

\theoremstyle{definition}
\newtheorem{definition}[theorem]{Definition}
\newtheorem{example}[theorem]{Example}
\newtheorem*{example*}{Example}

\theoremstyle{remark}
\newtheorem*{remark*}{Remark}

\numberwithin{equation}{section}

\makeatletter
\renewcommand{\LT@makecaption}[3]{%
  \LT@mcol\LT@cols c{\hbox to\z@{\hss\parbox[t]\LTcapwidth{%
        \sbox\@tempboxa{%
          {\@captionheadfont#1{#2}}.%
          \@captionfont\upshape\enspace#3}%
        \ifdim\wd\@tempboxa>\hsize
            {\@captionheadfont#1{#2}}.%
            \@captionfont\upshape\enspace#3%
            \else
            \hbox to\hsize{\hfil\box\@tempboxa\hfil}%
            \fi
            \endgraf\vskip\baselineskip}%
      \hss}}}
\makeatother 

\makeatletter
\newenvironment{sizeequation}[1]{%
  \skip@=\baselineskip
  #1%
  \baselineskip=\skip@
  \equation
}{\endequation \ignorespacesafterend}

\newenvironment{sizealign}[1]{%
  \skip@=\baselineskip
  #1%
  \baselineskip=\skip@
  \align
}{\endalign \ignorespacesafterend}  
  
\newenvironment{sizemultline}[1]{%
  \skip@=\baselineskip
  #1%
  \baselineskip=\skip@
  \multline
}{\endmultline \ignorespacesafterend}  
  
\makeatother

\begin{document}

\title[Quadratic differential systems and Chazy equations]{Quadratic differential systems\\ and Chazy equations,~I}
\author{Robert S. Maier}
\address{Depts.\ of Mathematics and Physics, University of Arizona, Tucson,
AZ 85721, USA}
\email{rsm@math.arizona.edu}
\date{}

\begin{abstract}
Generalized Darboux--Halphen (gDH) systems, which form a versatile class of
three-dimensional homogeneous quadratic differential systems (HQDS's), are
introduced.  They generalize the Darboux--Halphen (DH) systems considered
by other authors, in~that any non\nobreakdash-DH gDH system is affinely but
not projectively covariant.  It is shown that the gDH class supports a rich
collection of rational solution-preserving maps: morphisms that transform
one gDH system to another.  The proof relies on a bijection between (i)~the
solutions with noncoincident components of any `proper' gDH system, and
(ii)~the solutions of a generalized Schwarzian equation (gSE) associated
to~it, which generalizes the Schwarzian equation (SE) familiar from the
conformal mapping of hyperbolic triangles.  The gSE can be integrated
parametrically in~terms of the solutions of a Papperitz equation, which is
a generalized Gauss hypergeometric equation.  Ultimately, the rational gDH
morphisms come from hypergeometric transformations.  A~complete
classification of proper non\nobreakdash-DH gDH systems with the Painlev\'e
property (PP) is also carried~out, showing how some are related by rational
morphisms.  The classification follows from that of non\nobreakdash-SE
gSE's with the~PP, due to Garnier and Carton-LeBrun.  As examples, several
non\nobreakdash-DH gDH systems with the~PP are integrated explicitly
in~terms of elementary and elliptic functions.  
\end{abstract}


\maketitle

\section{Introduction: (g)DH Systems, Chazy Equations}
\label{sec:intro}

Autonomous differential systems $\dot x=Q(x)$,
$x=(x_1,\dots,x_d)\in{\mathbb{K}}^d$, where the base field $\mathbb{K}$ is
$\mathbb{R}$ or~$\mathbb{C}$ and each component~$Q_i$ of
$Q\colon{\mathbb{K}}^d\to{\mathbb{K}}^d$ is a homogeneous quadratic form in
$x_1,\dots,x_d$, are interesting continuous-time dynamical systems that
have found application in the physical and biological sciences.  The
independent variable will be denoted by~$\tau$.  If an initial condition
$x(\tau_0)=x^0\in\mathbb{K}^d$ is imposed at any point
$\tau=\tau_0\in\mathbb{K}$, an analytic solution $x=x(\tau)$ will exist
locally, as a power series in~$\tau-\tau_0$.  But, global behavior may be
difficult to determine.

An algebraic approach to such homogeneous quadratic differential systems
(called HQDS's here) is often useful, e.g., when studying the stability of
the point $0\in{\mathbb{K}}^d$, the possible phase portraits, and system
equivalences \cite{Kaplan79,Kinyon95,Markus60,Walcher91}.  To any HQDS
$\bigl({\mathbb{K}}^{d}\!,\ \dot{x}= Q( x)\bigr)$ there is naturally
associated an algebra $\mathfrak{A}=(\mathbb{K}^d,*)$ over~$\mathbb{K}$.
Any $x\in\mathfrak{A}$ can be written as $\sum_{i=1}^dx_ie_i$, where
$e_1,\dots,e_d$ are basis vectors for~$\mathbb{K}^d$, and the product~$*$
on~$\mathbb{K}^d$ is defined by $x*x\defeq Q(x)$ and
  \begin{equation}
    \label{eq:polarization}
    x * y \defeq  \left[ (x+y)*(x+y) - x*x - y*y    \right] / 2,
  \end{equation}
which is the usual polarization identity.  The HQDS then becomes an
evolution equation $\dot x=x*x$ in~$\mathfrak{A}$.  As defined, the
product~$*$ is commutative but non-associative, meaning not necessarily
associative; and there may be no corresponding multiplicative identity
element.  Any idempotent $p\in\mathfrak{A}$, satisfying $p*p=p$, yields a
1\nobreakdash-parameter family of ray solutions of the HQDS, namely
$x(\tau)=\allowbreak -(\tau-\tau_*)^{-1}p$, and any nilpotent
$n\in\mathfrak{A}$, satisfying $n*n=0$, yields a 1\nobreakdash-parameter
family (i.e., ray) of constant solutions, namely $x(\tau)\equiv K\,n$,
$K\in\mathbb{K}$.

An analytic~map $\Phi\colon{\mathbb{K}^d}\to{\mathbb{K}^d}$ with
$\Phi(0)=0$ (or~more generally, a germ of an analytic function) is called
\emph{solution-preserving} from an HQDS
$\bigl({\mathbb{K}}^{d}\!,\ \dot{\tilde x}=\tilde Q(\tilde x)\bigr)$ to an
HQDS $\bigl({\mathbb{K}}^{d}\!,\ \dot x=Q(x)\bigr)$, if $x(\tau)\defeq
\Phi\left(\tilde x(\tau)\right)$ is a local solution of $\dot x=Q(x)$ for
each local solution $\tilde x=\tilde x(\tau)$ of $\dot {\tilde x}=\tilde
Q(\tilde x)$ with initial condition $\tilde x(\tau_0)$ sufficiently near
the origin~\cite{Kinyon95,Walcher91}.  One writes $Q=\Phi_*\tilde Q$.  A
straightforward case is when $\Phi$~is linear and invertible, and hence
everywhere defined.  If $\Phi$~equals some $T\in{\it GL}(d,\mathbb{K})$
then $Q=T \circ \tilde Q \circ T^{-1}$, and the corresponding algebras
${\mathfrak{A}},\tilde{\mathfrak{A}}$ are isomorphic.

The classification of the qualitative behaviors of
$d$\nobreakdash-dimensional HQDS's thus reduces to placing commutative,
non-associative algebras~$\mathfrak{A}=(\mathbb{K}^d,*)$ into equivalence
classes of various sorts, often with the aid of algebraic invariants.  The
space of commutative products~$*$ on~${\mathbb{K}}^d$ has $d^2(d+1)/2$
parameters, of which such invariants are functions.  A full classification
when $d=2$ has been carried out \cite{Date79,Markus60,Sibirsky88}; and has
revealed, \emph{inter alia}, which real planar HQDS's have unbounded
solutions~\cite{Dickson70}.  But the structure of the
18\nobreakdash-parameter space of non-associative algebras with $d=3$,
i.e., of quadratic vector fields on ${\mathbb{R}}^3$ or~${\mathbb{C}}^3$,
is not fully understood.

Any three-dimensional HQDS $\dot x=Q(x)\eqdef x*x$, with 18~parameters, can
be written component-wise as
\begin{equation}
  \label{eq:HQDS0}
  \left\{
  \begin{aligned}
    \dot{  x}_1 &= a_{11}  x_1^2 + a_{12}  x_2^2 + a_{13}  x_3^2  + b_{11}  x_2  x_3 + b_{12}  x_3  x_1 + b_{13}  x_1  x_2,\\
    \dot{  x}_2 &= a_{21}  x_1^2 + a_{22}  x_2^2 + a_{23}  x_3^2  + b_{21}  x_2  x_3 + b_{22}  x_3  x_1 + b_{23}  x_1  x_2,\\
    \dot{  x}_3 &= a_{31}  x_1^2 + a_{32}  x_2^2 + a_{33}  x_3^2  + b_{31}  x_2  x_3 + b_{32}  x_3  x_1 + b_{33}  x_1  x_2.
  \end{aligned}
  \right.
\end{equation}
It is not known exactly which such systems are integrable, in any sense;
though Kovalevskaya exponents~\cite{Goriely2001} are a useful tool.
Conditions in~terms of them for the existence of polynomial first
integrals~\cite{Tsygvintsev2000} and for algebraic
integrability~\cite{Yoshida87a} are known.  But only HQDS's of certain
types, such as Lotka--Volterra systems, have been subjected to detailed
Kovalevskaya--Painlev\'e analyses
\cite{Bountis84,Constandinides2011,Goriely92,Grammaticos89,Hoyer1879}, or
other integrability
analyses~\cite{GonzalezGascon2000,Maciejewski2001,Moulin99}.  Nor have the
cases when the individual components $x_1,x_2,x_3$ of a solution
$x=x(\tau)$ satisfy `nice' nonlinear scalar ODE's of low degree been fully
characterized.  But the piecemeal integration of HQDS's of the
form~(\ref{eq:HQDS0}), with the aid of elliptic, hypergeometric, or other
special functions, has been carried~on since the 19th-century work of
Darboux, Halphen~\cite{Halphen1881c}, and Brioschi~\cite{Brioschi1881}; and
Hoyer~\cite{Hoyer1879}.

A theme of the present paper, and especially of Part~II, is the connection
between HQDS's on~${\mathbb{C}}^3$ and nonlinear scalar ODE's in the
\emph{Chazy class}, which are sometimes satisfied by components
$x_i=x_i(\tau)$, or their linear combinations.  The Chazy class includes
all third-order autonomous scalar ODE's of the form
\begin{equation}
\label{eq:chazyclass}
\dddot u =  A\,u\ddot u + B\,\dot u^2 + C\,u^2\dot u + D\,u^4.
\end{equation}
(Besides the invariance under translation of~$\tau$, note the invariance
under $\tau\mapsto\tau/\lambda$ when accompanied by $u\mapsto \lambda u$ or
by $(A,B,C,D)\mapsto(\lambda A,\lambda B,\lambda^2 C,\lambda^3 D)$.)  Such
ODE's have analytic local solutions $u=u(\tau)$, which may not extend
analytically or meromorphically to $\mathbb{C}\ni\tau$.  ODE's in the Chazy
class appear in symmetry reductions of the self-dual Yang--Mills equations
for the gauge group $\textrm{Diff}(S^3)$ \cite{Ablowitz2003}, and of the
Prandtl boundary layer equations~\cite{Ludlow2000,Naz2008}; and also in the
construction of ${\it SU}(2)$-invariant hypercomplex
manifolds~\cite{Hitchin98}.  In many applications what appears is actually
a HQDS, from which a Chazy-class ODE can optionally be derived.

The Chazy class was introduced in an early (1911) extension of Painlev\'e's
classification of nonlinear scalar ODE's from the second to the third
order.  Chazy (\cite{Chazy11};~see also~\cite{Bureau64b}) imposed the
Painlev\'e property (PP), according to which no~local solution of the ODE
may have a branch point; at~least, not one that is movable, with a
location depending on the choice of initial conditions.  Nearly all
nonlinear third-order scalar ODE's that have the~PP, and lie in the
`polynomial class' (i.e., express $\dddot u$ as a polynomial in $u,\dot u,
\ddot u$), turn~out to reduce to the form~(\ref{eq:chazyclass}) when
recessive, sub-dominant terms are omitted.  Up~to normalization there are
12~possible nonzero choices for the coefficient vector $(A,B,C,D)$;
equivalently, 12~choices for the element $[A:B:C:D]$ of the weighted
projective space $\mathbb{P}^3(1,1,2,3)$.  The corresponding ODE's are now
labeled Chazy-I through Chazy-XII~\cite{Cosgrove2000,Sasano2010}, with
Chazy-XI and~XII being infinite families, parametrized by a positive
integer~$N$.

The best known Chazy equations are Chazy-III and~XII.  In those two
equations, $[A:B:C:D]$ is respectively
\begin{displaymath}
[2:-3:0:0],\qquad [2(N^2-36):-3(N^2+12):48(N^2-36):-4(N^2-36)^2].
\end{displaymath}
The former, often called the classical Chazy equation or simply \emph{the}
Chazy equation, is the $N\to\infty$ limit of the latter (in which
$N\neq1,6$ for the PP to obtain).  Generically, any solution of
Chazy\nobreakdash-III or of Chazy\nobreakdash-XII with $N\neq\allowbreak
2,3,4,5$, determined by initial data $(u,\dot u,\ddot u)$ at some point
$\tau=\tau_0$, is only locally defined and has a \emph{natural boundary} in
the complex $\tau$-plane, beyond which it cannot be analytically or
meromorphically continued.  The maximal domain of definition of the
solution, extending from $\tau_0$ to this movable boundary, is a disk or a
half-plane~\cite{Joshi93}.

Chazy himself showed that Chazy-III and~XII can be integrated
\emph{parametrically} with the aid of a particular second-order
\emph{linear} ODE: the Gauss hypergeometric equation (GHE), which is
satisfied by the hypergeometric function ${}_2F_1(t)$.  The independent
variable~$\tau$ of the Chazy equation is represented as the ratio of a pair
of solutions of a certain GHE, and the dependent variable~$u$ as the
logarithmic derivative of one of the pair.  There is an alternative
interpretation of solutions of Chazy-III and~XII, involving a HQDS\null.
For both ODE's, any local solution $u=u(\tau)$ can be obtained from a local
solution $x=x(\tau)$ of a corresponding 3\nobreakdash-dimensional HQDS
$\dot x=Q(x)$ (a~permutation-invariant Darboux--Halphen system) as an
average of components: $u=(x_1+x_2+x_3)/3$.  Recently, Chakravarty and
Ablowitz~\cite{Chakravarty2010} produced several additional representations
of the solutions of Chazy-III, as combinations of components of the
solutions of other such systems.

The focus of Parts I and~II of this paper is therefore on a versatile and
rather general class of 3\nobreakdash-dimensional HQDS's of the form
$\left(\mathbb{C}^3,\ \dot x=Q(x)\eqdef x*x\right)$, which will be called
\emph{generalized Darboux--Halphen} (gDH) \emph{systems}.  By definition,
any gDH system can be written component-wise as
\begin{equation}
\label{eq:gDH}
\left\{
\begin{aligned}
  \dot x_1 = -a_1(x_1-x_2)(x_3-x_1) + (b_1\,x_2x_3+b_2\,x_3x_1+b_3\,x_1x_2) - c\,x_2x_3, \\
  \dot x_2 = -a_2(x_2-x_3)(x_1-x_2) + (b_1\,x_2x_3+b_2\,x_3x_1+b_3\,x_1x_2) - c\,x_3x_1, \\
  \dot x_3 = -a_3(x_3-x_1)(x_2-x_3) + (b_1\,x_2x_3+b_2\,x_3x_1+b_3\,x_1x_2) - c\,x_1x_2 ,
\end{aligned}
\right.
\end{equation}
with $(a_1,a_2,a_3;b_1,b_2,b_3;c)\in
{\mathbb{C}}^3\times{\mathbb{C}}^3\times{\mathbb{C}}$.  This
is the specialization
\begin{equation}
a_{ij}=a_i\delta_{ij},\qquad b_{ij}=(2a_i-c)\delta_{ij} - a_i + b_j  
\end{equation}
of the general HQDS~(\ref{eq:HQDS0}).  The system~(\ref{eq:gDH}) will be
denoted by ${\rm gDH}(a_1,\nobreak a_2,\nobreak a_3;\allowbreak
b_1,\nobreak b_2,\nobreak b_3;c)$.  Such systems make~up a
dimension\nobreakdash-7 subspace of the linear space of
3\nobreakdash-dimensional HQDS's.  Topics to be treated include the
integration of gDH systems, by an implicitly algebro-geometric procedure
that generalizes the integration of Chazy-III and~XII; rational but
nonlinear morphisms between gDH systems; the Painlev\'e properties of gDH
systems; the nonlinear ODE's satisfied by (linear combinations~of) gDH
components~$x_i$; and the extent to which solutions of the Chazy equations
can be represented in this way.

The general linear group ${\it GL}(3,\mathbb{C})$ acting on
$x\in{\mathbb{C}}^3$, resp.\ the special linear group ${\it
  SL}(3,\mathbb{C})$, has~9, resp.~8 parameters.  Hence up~to linear
equivalence, the space of 3\nobreakdash-dimensional HQDS's has only
$18-\nobreak9=9$ parameters; or $18-\nobreak8=10$ if equivalence of HQDS's
under $x\mapsto\lambda x$, i.e.\ under $\tau\mapsto\lambda\tau$, is not
considered.  Also, no~nontrivial 1\nobreakdash-parameter group of
transformations $T\in{\it SL}(3,\mathbb{C})$ stabilizes the gDH subspace.
It follows that up~to linear equivalence, the gDH subspace has a small
codimension in the full HQDS parameter space, i.e., $10-\nobreak7=3$.
That~is, it is fairly generic.  However, gDH systems are significantly
different from the parametrized 3\nobreakdash-dimensional HQDS's
(of~Lotka--Volterra or Hoyer form) that have been integrated or tested for
integrability by other authors, as mentioned above.

If a condition $x=x^0=(x_1^0,x_2^0,x_3^0)\in\mathbb{C}^3$ is imposed at any
$\tau_0\in\mathbb{C}$, an analytic solution $x=x(\tau)$ of any gDH system
will exist for $\tau$~in a neighborhood of~$\tau_0$.  It is easy to see
that if two components coincide at~$\tau_0$, they will do so at all~$\tau$.
(If $x_j=x_k$ then $\dot x_j=\dot x_k$.)  In fact, any gDH system with
$c\neq b_1+b_2+b_3$ has a 1\nobreakdash-parameter family of meromorphic ray
solutions, proportional to~$(\tau-\tau_*)^{-1}$, i.e.
\begin{equation}
\label{eq:raysoln}
  x(\tau)=(c-b_1-b_2-b_3)^{-1}(\tau-\tau_*)^{-1}(1,1,1),
\end{equation}
in which all components coincide.  This family of `scale-invariant'
solutions comes from an idempotent of the algebra~$\mathfrak{A}$ associated
to~(\ref{eq:gDH}): an element $p_0\propto\allowbreak
e_0\defeq\allowbreak e_1+\nobreak e_2+\nobreak e_3$, satisfying
$p_0*p_0=p_0$.  If $c=\allowbreak b_1+\nobreak b_2+\nobreak b_3$ then $e_0$~is
nilpotent (i.e.\ $e_0*e_0=0$), and any constant function $x(\tau)\equiv
x_1^0(1,1,1)$ will be a solution.

The $b_1=b_2=b_3\eqdef b$ case of the gDH system~(\ref{eq:gDH}) has been
considered by Bureau \cite{Bureau72,Bureau87}, but its subcase $b=c/2$ is
much better known.  Any gDH system satisfying $b_1=b_2=b_3=c/2$ will be
called a \emph{Darboux--Halphen} (DH) \emph{system} and denoted by ${\rm
  DH}(a_1,a_2,a_3;c)$.  Most attention below will be directed to
\emph{proper} DH systems, which also satisfy (i)~$c\neq0$ and (ii)
$c-a_1-a_2-a_3\neq0$.  By examination, any proper DH system can
alternatively be written as
\begin{gather}
\label{eq:DH}
\left\{
\begin{aligned}
  \dot x_1 &= (c/2)\left[x_1^2 + (1+\rho\alpha_1)(x_1-x_2)(x_3-x_1)\right], \\
  \dot x_2 &= (c/2)\left[x_2^2 + (1+\rho\alpha_2)(x_2-x_3)(x_1-x_2)\right], \\
  \dot x_3 &= (c/2)\left[x_3^2 + (1+\rho\alpha_3)(x_3-x_1)(x_2-x_3)\right],
\end{aligned}
\right.
\\[\jot]
\begin{aligned}
\rho\defeq 2/(1-\alpha_1-\alpha_2-\alpha_3),
\end{aligned}\nonumber
\end{gather}
where 
\begin{equation}
\label{eq:computealphas}
\alpha_i=-a_i/(c-\nobreak a_1-\nobreak a_2-\nobreak a_3)
\end{equation}
defines `angular' parameters $\alpha_1,\alpha_2,\alpha_3\in\mathbb{C}$,
satisfying $\alpha_1+\alpha_2+\alpha_3\neq1$.

The system~(\ref{eq:DH}) will be denoted by ${\rm
  DH}(\alpha_1,\alpha_2,\alpha_3\,|\, c)$.  It is a standardized form of
the Darboux--Halphen HQDS, which was introduced in 1881 by
Halphen~\cite{Halphen1881c} and has been studied more recently
\cite{Ablowitz2003,Harnad2000,Maier15,Ohyama96}.  For instance, it arises
in the conformal mapping of hyperbolic triangles, where the angular
parameters have units of $\pi$~radians, and
$\alpha_1+\alpha_2+\alpha_3\neq1$ is a non-Euclidean condition.  The
permutation-invariant DH system used in representing solutions of Chazy-III
and~XII, mentioned above, turns~out to be ${\rm DH}(0,0,0\,|\,c)$,
resp.\ ${\rm DH}(\frac2N,\frac2N,\frac2N\,|\,c)$.  The less symmetric
system ${\rm DH}(0,\frac13,\frac12\,|\,c)$, resp.\ ${\rm
  DH}(\frac1N,\frac13,\frac12\,|\,c)$, can also be used.

A nice feature of any proper DH system ${\rm
  DH}(\alpha_1,\alpha_2,\alpha_3\,|\,c)$ is that the corresponding
non-associative algebra $\mathfrak{A}=(\mathbb{C}^3,*)$ is \emph{unital}:
the above idempotent~$p_0$, which for a proper DH system equals
$(2/c)e_0=\allowbreak (2/c)(e_1+\nobreak e_2+\nobreak e_3)$, is a
multiplicative identity element of~$\mathfrak{A}$.  In any proper DH
system, $c\neq0$ can optionally be `scaled~out' and omitted from the
parameter vector $(\alpha_1,\alpha_2,\alpha_3\,|\,c)$.  In many contexts
its precise value is unimportant.  If unspecified, by~default its value
will be~2.

It is known that no proper DH system ${\rm
  DH}(\alpha_1,\alpha_2,\alpha_3\,|\,c)$ is algebraically
integrable~\cite{Maciejewski95,Valls2006b}.  But each DH system ${\rm
  DH}(a_1,a_2,a_3;c)$, proper or improper, has a Lie point symmetry that
the general gDH system does not.  It is stable under
$\tau\mapsto(A\tau+B)/(C\tau+D)$ for any $\pm\left(\begin{smallmatrix} A &
  B \\ C & D
\end{smallmatrix}\right)\in{\it PSL}(2,\mathbb{C})$, 
provided that each component~$x_i$ is simultaneously transformed in an
affine-linear way~\cite{Harnad2000}.  This facilitates
integration~\cite{Ablowitz2003,Chakravarty2010,Harnad2000}.  A related fact
is that if
$(\alpha_1,\alpha_2,\alpha_3)=\allowbreak(\frac1{N_1},\frac1{N_2},\frac1{N_3})$
where each $N_i$ is a positive integer or~$\infty$, with
$\alpha_1+\nobreak\alpha_2+\allowbreak\alpha_3<\nobreak1$, the solutions of
${\rm DH}(\alpha_1,\alpha_2,\alpha_3\,|\,c)$ will have a modular
interpretation~\cite{Harnad2000,Maier15}.  If an appropriate initial
condition $x(\tau_0)=x^0$ is imposed at any point~$\tau=\tau_0$ in the
upper half-plane~$\mathfrak{H}$, the components $x_i=x_i(\tau)$ will become
{\em quasi-modular forms\/}: single-valued functions on~$\mathfrak{H}$ that
are logarithmic derivatives of conventional modular forms, and transform
specially under a triangle subgroup $\Delta(N_1,N_2,N_3)$ of~${\it
  PSL}(2,{\mathbb{R}})$.  A~fundamental domain for this subgroup, acting
on~$\mathfrak{H}$, can be chosen to be a hyperbolic triangle with angles
$\pi(\frac1{N_1},\frac1{N_2},\frac1{N_3})$; and the group elements will
reflect this domain and its images through their sides, thereby
tessellating~$\mathfrak{H}$.  The real axis will be a natural boundary
through which $x=x(\tau)$ cannot be continued.

The simplest example is $(N_1,N_2,N_3)=(\infty,3,2)$, when the DH system
becomes a famous differential system of Ramanujan.  The group
$\Delta(\infty,3,2)$ is the classical modular group ${\it
  PSL}(2,{\mathbb{Z}})$.  For an appropriate initial condition
$x(\tau_0)=x^0$, the components $x_2(\tau),x_3(\tau)$ will be proportional
to logarithmic derivatives of the Eisenstein series $E_4,E_6$.  These are
well-known modular forms for ${\it PSL}(2,\mathbb{Z})$, defined
on~$\mathfrak{H}$ as $q$\nobreakdash-series: power series in the half-plane
variable $q=\exp(2\pi{\rm i}\tau)$.  For other (generic) choices of initial
condition, the solution $x=x(\tau)$ will come from the preceding by
applying an element of~${\it PSL}(2,{\mathbb{C}})$ to the independent
variable~$\tau$; so it will be defined not on~$\mathfrak{H}$, but on some
other half-plane or disk.

In general, the 3\nobreakdash-dimensionality of the group manifold of ${\it
  PSL}(2,\mathbb{C})$ corresponds to that of the space of disks (including
infinite disks, i.e.\ half-planes) in the complex $\tau$\nobreakdash-plane,
to that of the solution space of any system ${\rm
  DH}(\frac1{N_1},\frac1{N_2},\frac1{N_3})$, and to that of its space of
initial conditions $x(\tau_0)$.  The phenomenon of a natural boundary
(a~`wall of poles') may have been first discovered for Chazy-III and~XII,
which are associated to ${\rm DH}(0,\frac13,\frac12)$ and ${\rm
  DH}(\frac1N,\frac13,\frac12)$, but in~fact it is displayed by any generic
solution of ${\rm DH}(\alpha_1,\alpha_2,\alpha_3)$ with
$(\alpha_1,\alpha_2,\alpha_3) =\allowbreak
(\frac1{N_1},\frac1{N_2},\frac1{N_3})$, provided that
$\alpha_1+\nobreak\alpha_2+\allowbreak\alpha_3<\nobreak1$ is satisfied.  As
will be seen, though, non\nobreakdash-DH gDH systems differ significantly
from DH systems: their solutions have no natural boundaries.

The chief results of Part~I of this paper are the following.

\smallskip
(1)\quad{}Any ${\rm gDH}(a_1,a_2,a_3;b_1,b_2,b_3;c)$ that is `proper' (see
Definition~\ref{def:proper}) can be integrated parametrically with the aid
of the Gauss hypergeometric function ${}_2F_1(t)$.  The integration
procedure, based on Theorem~\ref{thm:integration}, extends the known
procedure for closed-form integration of proper DH systems, and hence the
integration of the Chazy-III and~XII equations in~terms of~${}_2F_1(t)$.
The accompanying Theorem~\ref{thm:integration0}, which is a special case,
applies the procedure to the integration of any ${\rm
  DH}(\alpha_1,\alpha_2,\alpha_3\,|\,c)$.  Both theorems employ not the
Gauss hypergeometric equation (GHE), but rather its generalization, the
Papperitz equation~(PE).

Remarkably, the PE\nobreakdash-based integration of proper gDH systems does
not require the presence of a 3\nobreakdash-parameter group of Lie
symmetries.  Non-DH gDH systems are covariant under affine transformations
$\tau\mapsto\allowbreak A\tau+\nobreak B$, but not under projective
transformations $\tau\mapsto\allowbreak (A\tau+\nobreak B)/\allowbreak
(C\tau+\nobreak D)$.

\smallskip
(2)\quad{}For any choice of distinct singular points
$t_1,t_2,t_3\in\mathbb{P}^1_t$, there is a correspondence between the
solutions $x=x(\tau)$ of any proper gDH system and the solutions
$t=t(\tau)$ of an associated nonlinear third-order ODE: a generalized
Schwarzian equation ${\rm gS}_{t_1,t_2,t_3}(\nu_1,\nobreak \nu_1';
\allowbreak \nu_2,\nobreak \nu_2'; \allowbreak \nu_3,\nobreak \nu_3';\bar
n)$.  (See Theorem~\ref{thm:gSE}; the parameters of the gSE are
birationally related to $(a_1,\nobreak a_2,\nobreak a_3;\allowbreak
b_1,\nobreak b_2,\nobreak b_3;c)$ and constitute an alternative parameter
vector.)  The correspondence is not quite a bijection: it excludes
solutions $x=x(\tau)$ with coincident components, such as the ray
solution~(\ref{eq:raysoln}).  If $(t_1,\nobreak t_2,\nobreak
t_3)=\allowbreak (0,\nobreak 1,\nobreak \infty)$, the correspondence is
$t=-(x_2-\nobreak x_3)/\allowbreak(x_1-\nobreak x_2)$.

The $b_1=b_2=b_3$ case of this proper
$\textrm{gDH}\leftrightarrow\textrm{gSE}$ correspondence was previously
found by Bureau~\cite{Bureau72}.  In the DH case the gSE specializes to
Eq.~(\ref{eq:SE}), a Schwarzian ODE (based on a Schwarzian derivative)
which is familiar from conformal mapping~\cite{Nehari52}.  That proper DH
systems can be integrated with the aid of Schwarzian equations is well
known, but the extension to gDH systems and gSE's, the latter having no
direct conformal mapping interpretation, is new.

\smallskip
(3)\quad{}There is a large collection of nonlinear but \emph{rational}
solution-preserving maps $x=\Phi(\tilde x)$ between gDH systems.  This is
perhaps the most important result of Part~I of this paper.  Each of these
can be viewed as a map of projective planes, i.e., as a map
$\Phi\colon\mathbb{P}^2_{\tilde x}\to\mathbb{P}^2_x$.  If any such map is
applied to a HQDS $\dot {\tilde x}={\tilde Q}(\tilde x)\eqdef\allowbreak
\tilde x*\tilde x$ of the gDH form that satisfies certain conditions on its
parameters, another HQDS $\dot{x}=Q(x)\eqdef x*x$ of the gDH form will
result.  That~is, $Q=\Phi_*\tilde Q$.  This is the content of
Theorem~\ref{thm:gDHsystems}, accompanied by Tables \ref{tab:sigmas}
and~\ref{tab:restrictions}, and Figure~\ref{fig:only}.

The existence of these rational morphisms follows from the
PE\nobreakdash-based integration scheme for proper gDH systems.  At~base,
each comes from a lifting of a PE (on a Riemann sphere $\mathbb{P}^1_t$) to
another (on a Riemann sphere $\mathbb{P}^1_{\tilde t}$).  The lifting is
along a rational map $t=R(\tilde t)$.  Equivalently, each comes from the
lifting of a GHE to a GHE, or a gSE to another gSE\null.  In~fact, each is
associated to a \emph{hypergeometric transformation} (a~transformation of
a~${}_2F_1(t)$ to a~${}_2F_1(\tilde t)$), since such transformations come
from such liftings of GHE's.  The morphisms are denoted by ${\bf2},{\bf3}$,
etc., in reference to the associated ${}_2F_1$ transformations (quadratic,
cubic, etc.), most of which are classical, having been worked~out by
Goursat~\cite{Goursat1881}.

Several of the rational morphisms were found by Harnad and
McKay~\cite{Harnad2000}, who applied them to certain DH systems with no
free parameters, the solutions $x_i=x_i(\tau)$ of which are quasi-modular
forms for subgroups of ${\it PSL}(2,{\mathbb{Z}})$.  It is now clear that
they extend to more general DH systems, and from DH to gDH systems.  They
can even be applied to some non-gDH HQDS's.  For example, the complex
morphism denoted by~${\bf 3}_{\bf c}$ yields a solution-preserving map from
any 3\nobreakdash-dimensional HQDS with the symmetry ${\tilde x}_1
\to\nobreak {\tilde x}_2 \to\allowbreak {\tilde x}_3 \to\nobreak {\tilde
  x}_1$, such as the Leonard--May model of cyclic competition among three
species, to another 3\nobreakdash-dimensional HQDS\null.

\smallskip
(4)\quad{}There is a complete classification of proper non\nobreakdash-DH gDH systems
with the Painlev\'e property~(PP), given in
Theorem~\ref{thm:PPclassification} and
Table~\ref{tab:PPclassification}\null.  It is a corollary of the
classification of non-Schwarzian gSE's with the~PP, begun by
Garnier~\cite{Garnier12} and completed by
Carton-LeBrun~\cite{CartonLeBrun69b}.  The new classification supplements
the known result that ${\rm DH}(\alpha_1,\alpha_2,\alpha_3)$ has the PP if
and only if $(\alpha_1,\nobreak\alpha_2,\nobreak\alpha_3)= \allowbreak
(\frac1{N_1},\nobreak\frac1{N_2},\nobreak\frac1{N_3})$, where each~$N_i$ is
a nonzero integer or~$\infty$.  Interestingly, many of the proper non\nobreakdash-DH
gDH systems with the PP are related by the rational solution-preserving
maps $x=\Phi(\tilde x)$ of~\S\,\ref{sec:transformations}.  The maps are
indicated in the final column of Table~\ref{tab:PPclassification}.

For most proper non\nobreakdash-DH gDH systems with the~PP, generic
solutions $x=x(\tau)$ are doubly periodic, i.e.\ elliptic.  (See
Theorem~\ref{thm:PPintegration} and Table~\ref{tab:PPintegration}.)  Simple
periodicity also occurs, and there are proper gDH systems with the~PP for
which the poles of $x=x(\tau)$ form a polynomially or exponentially
stretched lattice, rather than a regular one.  All this is of~interest
because an irregular pattern of singularities in the complex
$\tau$\nobreakdash-plane would tend to indicate
non-integrability~\cite{Bessis86}.  But no proper non\nobreakdash-DH gDH
system with the~PP has any solution with a natural boundary, unlike the
proper DH systems with the~PP.

As examples, several proper non\nobreakdash-DH gDH systems with the PP,
which are `basic' in the sense that they are not images of other such
systems under rational morphisms, are integrated in~terms of elementary and
elliptic functions.  Papperitz-based integration suffices, and is
facilitated in most cases by the PE having trivial monodromy.  (See
Examples \ref{ex:1}--\ref{ex:4}; Example~\ref{ex:2} includes a complete
integration of the $N=4$ case of Chazy\nobreakdash-XI [without recessive
  terms].)  For each of these gDH systems a pair $I_1,I_2$ of first
integrals is derived, which are typically rational in~$x_1,x_2,x_3$.  Some
final remarks are made on gDH systems that may lack the~PP but are
nonetheless integrable.  This includes certain improper gDH systems.

\smallskip
Part~II of this paper will explore the scalar ODE's satisfied by (linear
combinations~of) components $x_i=x_i(\tau)$ of the solutions of gDH
systems.  It will be shown that the abovementioned DH representations of
the solutions of the Chazy\nobreakdash-III equation~\cite{Chakravarty2010}
are related by the rational morphisms of~\S\,\ref{sec:transformations}, and
that they can be generalized to DH representations of the solutions of
Chazy-XII, and to gDH representations of the solutions of
Chazy\nobreakdash-I, II, VII, and~XI\null.  Solutions of DH systems computed
with the aid of the theory of modular forms will also be discussed.

\section{Hypergeometric Integration of gDH Systems}
\label{sec:integration}

In this section we examine how the generalized Darboux--Halphen system
of~(\ref{eq:gDH}), ${\rm gDH}(a_1,a_2,a_3;b_1,b_2,b_3;c)$ can be
parametrically integrated with the aid of the Gauss hypergeometric
function~${}_2F_1(t)$.  We do this `backwards': we start with the Papperitz
equation, the most general second-order linear Fuchsian ODE with three
singular points on the Riemann sphere
$\mathbb{P}^1=\mathbb{C}\cup\{\infty\}$.  It generalizes the Gauss
hypergeometric equation~(GHE)\null.  In Theorem~\ref{thm:integration} we
derive from any nonzero solution of the Papperitz equation a HQDS satisfied
by a triple $(x_1,x_2,x_3)\eqdef x$ of logarithmic derivatives, viewed as a
function of~$\tau$; both $x$ and~$\tau$ are constructed as local
functions of $t\in\mathbb{P}^1$.  The HQDS turns~out to be a (proper) gDH
system.  This approach is essentially that of Ohyama~\cite{Ohyama96}, but
basing our treatment on the Papperitz equation rather than on the GHE
allows us to treat non\nobreakdash-DH gDH systems, quite elegantly.  We show in
Theorem~\ref{thm:integration0} how, \emph{a~fortiori}, the Papperitz
equation can be used to integrate any proper DH system, and also explain
why any DH system has an extra Lie point symmetry: it is covariant under
maps $\tau\mapsto(A\tau+\nobreak B)/\allowbreak(C\tau+\nobreak D)$ that are
not affine.

Theorem~\ref{thm:gDHbasecoin} states that generic (`noncoincident')
solutions $x=x(\tau)$ of any proper gDH system are bijective with the
solutions $t=t(\tau)$ of an associated nonlinear third-order differential
equation, the so-called gSE\null.  (If the singular values
on~$\mathbb{P}^1_t$ are chosen to be $0,1,\infty$, the bijection will be
$t=-(x_2-\nobreak x_3)/\allowbreak(x_1-\nobreak x_2)$.)  For any proper DH
system the gSE reduces to the Schwarzian equation~(\ref{eq:SE}), which is
well known in the theory of conformal mapping~\cite{Nehari52}.  But the
extension to non\nobreakdash-DH gDH systems, and in~fact the appearance of
the gSE in this context, are new.

\subsection{Papperitz equations}
\label{subsec:21}

A Papperitz equation (PE) on~$\mathbb{P}^1_t$ is of the form
$\mathcal{L}f=0$ with $\mathcal{L}\defeq D_t^2 + P_1(t)D_t + P_2(t)$, and
is determined by its Riemann P\nobreakdash-symbol
\begin{subequations}
\begin{gather}
\label{eq:PPsymbol}
\left\{  
\begin{array}{ccc}
  t_1 & t_2 & t_3 \\
  \hline
  \mu_1 & \mu_2 & \mu_3 \\
  \mu'_1 & \mu'_2 & \mu'_3
\end{array}
\right\}, \\ \intertext{which tabulates the unordered pair of
  characteristic exponents at each singular point.  Here, the singular
  points are the distinct points $t_1,t_2,t_3\in{\mathbb{P}}^1_t$.  The
  respective exponents $\mu_i,\mu_i'\in\mathbb{C}$ must satisfy Fuchs's
  relation $\sum_{i=1}^3(\mu_i+\nobreak\mu_i')=1$ for the equation to be
  Fuchsian, i.e.\ for each singular point to be regular, but otherwise they
  are unconstrained.  If no~$t_i$ is~$\infty$, the PE is} \left\{ D_t^2 +
\left( \sum_{i=1}^3 \frac{1-\mu_i-\mu_i'}{t-t_i} \right)D_t +\left[
  \sum_{i=1}^3\frac{\mu_i\mu_i'\,(t_i-t_j)(t_i-t_k)}{(t-t_i)^2(t-t_j)(t-t_k)}
  \right] \right\}f=0,
\label{eq:PE}
\end{gather}
\end{subequations}
the subscripts $j,k$ in the second summand being the two elements of
$\{1,2,3\}$ other than~$i$, in either order.  If one of $t_1,t_2,t_3$
is~$\infty$ then an obvious limit of~(\ref{eq:PE}) is taken.  The exponent
\emph{differences}
$(\alpha_1,\alpha_2,\alpha_3)=\allowbreak(\mu'_1-\nobreak\mu_1,\nobreak\mu'_2-\nobreak\mu_2,\nobreak\mu'_3-\nobreak\mu_3)$,
defined only up~to sign, will play an important role below.

Three special cases (normal forms) of the PE are important in applications
and calculations.  In each, $(t_1,t_2,t_3)=(0,1,\infty)$.  In the first,
with $\mu_1=\mu_2=0$, the P\nobreakdash-symbol~(\ref{eq:PPsymbol})
reduces to
\begin{subequations}
\begin{gather}
\label{eq:GHPsymbol}
\left\{  
\begin{array}{ccc}
  0 & 1 & \infty \\
  \hline
  0 & 0 & (1-\alpha_1-\alpha_2-\alpha_3)/2 \\
  \alpha_1 & \alpha_2 & (1-\alpha_1-\alpha_2+\alpha_3)/2
\end{array}
\right\}, \\
\intertext{and the equation~(\ref{eq:PE}) to the GHE}
\left[
D_t^2 + \left(\frac{1-\alpha_1}{t}+\frac{1-\alpha_2}{t-1}\right)D_t +\frac{(1-\alpha_1-\alpha_2)^2-\alpha_3^2}{4\,t(t-1)}
\right]f=0.
\label{eq:GHE}
\end{gather}
\end{subequations}
A local solution~$f$ of the GHE at $t=0$ associated to the exponent~$0$
is the Gauss hypergeometric function
\begin{equation}
\label{eq:hypseries}
  {}_2F_1(a,b;\,c;\,t) = \sum_{n=0}^\infty \frac{(a)_n(b)_n}{(c)_n}\,\frac{t^n}{n!},
\end{equation}
in which $(a)_n\defeq (a)(a+1)\cdots(a+n-1)$ and the hypergeometric
parameters $a,b;c$ come from the exponent differences
$(\alpha_1,\alpha_2,\alpha_3)$ by
\begin{equation}
\label{eq:abc}
(a,b;\,c)=\left(\tfrac12(1-\alpha_1-\alpha_2-\alpha_3),\tfrac12(1-\alpha_1-\alpha_2+\alpha_3);\,1-\alpha_1\right).
\end{equation}
Generically ${}_2F_1(a,b;c;t)$ is well-defined: it is defined if $c$~is not
a nonpositive integer, i.e., if the exponent $\alpha_1$ of the GHE at~$t=0$
is not a positive integer.  Also generically, a second solution~$f$ of the
GHE~(\ref{eq:GHE}) at~$t=0$ is given by
\begin{equation}
t^{1-c}\, {}_2F_1(a-c+1;\,b-c+1;\,2-c;\,t).
\end{equation}
It comes from the exponent $\alpha_1=1-c$ at~$t=0$.  The local function
${}_2F_1$ can also be used to construct a pair of local solutions
to~(\ref{eq:GHE}) at $t=1$ and at~$t=\infty$.

In the second special case of the PE, with $\mu_1+\mu_1'=\mu_2+\mu_2'=1$,
the P\nobreakdash-symbol~(\ref{eq:PPsymbol}) reduces to
\begin{subequations}
\begin{gather}
\label{eq:SAHPsymbol}
\left\{  
\begin{array}{ccc}
  0 & 1 & \infty \\
  \hline
  (1-\alpha_1)/2 & (1-\alpha_2)/2 & (-1-\alpha_3)/2 \\
  (1+\alpha_1)/2 & (1+\alpha_2)/2 & (-1+\alpha_3)/2
\end{array}
\right\}, \\
\intertext{and the equation~(\ref{eq:PE}) to}
\left[
D_t^2 + \left(
\frac{1-\alpha_1^2}{4\,t^2} + \frac{1-\alpha_2^2}{4\,(t-1)^2}
-\frac{1-\alpha_1^2-\alpha_2^2+\alpha_3^2}{4\,t(t-1)}
\right)
\right]f=0.
\label{eq:SAHE}
\end{gather}
\end{subequations}
This alternative normal form, self-adjoint in~that it has a null
coefficient function $P_1(t)$, is used, e.g., in conformal mapping.  In the
third special case of the PE, with $\mu_1\mu_1'=\mu_2\mu_2'$, the
P\nobreakdash-symbol~(\ref{eq:PPsymbol}) reduces to
\begin{subequations}
\begin{gather}
\label{eq:thirdHPsymbol}
\left\{  
\begin{array}{ccc}
  0 & 1 & \infty \\
  \hline
  \left[(1-\alpha_1-\alpha_3)^2-\alpha_2^2\right]/4(1-\alpha_3) & \left[(1-\alpha_2-\alpha_3)^2-\alpha_1^2\right]/4(1-\alpha_3) & 0 \\
  \left[(1+\alpha_1-\alpha_3)^2-\alpha_2^2\right]/4(1-\alpha_3) & \left[(1+\alpha_2-\alpha_3)^2-\alpha_1^2\right]/4(1-\alpha_3)  & \alpha_3
\end{array}
\right\}, \\
\intertext{and the equation~(\ref{eq:PE}) to}
\left[
D_t^2 + 
\left(
\frac{1-\alpha_1^2+\alpha_2^2 - \alpha_3^2}{2(1-\alpha_3)\,t}
+\frac{1+\alpha_1^2-\alpha_2^2 - \alpha_3^2}{2(1-\alpha_3)\,(t-1)}
\right)D_t
+\frac{\lambda}{16(1-\alpha_3)^2\,t^2(1-t)^2}
\right]f=0,
\label{eq:thirdHE}
\end{gather}
\end{subequations}
with $\lambda=\left[(1-\alpha_3)^2-(\alpha_1+\alpha_2)^2\right]
\left[(1-\alpha_3)^2-(\alpha_1-\alpha_2)^2\right]$.  This third normal form
is distinguished by its coefficient function~$P_2(t)$ having only double
poles.

By manipulation of P\nobreakdash-symbols, the solutions of the general
PE~(\ref{eq:PE}) can be expressed in~terms of those of the
GHE~(\ref{eq:GHE}), such as~${}_2F_1$.  Let $f$~be a solution
of~(\ref{eq:PE}) and let $f_i=\Delta_i f$, $i=1,2,3$, where 
\begin{align}
\label{eq:fidef}
\Delta_i(t) & \defeq  
\left[-\,\frac{(t_i-t)(t_j-t_k)}{(t-t_j)(t_k-t_i)}\right]^{\mu_j}
\left[-\,\frac{(t-t_i)(t_j-t_k)}{(t_k-t)(t_i-t_j)}\right]^{\mu_k} \\
&\hphantom{:}\propto (t-t_i)^{\mu_j+\mu_k}(t-t_j)^{-\mu_j}(t-t_k)^{-\mu_k},\nonumber
\end{align}
in which $(i,j,k)$ is the cyclic permutation of $(1,2,3)$ determined
by~$i$.  (Since the right side of~(\ref{eq:fidef}) is expressed in~terms of
cross-ratios, $f_i$~will be well-defined even if one of $t_1,t_2,t_3$ is
taken to~$\infty$.)  For each~$i$, $f_i$~will be a solution of the PE that
has P\nobreakdash-symbol
\begin{equation}
\label{eq:enroute}
\left\{  
\begin{array}{ccc}
  t_j & t_k & t_i \\
  \hline
  0 & 0 & \mu_i+\mu_j+\mu_k \\
  \mu'_j-\mu_j & \mu'_k-\mu_k & \mu'_i+\mu_j+\mu_k
\end{array}
\right\}.
\end{equation}
To obtain a solution of a GHE from~$f_i$, one must apply a M\"obius
transformation: $t\mapsto S_{ijk}(t)\defeq
(t-t_j)(t_k-t_i)/(t-t_i)(t_k-t_j)$, which takes $t_j,t_k,t_i$
to~$0,1,\infty$.  If one defines~$F_{ijk}$ by
$F_{ijk}(t)=f_i(S_{ijk}^{-1}(t))$, then $F_{ijk}(t)$ will be the solution
of the PE with P\nobreakdash-symbol
\begin{equation}
\left\{  
\begin{array}{ccc}
  0 & 1 & \infty \\
  \hline
  0 & 0 & \mu_i+\mu_j+\mu_k\\
  \mu'_j-\mu_j & \mu'_k-\mu_k & \mu'_i+\mu_j+\mu_k
\end{array}
\right\},
\end{equation}
i.e., of a certain GHE (cf.~(\ref{eq:GHPsymbol})).  Reversing this, one
sees that if $F_{ijk}$~is any solution of this GHE, such as the
appropriate~${}_2F_1$, the function $f_i(t) = F_{ijk}(S_{ijk}(t))$ will be
a solution of the PE with P\nobreakdash-symbol~(\ref{eq:enroute}), and
furthermore that
\begin{equation}
  f(t) = \Delta_i^{-1}(t)
\,F_{ijk}(S_{ijk}(t)) 
\end{equation}
will be a solution of the original PE~(\ref{eq:PE}).  By permuting $i,j,k$,
one can generate $3!=6$ local solutions of~(\ref{eq:PE}) that are expressed
in~terms of~${}_2F_1$.

In a similar way, one can express solutions of any PE in~terms of solutions
of either of the remaining two normal forms of the PE,
Eqs.\ (\ref{eq:SAHE}) and~(\ref{eq:thirdHE}).  Each normal form is
associated to a distinct triple $\Delta_i$, $i=1,2,3$, analogous
to~(\ref{eq:fidef}).

\subsection{From PE to gDH}
\label{subsec:fromPEtogDH}

The following theorem facilitates the hypergeometric integration of gDH and
DH systems.  It also introduces an alternative parametrization of gDH
systems, based not on $(a_1,\nobreak a_2,\nobreak a_3;\allowbreak
b_1,\nobreak b_2,\nobreak b_3;c)$ but on a birationally equivalent
parameter vector $(\nu_1,\nobreak \nu_1'; \allowbreak \nu_2,\nobreak \nu_2'; \allowbreak \nu_3,\nobreak \nu_3';\allowbreak \bar n; c )$.

\begin{theorem}
  Let $f=f(t)$ be a nonzero analytic local solution of a Papperitz equation
  of the form\/ {\rm(\ref{eq:PE})} with exponents\/ $(\mu_1,\nobreak
  \mu_1'; \allowbreak \mu_2,\nobreak \mu_2'; \allowbreak \mu_3,\nobreak
  \mu_3')$ satisfying Fuchs's relation\/
  $\sum_{i=1}^3(\mu_i+\nobreak\mu_i')=1$, and with\/ {\rm(}distinct\/{\rm)}
  singular points\/ $t_1,t_2,t_3\in\mathbb{P}^1$.  Define functions\/
  $\tau$ and\/ $x_1,x_2,x_3$ of the local parameter\/~$t$ as follows.
  \begin{enumerate}
  \item For some\/ $\bar n\in\mathbb{C}\setminus\{0\}$ and some
    `offset vector'\/
    $\kappa=(\kappa_1,\kappa_2,\kappa_3)\in{\mathbb{C}}^3$ satisfying\/
    $\kappa_1+\kappa_2+\kappa_3=1$, choose\/ $\tau=\tau(t)$ to satisfy
    \begin{displaymath}
      \frac{{\rm d}\tau}{{\rm d}t} =  K^{-2}(t) f^{-1/\bar n}(t), 
    \end{displaymath}
    where
    \begin{sizeequation}{\small}
      \label{eq:Kdef}
      K^2(t)\defeq 
      \left[ \frac{(t-t_1)^2(t_2-t_3)}{(t_1-t_2)(t_3-t_1)} \right]^{\kappa_1}
      \left[ \frac{(t-t_2)^2(t_3-t_1)}{(t_2-t_3)(t_1-t_2)} \right]^{\kappa_2}
      \left[ \frac{(t-t_3)^2(t_1-t_2)}{(t_3-t_1)(t_2-t_3)} \right]^{\kappa_3}.
    \end{sizeequation}
  \item For some\/ $c\in\mathbb{C}\setminus\{0\}$, define each\/ $x_i$ as a
    logarithmic derivative:
    \begin{align*}
      x_i(t) &= c^{-1} {\bar{n}^{-1}}\,\frac{\rm d}{{\rm d}\tau}\log f_i \\
      &= c^{-1} \,\frac{\rm d}{{\rm d}\tau}\log\left[
        K^{-2}\Delta_i^{1/\bar n} ({\rm d}\tau/{\rm d}t)^{-1}
        \right],
    \end{align*}
    where\/ $f_i=\Delta_if$ and
    \begin{equation}
      \label{eq:Deltaidef}
      \Delta_i(t) \defeq   
      \left[-\,\frac{(t_i-t)(t_j-t_k)}{(t-t_j)(t_k-t_i)}\right]^{\mu_j}
      \left[-\,\frac{(t-t_i)(t_j-t_k)}{(t_k-t)(t_i-t_j)}\right]^{\mu_k},
    \end{equation}
  \end{enumerate}
  in which\/ $(i,j,k)$ is the cyclic permutation of\/ $(1,2,3)$ determined
  by\/ $i$.  Then\/ $x=(x_1,x_2,x_3)$, viewed as a function of\/~$\tau$,
  will satisfy the generalized Darboux--Halphen system\/
  {\rm(\ref{eq:gDH})}, ${\rm gDH}(a_1,\nobreak a_2,\nobreak a_3;\allowbreak
  b_1,\nobreak b_2,\nobreak b_3;c)$, with parameters\/ $(a_1,\nobreak
  a_2,\nobreak a_3;\allowbreak b_1,\nobreak b_2,\nobreak b_3)$ computed
  from\/ $(\mu_1,\nobreak \mu_1'; \allowbreak \mu_2,\nobreak \mu_2';
  \allowbreak \mu_3,\nobreak \mu_3';\bar n)$ and\/~$c$, i.e., from\/
  $(\nu_1,\nobreak \nu_1'; \allowbreak \nu_2,\nobreak \nu_2'; \allowbreak
  \nu_3,\nobreak \nu_3';\bar n)$ and\/~$c$, by
  \begin{equation}
    \label{eq:zetasntoab}
    \left\{
  \begin{aligned}
    a_i &=
    \frac{\nu_i}{\nu_1+\nu_2+\nu_3+{\bar{n}}}\,
    c,\\
    b_i &=
    \frac{-{\bar{n}}\nu'_i - (\bar n-1)(\nu_j+\nu_k)  - {\bar{n}}(\bar n - 1)}
    {\nu_1+\nu_2+\nu_3+{\bar{n}}}\,c,
  \end{aligned}
  \right.
  \end{equation}
  where
  \begin{equation}
    \nu_i\defeq \mu_i+(2\kappa_i-1){\bar{n}}, \qquad
    \nu_i'\defeq \mu_i'+2(\kappa_i-1){\bar{n}}
    \label{eq:zetas}
  \end{equation}
  are \emph{offset} exponents, satisfying\/ $\sum_{i=1}^3 (\nu_i+\nobreak
  \nu_i')=\allowbreak 1-\nobreak 2\bar n$.  It is assumed in this that the
  denominator\/ $\rho^{-1}\defeq \nu_1+\nu_2+\nu_3+\bar n$
  in\/~{\rm(\ref{eq:zetasntoab}\rm)} is nonzero.  Equivalently,
  $(a_1,\nobreak a_2,\nobreak a_3;\allowbreak b_1,\nobreak b_2,\nobreak
  b_3)$ are determined implicitly by the inverse formulas
  \begin{equation}
    \left\{
    \begin{aligned}
      \bar n &= \frac{2c-b_1-b_2-b_3}c,\\
      \nu_i &= \bar n\left(\frac{a_i}{c-a_1-a_2-a_3}\right),\\
      \alpha_i &\defeq\nu_i'-\nu_i = \frac{-c-a_i+b_j+b_k}{c-a_1-a_2-a_3},
    \end{aligned}
    \right.
    \label{eq:abtozetasn}
  \end{equation}
  which express\/ $(\nu_1,\nobreak \nu_1'; \allowbreak \nu_2,\nobreak \nu_2'; \allowbreak \nu_3,\nobreak \nu_3';\bar n)$ in terms of\/
  $(a_1,\nobreak a_2,\nobreak a_3;\allowbreak b_1,\nobreak b_2,\nobreak
  b_3)$ and\/~$c$.
  \label{thm:integration}
\end{theorem}

\begin{remark*}
  The formulas (\ref{eq:Kdef}),(\ref{eq:Deltaidef}) for $K^2$
  and~$\Delta_i$ (the latter based on cross-ratios) are perhaps overly
  elaborate: they were crafted to cover the case when one of $t_1,t_2,t_3$
  is taken to~$\infty$.  If each~$t_i$ is finite then
  \begin{equation}
    \label{eq:2ndremark}
    \begin{aligned}
      K(t) &\propto (t-t_1)^{\kappa_1}(t-t_2)^{\kappa_2} (t-t_3)^{\kappa_3}, \\
      \Delta_i(t) &\propto (t-t_i)^{\mu_j+\mu_k}(t-t_j)^{-\mu_j}(t-t_k)^{-\mu_k},
    \end{aligned}
  \end{equation}
  the validity of the theorem being unaffected by the choice of
  proportionality constants.  For instance, if $(t_1,t_2,t_3)=(0,1,\infty)$
  then
  \begin{equation}
    \label{eq:Deltas}
    (\Delta_1,\Delta_2,\Delta_3)=\left(
    (-t)^{\mu_2+\mu_3}(t-1)^{-\mu_2},\,
    (-t)^{-\mu_1}(t-1)^{\mu_1+\mu_3},\,
    (-t)^{-\mu_1}(t-1)^{-\mu_2}
    \right),
  \end{equation}
  so that in this case, if $\mu_1+\mu_2+\mu_3=\nu_1+\nu_2+\nu_3+\bar
  n\eqdef \rho^{-1}$ is nonzero, one can simply write
  \begin{align}
    \left({\Delta_1}^\rho,{\Delta_2}^\rho,{\Delta_3}^\rho\right)&={\Delta_1}^\rho\times
    \bigl(1,\,(1-t)/t,\,-1/t\bigr), \\
    \left({f_1}^\rho,{f_2}^\rho,{f_3}^\rho\right)&={f_1}^\rho\times
    \bigl(1,\,(1-t)/t,\,-1/t\bigr).
  \end{align}
  Also in this case, if
  $(\kappa_1,\kappa_2,\kappa_3)=(0,0,1)$ then $K^2$~will simply equal
  unity.
\end{remark*}
\begin{remark*}
  The expressions for $\Delta_i$, $i=1,2,3$, originate in the first normal
  form of the PE (the GHE); see~(\ref{eq:fidef}).  As $f_i=\Delta_i f$,
  each of $f_1,f_2,f_3$ satisfies a GHE\null.  Thus the gDH
  system~(\ref{eq:gDH}) for the vector of logarithmic derivatives
  $(x_1,x_2,x_3)\eqdef x$ is associated specifically to the GHE\null.
  Alternative gDH systems (i.e., non-gDH HQDS's) associated to the other
  two normal forms of the PE can readily be derived; but up~to linear
  equivalence, i.e.\ up~to the action of ${\it GL}(3,\mathbb{C})$ on
  $x\in{\mathbb{C}}^3$, they will be no~different.
\end{remark*}

\begin{proof}
  For convenience, define the logarithmic derivatives
  \begin{equation}
    \begin{aligned}
      \bar\kappa(t)&= K_t/K = \kappa_1(t-t_1)^{-1} + \kappa_2(t-t_2)^{-1} + \kappa_3(t-t_3)^{-1},\\
      \delta_i(t) &= (\Delta_i)_t/\Delta_i=(\mu_j+\mu_k)(t-t_i)^{-1}- \mu_j(t-t_j)^{-1}- \mu_k(t-t_k)^{-1}.
    \end{aligned}
  \end{equation}
Also introduce an \emph{ad hoc} notation for certain homogeneous quadratic
polynomials, 
\begin{equation}
  \mathopen{\pmb{\lbrack}}A_2,A_1,A_0\mathclose{\pmb{\rbrack}} \defeq  \sum_{m=0}^2 A_m 
\left(K^2f^{1/\bar n}\right)^{2-m}\left(\frac{f_\tau}{f}\right)^m.
\end{equation}
By definition,
\begin{equation}
  x_i=c^{-1}{\bar{n}^{-1}}\left(\frac{f_\tau}f + \delta_i\,K^2f^{1/\bar n}\right),
\end{equation}
so that
\begin{subequations}
\label{eq:adhoc2}
  \begin{align}
    (x_i-x_j)(x_k-x_i) &=c^{-2}{\bar{n}^{-2}}\,\mathopen{\pmb{\lbrack}}0,\,0,\,(\delta_i-\delta_j)(\delta_k-\delta_i)\mathclose{\pmb{\rbrack}},\\
    x_px_q &= c^{-2}{\bar{n}^{-2}}\,\mathopen{\pmb{\lbrack}}1,\,\delta_p+\delta_q,\,\delta_p\delta_q\mathclose{\pmb{\rbrack}}.
  \end{align}
\end{subequations}
The PE~(\ref{eq:PE}) is of the form $\left[D_t^2+P_1(t)D_t+P_2(t)\right]f=0$,
and changing the variable of differentiation from $t$ to~$\tau$ converts it
to 
\begin{equation}
\label{eq:adhoc1}
  \frac{f_{\tau\tau}}f +\, \mathopen{\pmb{\lbrack}}-\bar n^{-1}\!,\,P_1-2\,\bar\kappa,\,P_2\mathclose{\pmb{\rbrack}}=0.
\end{equation}
A similar manipulation of the formula for $\dot x_i = {\rm d}x_i/{\rm
  d}\tau$ yields
\begin{equation}
  \dot x_i = c^{-1}{\bar{n}^{-1}}\left\{
\frac{f_{\tau\tau}}f + \mathopen{\pmb{\big\lbrack}}-1,\,{\bar{n}}^{-1}\delta_i,\,(\delta_i)_t + 2\,\bar\kappa\,\delta_i\mathclose{\pmb{\big\rbrack}}
\right\},
\end{equation}
which combined with~(\ref{eq:adhoc1}) gives
\begin{equation}
\label{eq:adhoc3}
  \dot x_i = c^{-1}{\bar{n}^{-1}}\,
\mathopen{\pmb{\big\lbrack}}-1+\bar n^{-1}\!,\,- P_1 + \bar n^{-1}\delta_i + 2\,\bar\kappa,\,-P_2 + (\delta_i)_t + 2\,\bar\kappa\,\delta_i\mathclose{\pmb{\big\rbrack}}.
\end{equation}
The representations (\ref{eq:adhoc2}) and~(\ref{eq:adhoc3}) reduce the
question of the linear dependence of the local functions
$(x_i-x_j)(x_k-x_i)$, $\{x_px_q\}_{p,q=1}^3$, and~$\dot x_i$ to linear
algebra over~$\mathbb{C}(t)$, as each function is effectively a
3\nobreakdash-vector
$\mathopen{\pmb{\lbrack}}A_2(t),A_1(t),A_0(t)\mathclose{\pmb{\rbrack}}$ of
rational functions of~$t$.  Each of the three equations in the gDH
system~(\ref{eq:gDH}) can thus be viewed as a relation of dependence among
3\nobreakdash-vectors of rational functions of~$t$, with $a_1,\nobreak
a_2,\nobreak a_3;\allowbreak b_1,\nobreak b_2,\nobreak b_3;\nobreak c$
appearing as coefficients.  If $a_1,\nobreak a_2,\nobreak a_3;\allowbreak
b_1,\nobreak b_2,\nobreak b_3$ are computed from the formulas in the
theorem, the validity of each relation can be verified by hand, though the
use of a computer algebra system is recommended.
\end{proof}

Theorem~\ref{thm:integration} reveals that any `proper' gDH system (to be
defined shortly) can be integrated parametrically in~terms of solutions of
the GHE; and generically, in~terms of the canonical hypergeometric
function~${}_2F_1$.  The integration proceeds as follows.  Given
$(a_1,\nobreak a_2,\nobreak a_3;\allowbreak b_1,\nobreak b_2,\nobreak
b_3;\nobreak c)$, one first uses the formulas~(\ref{eq:zetasntoab}) to
compute $(\nu_1,\nobreak \nu_1'; \allowbreak \nu_2,\nobreak \nu_2';
\allowbreak \nu_3,\nobreak \nu_3';\bar n)$, and thus $(\mu_1,\nobreak
\mu_1'; \allowbreak \mu_2,\nobreak \mu_2'; \allowbreak \mu_3,\nobreak
\mu_3';\bar n)$, the exponent offset vector~$\kappa$ (satisfying
$\kappa_1+\nobreak\kappa_2+\allowbreak\kappa_3=\nobreak1$) being chosen
arbitrarily.  Let $f$ be a nonzero local solution of a PE~(\ref{eq:PE})
with exponents $(\mu_1,\nobreak \mu_1'; \allowbreak \mu_2,\nobreak \mu_2';
\allowbreak \mu_3,\nobreak \mu_3')$, the choice of singular points
$(t_1,t_2,t_3)$ of the PE also being arbitrary.  The PE having a
2\nobreakdash-dimensional space of solutions, there are two free parameters
in the choice of~$f$, as one can write $f=K_1f^{(1)}+\nobreak K_2f^{(2)}$;
and $f$~can generically be expressed in~terms of~${}_2F_1$.  The gDH
variables~$\tau$ (independent) and $x_1,x_2,x_3$ (dependent) will be
parametrized by~$t$ as
\begin{equation}
\label{eq:integration}
\left\{
\begin{aligned}
\tau(t) 
&= \tau_0 + \int^t \left(\frac{{\rm d}\tau}{{\rm d}t}\right)\,dt\\
&=\tau_0 + C \int^t
      (t-t_1)^{-2\kappa_1}(t-t_2)^{-2\kappa_2}(t-t_3)^{-2\kappa_3}\,
      f^{-1/\bar n}(t)\,dt,   \\
x_i(t)&= c^{-1}{\bar{n}}^{-1}
    \left(\frac{{\rm d}\tau}{{\rm d}t}
    \right)^{-1}
    \left(\frac{(f_i)_t}{f_i}
    \right),
\end{aligned}
\right.
\end{equation}
where $f_i\defeq \Delta_if$, $i=1,2,3$, are as in the theorem (they depend
on $\mu_1,\mu_2,\mu_3$), and $C$~depends on~$t_1,t_2,t_3$.  By transposing
the ordered pairs $(\nu_i,\nu_i')$ and hence the ordered pairs
$(\mu_i,\mu_i')$, one can produce $2^3=8$ formally distinct integrations.

In~(\ref{eq:integration}) the formula for $\tau=\tau(t)$ contains an
additional parameter, $\tau_0\in\mathbb{C}$.  Thus there are three free
parameters in the local solution $t\mapsto(\tau;x_1,x_2,x_3)$.  Hence any
integration of this type will generate a 3\nobreakdash-parameter family of
solutions of the gDH system.  PE\nobreakdash-based integration is therefore
\emph{complete}: it constructs the general solution.  (Special solutions
with less than the full complement of three free parameters will be
discussed below.)  The freedom in the choice of~$\tau_0$ and the choice of
scale of the PE solution~$f$ can be viewed as the cause of the covariance
of the gDH system under affine transformations, i.e., under $\tau\mapsto
A\tau+\nobreak B$.

There is an obvious caveat.  Theorem~\ref{thm:integration} assumes that
$\bar n\neq\infty$ and $\bar n\neq0$, i.e., that $c\neq0$ and
$2c-b_1-b_2-b_3\neq0$; also, that the calculation of the exponent
parameters $\nu_i,\nu_i'$ does not involve a division by zero, i.e.,
$c-\nobreak a_1-\nobreak a_2-\nobreak a_3\neq0$.

\begin{definition}
\label{def:proper}
  A gDH system (\ref{eq:gDH}) is \emph{proper} if its parameter vector
  $(a_1,\nobreak a_2,\nobreak a_3; \allowbreak b_1,\nobreak b_2,\nobreak
  b_3;c)\in{\mathbb{C}}^3\times{\mathbb{C}}^3\times{\mathbb{C}}$ satisfies
  (i)~$c\neq 0$, (ii) $c-\nobreak a_1-\nobreak a_2-\nobreak a_3\neq0$, and
  (iii)~$2c-\nobreak b_1-\nobreak b_2-\nobreak b_3\neq 0$.  Thus a proper
  DH system (see the Introduction) is a proper gDH system that satisfies
  $b_1=b_2=b_3=c/2$.

  Equivalently, a gDH system is proper if its alternative parameter vector
  $(\nu_1,\nobreak \nu_1'; \allowbreak \nu_2,\nobreak \nu_2'; \allowbreak
  \nu_3,\nobreak \nu_3';\allowbreak \bar n; c )\in \allowbreak
     {\mathbb{C}}^3\times{\mathbb{C}}^3\times{\mathbb{C}\times{\mathbb{C}}}$,
     constrained to satisfy $\sum_{i=1}^3(\nu_1+\nu_i')=\allowbreak
     1-\nobreak 2\bar n$ and birationally related to $(a_1,\nobreak
     a_2,\nobreak a_3; \allowbreak b_1,\nobreak b_2,\nobreak b_3;c)$ by
     (\ref{eq:zetasntoab}),(\ref{eq:abtozetasn}), satisfies (i)~$c\neq 0$,
     (ii)~the condition that
  \begin{displaymath}
    \rho^{-1}\defeq \nu_1+\nu_2+\nu_3+\bar n = (1-\alpha_1-\alpha_2-\alpha_3)/2
  \end{displaymath}
  be nonzero, where $\alpha_i\defeq\nu_i'-\nu_i$, $i=1,2,3$, are called the
  angular parameters of the gDH system, and (iii)~$\bar n\neq 0$.
\end{definition}

The just-described integration procedure can be applied to any
\emph{proper} gDH system.  The system parameter $\rho^{-1}$ (or~its
reciprocal~$\rho$) will appear again, and it should be noted that
\begin{equation}
\label{eq:rhodef0}
\rho=\frac{c-a_1-a_2-a_3}{2\,c-b_1-b_2-b_3}
\end{equation}
when expressed in~terms of $(a_1,a_2,a_3;b_1,b_2,b_3;c)$.

There is another caveat: it will shortly become clear that although
PE\nobreakdash-based integration is complete, it is not adapted to
constructing special solutions in which two or more of $x_1,x_2,x_3$
coincide (necessarily, at all~$\tau$).  But such solutions can often be
found by examination.  The solution~(\ref{eq:raysoln}), with $x_1=x_2=x_3$,
has been noted.  It is also easy to check that any gDH
system~(\ref{eq:gDH}) with $a_1=0$ and $2b_1= b_2+\nobreak b_3=c\neq0$,
such as a proper DH system ${\rm DH}(0,\alpha_2,\alpha_3\,|\,c)$, will have
solution
\begin{equation}
\label{eq:raysoln2}
  x(\tau)=(c-b_1-b_2-b_3)^{-1}(\tau-\tau_*)^{-1}(1,1,1) + A(\tau-\tau_*)^{-2}(1,0,0)
\end{equation}
for any $A\in\mathbb{C}$.  This is a 2\nobreakdash-parameter family of
solutions with $x_2=x_3$, the $A=0$ case of which
is~(\ref{eq:raysoln}).\footnote{It generalizes a well-known
  2\nobreakdash-parameter family of solutions of Chazy-III~\cite{Joshi93}.
  There are only two free parameters because for the ray
  solution~(\ref{eq:raysoln}) of the gDH system (parameters constrained as
  stated), the set~$\mathcal{R}$ of Kovalevskaya
  exponents~\cite{Goriely2001,Tsygvintsev2000} is degenerate:
  $\mathcal{R}=\{-1,\nobreak-1,\nobreak-1\}$.}
\begin{proposition}
  In any gDH system\/ {\rm(\ref{eq:gDH})}, the elements\/
  $e=x_1e_1+x_2e_2+x_3e_3$ of the corresponding non-associative algebra\/
  $\mathfrak{A}$ that satisfy\/ $e*e\propto e$ include the following seven:
  $e_0;e_1,\nobreak e_2,\nobreak e_3;\allowbreak e_1',\nobreak
  e_2',\nobreak e_3'$, in which
  \begin{align*}
    e_0 & \defeq e_1+e_2+e_3,\\
    e_i' & \defeq (c-a_i-b_i)e_i + (-c-a_i+b_j+b_k)(e_j+e_k),\qquad i=1,2,3,
  \end{align*}
  where\/ $j,k$ are the elements of\/ $\{1,2,3\}$ other than\/ $i$.  If any
  such\/ $e$ is not nilpotent {\rm(}i.e., $e*e\neq0${\rm)}, then it yields
  an idempotent\/ $p\propto e$ of\/ $\mathfrak{A}$, and a\/
  {\rm1}\nobreakdash-parameter family of ray solutions of\/
  {\rm(\ref{eq:gDH})} along\/~$p$, i.e., $x(\tau)=\allowbreak
  -(\tau-\nobreak\tau_*)^{-1}p$.  If\/ $e$ is nilpotent there is a ray of
  constant solutions instead, i.e., $x(\tau)\equiv K\,e$, $K\in\mathbb{C}$.
\label{prop:defineps}
\end{proposition}
\begin{proof}
  Substitute $e=x_1e_1+x_2e_2+x_3e_3$ into $e*e=\lambda e$, the product~$*$
  coming from (\ref{eq:polarization}) and~(\ref{eq:gDH}).  If $\lambda\neq0$
  then $e$~is idempotent; otherwise nilpotent.
\end{proof}
\begin{remark*}
  The preceding illustrates the fact that any non-associative algebra
  over~$\mathbb{R}$ must have at~least one idempotent or
  nilpotent~\cite{Kaplan79}.  Additionally, as any gDH
  algebra~$\mathfrak{A}$ has dimension $d=3$ over~$\mathbb{C}$, the number
  of its pairwise linearly independent idempotents and nilpotents, if
  finite, is $\le 2^d\nobreak-1=\allowbreak 2^3-\nobreak1=7$.
  (See~\cite[Thm.~3.3]{Walcher91}.)  For generic $(a_1,\nobreak
  a_2,\nobreak a_3;\allowbreak b_1,\nobreak b_2,\nobreak b_3;\nobreak c)$,
  by examination there will be no nilpotents and exactly seven idempotents,
  namely $p_0;p_1,\nobreak p_2,\nobreak p_3;\allowbreak p_1',\nobreak
  p_2',\nobreak p_3'$, defined as above as scaled versions of
  $e_0;e_1,\nobreak e_2,\nobreak e_3;\allowbreak e_1',\nobreak
  e_2',\nobreak e_3'$.  In nongeneric cases the latter may degenerate or
  become nilpotent.
\end{remark*}

Just as $e=e_0\defeq e_1+e_2+e_3$ yields the ray
solution~(\ref{eq:raysoln}), for which $x_1=x_2=x_3$ at all~$\tau$, so do
$e=e_i,e_i'$ yield respectively the ray solutions
\begin{subequations}
\label{eq:newscale}
\begin{sizealign}{\small}
x(\tau)&= -(\tau-\tau_*)^{-1}p_i = -a_i^{-1}(\tau-\tau_*)^{-1}e_i,\\
x(\tau)&=-(\tau-\tau_*)^{-1}p_i' = \left[c(c-a_i-b_j-b_k)+a_i(b_i+b_j+b_k)\right]^{-1}(\tau-\tau_*)^{-1}e_i',
\end{sizealign}
\end{subequations}
for which $x_j=x_k$ at all~$\tau$.  (If the factor multiplying $e=e_i,e_i'$
diverges then $e$~is nilpotent, and $x\equiv Ke$, $K\in\mathbb{C}$, is a
ray of nilpotents.)  The solutions~(\ref{eq:newscale}) can be embedded in
families like~(\ref{eq:raysoln2}), which may or may not have three free
parameters.

In the integration of differential equations, solutions without a full
complement of free parameters are traditionally called `singular'
solutions~\cite{Davis62}.  They lie off the general-solution manifold, and
may not be generated by a generic integration procedure.  The following
more restrictive definition will be useful here.

\begin{definition}
\label{def:singular}
Any meromorphic local solution $x=x(\tau)$ of the gDH system (\ref{eq:gDH})
is \emph{coincident} if two or more components are coincident (necessarily,
at all~$\tau$).
\end{definition}

The following two lemmas will be used below.  The first
(Lemma~\ref{lem:usedlater}) is phrased so as to refer to
PE\nobreakdash-based integration, but its formulas will find broader
application.  The most important are (\ref{eq:zerothclaim})
and~(\ref{eq:secondclaim}), which construct $x=x(\tau)$ from $t=t(\tau)$,
and reconstruct $t=t(\tau)$ from any noncoincident $x=x(\tau)$.  Outside
PE\nobreakdash-based integration, the latter will be used as the
\emph{definition} of $t=t(\tau)$; and the two will be called the
$t(\cdot)\mapsto\allowbreak x(\cdot)$ map and the
$x(\cdot)\mapsto\allowbreak t(\cdot)$ map.

\begin{lemma}
As in Theorem\/~{\rm\ref{thm:integration}}, let\/ $\tau=\tau(t)$ and\/
$x=x(t)$ be computed from a nonzero local Papperitz solution\/ $f=f(t)$,
the latter via\/ $x_i = c^{-1}\bar n^{-1} \dot f_i/f_i$, where the dot
indicates\/ ${\rm d}/{\rm d}\tau$.  Then for\/ $i=1,2,3$, with\/ $j,k$ the
elements other than\/~$i$,
\begin{sizeequation}{\small}
\label{eq:zerothclaim}
\begin{split}
x_i &= c^{-1}\frac{{\rm d}}{{\rm d}\tau}\log
\left[
\frac{\dot t}{(t-t_i)^{-(\nu_j+\nu_k)/\bar n} \, (t-t_j)^{(\nu_j+\bar n)/\bar n} \, (t-t_k)^{(\nu_k+\bar n)/\bar n}}
\right]
\\
  &= c^{-1}\left[
\frac{\ddot t}{\dot t} - \bar n^{-1}\sum_{l=1}^3(\nu_l+\bar n)\frac{\dot t}{t-t_l}
\right]
+c^{-1}\bar n^{-1}(\nu_1+\nu_2+\nu_3+\bar n)\,\frac{\dot t}{t-t_i},
\end{split}
\end{sizeequation}
where\/ $\nu_l\defeq \mu_l + (2\kappa_l-1)\bar n$ as usual, so\/
$\nu_1+\nu_2+\nu_3+\bar n=\mu_1+\mu_2+\mu_3$.  Also,
\begin{equation}
x_i-x_j = c^{-1}{\bar{n}^{-1}}\, (\nu_1+\nu_2+\nu_3+\bar
n)\left[\frac{\dot t}{t-t_i} - \frac{\dot t}{t-t_j}\right]
\label{eq:firstclaim}
\end{equation}
for any\/ $i,j\in\{1,2,3\}$.  Furthermore, if\/ $x_1,x_2,x_3$ do not
coincide, then
\begin{equation}
\label{eq:secondclaim}
  t= -\left[
    \frac
        {t_1t_2(x_1-x_2) + t_2t_3(x_2-x_3) + t_3t_1(x_3-x_1)}
        {t_1(x_2-x_3) + t_2(x_3-x_1) + t_3(x_1-x_2)}
    \right].
\end{equation}
If\/ $\rho^{-1}\defeq \nu_1+\nu_2+\nu_3+\bar n$ is nonzero then\/
${f_1}^\rho+{f_2}^\rho+{f_3}^\rho=0$, and if\/ $x_1,x_2,x_3$ do not
coincide then
\begin{equation}
  \label{eq:fourthclaim}
[x_2-x_3:x_3-x_1:x_1-x_2] = [{f_1}^\rho:{f_2}^{\rho}:{f_3}^\rho],  
\end{equation}
interpreted as an equality between\/ $\mathbb{P}^2$-valued functions.
\label{lem:usedlater}
\end{lemma}

\begin{proof}
  These four formulas are all consequences of the definitions of $K^2$ and
  $\Delta_i$ given in Theorem~\ref{thm:integration}.  By rewriting
  logarithmic derivatives,
  \begin{align}
    x_i&=c^{-1}\bar n^{-1}(\dot f/f + \dot\Delta_i/\Delta_i)\\
    &= c^{-1}\bar n^{-1}(\bar n\,\ddot t/\dot t - \bar n{(K^2)}^{\textbf{.}}/K^2 + \dot\Delta_i/\Delta_i)\nonumber\\
    &= c^{-1}\ddot t/\dot t + c^{-1}\left[-(K^2)_t/K^2 + {\bar
        n}^{-1}(\Delta_i)_t/\Delta_i\right]\,\dot t,\nonumber
  \end{align}
  which leads to~(\ref{eq:zerothclaim}).  Equation~(\ref{eq:firstclaim}) is
  a corollary.  It follows from~(\ref{eq:firstclaim}) that
  \begin{equation}
    \label{eq:solute}
    \frac{x_i-x_j}{x_j-x_k} = -\,\frac{(t_k-t)(t_i-t_j)}{(t-t_i)(t_j-t_k)},
  \end{equation}
  and the formula~(\ref{eq:secondclaim}) is obtained by
  solving~(\ref{eq:solute}) for~$t$.  One can check by hand that
  ${\Delta_1}^\rho + {\Delta_2}^\rho + {\Delta_3}^\rho=0$, which implies
  ${f_1}^\rho+{f_2}^\rho+{f_3}^\rho=0$;
  and that
  \begin{multline}
    \label{eq:assist}
    [{f_1}^\rho:{f_2}^\rho:{f_3}^\rho] = 
    [{\Delta_1}^\rho:{\Delta_2}^\rho:{\Delta_3}^\rho]\\ = [(t_2-t_3)(t-t_1):(t_3-t_1)(t-t_2):(t_1-t_2)(t-t_3)].
  \end{multline}
  Combining (\ref{eq:firstclaim}) with~(\ref{eq:assist})
  yields~(\ref{eq:fourthclaim}).  
\end{proof}

\begin{lemma}
For any noncoincident solution\/ $x=x(\tau)$ of a proper gDH system of the
form\/ {\rm(\ref{eq:gDH})}, if\/ $t=t(\tau)$ is \emph{defined} by the\/
$x(\cdot)\mapsto t(\cdot)$ map\/ {\rm(\ref{eq:secondclaim})}, then
\begin{align}
\label{eq:fifthclaim}
  \dot t &=  c\,{\bar{n}}\,\rho
  \left\{
  \frac
  {(t_1-t_2)(t_2-t_3)(t_3-t_1) \cdot (x_1-x_2)(x_2-x_3)(x_3-x_1)   }
  {\left[t_1(x_2-x_3) + t_2(x_3-x_1) + t_3(x_1-x_2)\right]^2}
  \right\},\\
  \frac{\ddot t}{\dot t} &= c\rho\, \frac{\sum t_i(x_j-x_k)\left[
(\nu_i-\bar n)x_i + (\nu_j+\bar n)x_j + (\nu_k+\bar n)x_k\right]}
{t_1(x_2-x_3) + t_2(x_3-x_1) + t_3(x_1-x_2)},
\label{eq:sixthclaim}
\end{align}
the summation being over cyclic permutations\/ $i,j,k$ of\/ $1,2,3$, 
and\/ $\rho$ being defined as usual by\/ $\rho^{-1}\defeq \nu_1+\nu_2+\nu_3+\bar
n$, or equivalently by\/~{\rm(\ref{eq:rhodef0})}.  Additionally, the
reconstruction formulas\/ {\rm(\ref{eq:zerothclaim})} and\/
{\rm(\ref{eq:firstclaim})} for\/ $x_i$ and\/ $x_i-\nobreak x_j$ hold.
\label{lem:usedlater2}
\end{lemma}

\begin{remark*}
In the special case of a DH system with $(t_1,t_2,t_3)=(0,1,\infty)$, the
formulas (\ref{eq:fifthclaim}),\allowbreak(\ref{eq:sixthclaim}), and the
consequent $t(\cdot)\mapsto\allowbreak x(\cdot)$
map~(\ref{eq:zerothclaim}), were previously obtained as Eqs.\ (17)--(20) of
Ref.~\cite{Ablowitz99}.
\end{remark*}

\begin{proof}
Equations (\ref{eq:fifthclaim}),\allowbreak(\ref{eq:sixthclaim}) can be
verified by successively differentiating the gDH system~(\ref{eq:gDH}).
The formulas (\ref{eq:zerothclaim}),(\ref{eq:firstclaim}) then follow from
(\ref{eq:secondclaim}),\allowbreak(\ref{eq:fifthclaim}),\allowbreak(\ref{eq:sixthclaim})
by elementary manipulations.

It is worth noting that if $x=x(\tau)$ were constructed by
PE\nobreakdash-based integration, Eq.~(\ref{eq:fifthclaim}) would come by
eliminating~$t$ from the expressions for $x_1-\nobreak x_2$, $x_2-\nobreak
x_3$, $x_3-\nobreak x_1$ provided by~(\ref{eq:firstclaim}), and
Eq.~(\ref{eq:sixthclaim}) by substituting (\ref{eq:fifthclaim})
into~(\ref{eq:zerothclaim}), and solving for~$\ddot t/\dot t$.
\end{proof}

A `generalized Schwarzian' equation satisfied by any function $t=t(\tau)$
derived from a noncoincident gDH solution $x=x(\tau)$ will now be
introduced.  This ODE is invariant under any affine transformation
$\tau\mapsto \tau'\defeq\allowbreak A\tau+\nobreak B$.  It will play a role
in the Painlev\'e analysis of~\S\,\ref{sec:painleve}.

\begin{definition}
The (triangular) generalized Schwarzian equation, called the gSE and
denoted by ${\rm gS}_{t_1,t_2,t_3}(\nu_1,\nobreak \nu_1'; \allowbreak
\nu_2,\nobreak \nu_2'; \allowbreak \nu_3,\nobreak \nu_3';\bar n)$, is the
following nonlinear third-order equation for $t=t(\tau)$:
\begin{sizemultline}{\small}
\label{eq:gSE}
\frac{\dddot t}{\dot t^3} + (\bar n-2)
\left(\frac{\ddot t}{\dot t^2}\right)^{2}
+\left[\sum_{i=1}^3 \frac{1-(\nu_i+\bar n)-(\nu_i'+\bar n)}{t-t_i}\right]
\frac{\ddot t}{\dot t^2}\\
{}+
\frac1{\bar n}
\sum_{i=1}^3
\left[
\frac{(\nu_i+\bar n)(\nu_i'+\bar n)}{(t-t_i)^2}
+\frac{(\nu_i-\bar n)(\nu'_i-\bar n)-(\nu_j+\bar n)(\nu_j'+\bar
  n)-(\nu_k+\bar n)(\nu_k'+\bar n)}{(t-t_j)(t-t_k)}
\right]=0,
\end{sizemultline}
the subscripts $j,k$ in the second summand being the elements of
$\{1,2,3\}$ other than\/~$i$.  Here the vector $(\nu_1,\nobreak \nu_1';
\allowbreak \nu_2,\nobreak \nu_2'; \allowbreak \nu_3,\nobreak \nu_3';\bar
n)\in\allowbreak \bigl(\mathbb{C}^2\bigr)^3 \times
(\mathbb{C}\setminus\{0\})$ of parameters is required to satisfy the
Fuchsian condition $\sum_{i=1}^3 (\nu_i+\nobreak\nu_i') =\allowbreak
1-\nobreak2\bar n$.  Also, $t_1,t_2,t_3\in\mathbb{P}^1_t$ are distinct
singular values, so that if $t_3=\infty$ the gSE reduces to
\begin{sizemultline}{\small}
\label{eq:gSE2}
\frac{\dddot t}{\dot t^3} + (\bar n-2)
\left(\frac{\ddot t}{\dot t^2}\right)^{2}
+\left[\sum_{i=1}^2 \frac{1-(\nu_i+\bar n)-(\nu_i'+\bar n)}{t-t_i}\right]
\frac{\ddot t}{\dot t^2}\\
{}+
\frac1{\bar n}
\left\{\sum_{i=1}^2
\left[
\frac{(\nu_i+\bar n)(\nu_i'+\bar n)}{(t-t_i)^2}\right]
+\frac{(\nu_3-\bar n)(\nu'_3-\bar n)-(\nu_1+\bar n)(\nu_1'+\bar
  n)-(\nu_2+\bar n)(\nu_2'+\bar n)}{(t-t_1)(t-t_2)}\right\}
=0,
\end{sizemultline}
The gSE is invariant under each
transposition $\nu_i\leftrightarrow\nu_i'$.
\end{definition}

\begin{theorem}
Let\/ $x=x(\tau)$ be a noncoincident local meromorphic solution of a proper
gDH system\/ ${\rm gDH}(a_1,\nobreak a_2,\nobreak a_3;\allowbreak
b_1,\nobreak b_2,\nobreak b_3;c)$, and let\/ $t=t(\tau)$ be computed from\/
$x=x(\tau)$ by Eq.~{\rm(\ref{eq:secondclaim})}, for some choice of
distinct\/ $t_1,t_2,t_3\in\mathbb{P}^1_t$.  {\rm(}E.g., if\/
$(t_1,t_2,t_3)$ is\/ $(0,1,\infty)$ then\/ $t=-(x_2-\nobreak
x_3)/\allowbreak (x_1-\nobreak x_2)$.{\rm)} Then\/ $t$~will satisfy a gSE\/
${\rm gS}_{t_1,t_2,t_3}(\nu_1,\nobreak \nu_1'; \allowbreak \nu_2,\nobreak
\nu_2'; \allowbreak \nu_3,\nobreak \nu_3';\bar n)$, where the parameters,
satisfying\/ $\sum_{i=1}^3 (\nu_i+\nobreak\nu_i') =\allowbreak
1-\nobreak2\bar n$, are computed from\/ $(a_1,\nobreak a_2,\nobreak
a_3;\allowbreak b_1,\nobreak b_2,\nobreak b_3;c)$ by the formulas\/
{\rm(\ref{eq:abtozetasn})}.
\label{thm:gSE}
\end{theorem}

\begin{proof}
  When $(t_1,t_2,t_3)=(0,1,\infty)$, the gSE~(\ref{eq:gSE}) can readily be
  derived by manipulating the gDH HQDS~(\ref{eq:gDH}).  That
  Eq.~(\ref{eq:gSE}) in the reduced form~(\ref{eq:gSE2}) is satisfied by
  the Brioschi variable $t=-(x_2-\nobreak x_3)/\allowbreak (x_1-\nobreak
  x_2)$ in any proper gDH system with
  $(t_1,t_2,t_3)=\allowbreak(0,1,\infty)$, satisfying $b_1=b_2=b_3$, was
  proved by Bureau~\cite{Bureau72}; and the restriction $b_1=\nobreak
  b_2=\nobreak b_3$ can be relaxed.

  The case when $t_1,t_2,t_3$ are in general position, and $t$~is obtained
  by the more general $x(\cdot)\mapsto\allowbreak t(\cdot)$ map,
  Eq.~(\ref{eq:secondclaim}), also follows from~(\ref{eq:gDH}).  (Using a
  computer algebra system is recommended.)  Verification is facilitated by
  the formulas for $\dot t,\ddot t/\dot t$ given in
  Lemma~\ref{lem:usedlater2}.  Alternatively, one can simply show that
  Eq.~(\ref{eq:gSE}) is reduced to~(\ref{eq:gSE2}) by the unique M\"obius
  transformation of~$\mathbb{P}_t^1$ that maps $(t_1,t_2,t_3)$ to
  $(t_1,t_2,\infty)$, and in~particular to~$(0,1,\infty)$.
\end{proof}

\begin{remark*}
  A closely related claim dealing with PE-based integration, namely that if
  $\tau=\tau(t)$ is computed from a nonzero local PE solution $f=f(t)$ by
  the procedure of Theorem~\ref{thm:integration}, then the locally defined
  inverse function $t=t(\tau)$ must satisfy the gSE~(\ref{eq:gSE}), can
  also be proved.  This is an exercise in calculus; cf.\ the derivation of
  Eq.~(\ref{eq:zerothclaim}).  Starting with the PE~(\ref{eq:PE}) satisfied
  by~$f$, one uses $\dot t= K^{2}(t) f^{1/\bar n}(t)$ to convert each ${\rm
    d}/{\rm d}\tau$ to a ${\rm d}/{\rm d}t$, and substitutes the
  definition~(\ref{eq:Kdef}) for $K^2=K^2(t)$.  After much manipulation one
  ends~up with a gSE satisfied by $t=t(\tau)$, including the Fuchsian
  condition, which in this derivation comes from Fuchs's relation
  $\sum_{i=1}^3 (\mu_i+\nobreak\mu_i')=1$ on the exponents of the~PE\null.
  Details are left to the reader.  A~significant feature of this
  alternative derivation is that unlike the initial (linear) PE for
  $f=f(t)$, the (nonlinear) gSE~(\ref{eq:gSE}) involves only the
  \emph{offset} exponents $\nu_1,\nobreak\nu_2,\nobreak\nu_3;\allowbreak
  \nu_1',\nobreak\nu_2',\nobreak\nu_3'$.  The offset vector~$\kappa$ does
  not appear explicitly.  Of~course if one integrates the gSE by
  substituting $\dot t=K^2(t)f^{1/\bar n}(t)$, with $K^2(t)$ defined as in
  Theorem~\ref{thm:integration}, then $\kappa_1,\kappa_2,\kappa_3$ will
  reappear at~least briefly.
\end{remark*}

Theorem~\ref{thm:gSE} provides the forward half of a (slightly restricted)
$\textrm{gDH}\leftrightarrow\textrm{gSE}$ correspondence, since it computes
a local gSE solution $t=t(\tau)$ from any local \emph{noncoincident}
solution of any \emph{proper} gDH system.  The reverse direction
$t(\cdot)\mapsto x(\cdot)$ is provided by Eq.~(\ref{eq:zerothclaim}), on
account of the final sentence of Lemma~\ref{lem:usedlater2}.  The proper
$\textrm{gDH}\leftrightarrow\textrm{gSE}$ correspondence is summarized as
follows.

\begin{theorem}
  For any proper gDH system\/ ${\rm gDH}(a_1,a_2,a_3;b_1,b_2,b_3;c)$, at
  any point in the\/ $\tau$\nobreakdash-plane there is a bijection\/
  $x\leftrightarrow t$ between\/ {\rm(i)} its meromorphic local solutions
  $x=x(\tau)$ that are noncoincident and\/ {\rm(ii)} the meromorphic local
  solutions\/ $t=t(\tau)$ of an associated gSE of the form\/
  {\rm(\ref{eq:gSE})}.  In the latter, the parameters\/ $(\nu_1,\nobreak
  \nu_1'; \allowbreak \nu_2,\nobreak \nu_2'; \allowbreak \nu_3,\nobreak
  \nu_3';\bar n)$ are computed from\/ $(a_1,\nobreak a_2,\nobreak
  a_3;\allowbreak b_1,\nobreak b_2,\nobreak b_3;c )$ by the formulas\/
  {\rm(\ref{eq:abtozetasn})}.  The maps\/ $t(\cdot)\mapsto\allowbreak
  x(\cdot)$, $x(\cdot)\mapsto\allowbreak t(\cdot)$ of the bijection are
  given by Eqs.\ {\rm(\ref{eq:zerothclaim})}, {\rm(\ref{eq:secondclaim})}.
  In all of this, the choice of {\rm(}distinct\/{\rm)}
  $t_1,t_2,t_3\in\mathbb{P}^1$ is arbitrary.  {\rm(}E.g., if\/
  $(t_1,t_2,t_3)$ is\/ $(0,1,\infty)$ then\/ $t=-(x_2-\nobreak
  x_3)/\allowbreak (x_1-\nobreak x_2)$.{\rm)}
\label{thm:gDHbasecoin}
\end{theorem}

If in disagreement with the hypotheses of the theorem, the solution
$x=x(\tau)$ is coincident (i.e., has coincident components), then the
$\mathbb{P}^1$\nobreakdash-valued function $t=t(\tau)$ computed
from~(\ref{eq:secondclaim}) will either be undefined (if $x_1=x_2=x_3$), or
be a constant: $t\equiv t_1$, $t_2$, or~$t_3$.  (The values $t=t_1,t_2,t_3$
correspond to the invariant planes $x_2-\nobreak x_3=0$, $x_3-\nobreak
x_1=0$, $x_1-\nobreak x_2=0$.)  But by convention, a function
$t\equiv\text{const}$ cannot be a solution of a Schwarzian equation such as
the gSE~(\ref{eq:gSE}), on account of divisions by zero.

Theorem~\ref{thm:gDHbasecoin} focuses on meromorphic local solutions of the
gSE, but of~course for a generic choice of $(\nu_1,\nobreak \nu_1';
\allowbreak \nu_2,\nobreak \nu_2'; \allowbreak \nu_3,\nobreak \nu_3';\bar
n)$, generic gSE solutions $t=t(\tau)$ will have movable branch points.  In
the Painlev\'e analysis of~$\S\,\ref{sec:painleve}$, such points will be
ruled~out by restricting the parameter vector.  By substituting the formal
statement $t=\allowbreak t(\tau)\sim\allowbreak t_* +\nobreak
C(\tau-\nobreak\tau_*)^p$ into the gSE, it is easy to verify the following.
If $t_*\neq t_1,t_2,t_3$ then $p\in\{ \pm1, \pm(n+\nobreak 1) \}$, with the
statement holding as~$\tau\to\tau_*$ if $\mathop{\text{Re}}p>0$ and
as~$\tau\to\infty$ if $\mathop{\text{Re}}p<0$.  In this, $n\defeq 1/(\bar n
-\nobreak 1)$ and $\bar n = (n+1)/n$.  For the gSE to have the Painlev\'e
property (PP), one expects the restriction that $n$~be an integer
or~$\infty$.  Also, if $t_*=t_i$ for one of $i=1,2,3$, one finds that
$p\in\{r_i,r_i'\}$, where $(r_i,r_i')\defeq -\bar n(1/\nu_i, 1/\nu_i')$.
As will be seen, for the gSE to have the~PP, the exponents $r_i,r_i'$ (if
defined and finite) will also need to be integers.

\subsection{Integrating DH systems}

Any proper DH system ${\rm DH}(\alpha_1,\alpha_2,\alpha_3\,|\,c)$ of the
form~(\ref{eq:DH}) is by definition a proper gDH system, with parameters
\begin{equation}
  a_i = -\alpha_ic/(1-\alpha_1-\alpha_2-\alpha_3),\qquad i=1,2,3,
\end{equation}
and $b_1=b_2=b_2=c/2$.  Moreover, it follows from the birational
correspondence (\ref{eq:zetasntoab}),(\ref{eq:abtozetasn}) that the
alternative gDH parameter vector $(\nu_1,\nobreak \nu_1'; \allowbreak
\nu_2,\nobreak \nu_2'; \allowbreak \nu_3,\nobreak \nu_3';\allowbreak \bar
n; c )$ equals $(-\alpha_1/2,\nobreak \alpha_1/2 ;\allowbreak
-\alpha_2/2,\nobreak \alpha_2/2 ;\allowbreak
-\alpha_3/2,\nobreak \alpha_3/2 ;\allowbreak 1/2;c )$.

Thus an integration scheme for proper DH systems follows from the
Papperitz-based scheme of~\S\,\ref{subsec:fromPEtogDH} for proper gDH
systems, by setting $\bar n=1/2$.  This DH-specific value for~$\bar n$ is
quite special: it permits the use of a \emph{ratio} of a pair of PE
solutions, and in~fact of a pair of Gauss hypergeometric functions.

\begin{lemma}
\label{lem:suppl}
  Let\/ $\bar n=1/2$ and let the offset vector\/ $\kappa$ be
  $\kappa_i=\allowbreak(1-\nobreak\mu_i-\nobreak\mu_i')/2$, $i=1,2,3$.
  Then the definition of\/ $\tau=\tau(t)$ given in
  Theorem\/~{\rm\ref{thm:integration}} is equivalent to
  \begin{displaymath}
    \tau(t) = f^{(1)}(t) / f^{(2)}(t),
  \end{displaymath}
  where\/ $f^{(1)},f^{(2)}$ are any independent analytic local solutions of
  the Papperitz equation.  Also, the proper gDH system produced by the
  theorem has\/ $b_1=b_2=b_3=c/2$, and thus reduces to a proper DH
  system\/ ${\rm DH}(\alpha_1,\alpha_2,\alpha_3\,|\,c)$.
\end{lemma}
\begin{proof}
The derivative of any ratio of solutions $f^{(1)}(t) / f^{(2)}(t)$ with
respect to~$t$ equals $W(t)/\!\left[f^{(2)}(t)\right]^2$, where $W(t)$ is
the Wronskian $f^{(1)}_tf^{(2)} - f^{(2)}_tf^{(1)}$.  The PE~(\ref{eq:PE})
is of the form $\left[D_t^2+Q(t)D_t+Q(t)\right]f=0$, so that
\begin{equation}
  W(t)=\exp\, -\!\int^t Q(t)\,dt \propto 
(t-t_1)^{1-\mu_1-\mu_1'}(t-t_2)^{1-\mu_2-\mu_2'} (t-t_3)^{1-\mu_3-\mu_3'},
\end{equation}
where the proportionality constant depends on the choice
of~$f^{(1)},f^{(2)}$.  If $\bar n=1/2$, the condition in
Theorem~\ref{thm:integration} that $\tau$ is defined to satisfy is that
${{\rm d}\tau}/{{\rm d}t} = K^{-2}/f^2$, where $f$~is a nonzero solution of
the PE and
\begin{equation}
K(t)\propto(t-t_1)^{\kappa_1}(t-t_2)^{\kappa_2}(t-t_3)^{\kappa_3},
\end{equation}
the proportionality constant being a function of $t_1,t_2,t_3$.  By the
preceding, when $\kappa_i=(1-\mu_i-\mu_i')/2$ this condition will be
satisfied by the choice $\tau=f^{(1)} / f^{(2)}$, if $f$~is a suitably
scaled version of~$f^{(2)}$.  Also, 
\begin{equation}
\label{eq:pair}
\begin{aligned}
\nu_i &= \mu_i + (2\kappa_i-1){\bar{n}} = \mu_i + \kappa_i - 1/2 = (\mu_i-\mu_i')/2=-\alpha_i/2,\\
\nu_i' &= \mu_i' + (2\kappa_i-1){\bar{n}} = \mu_i' + \kappa_i -1/2 = (\mu'_i-\mu_i)/2=+\alpha_i/2.  
\end{aligned}
\end{equation}
That $b_1=b_2=b_3=c/2$ follows by substituting these expressions (and $\bar
n=1/2$) into the formula~(\ref{eq:zetasntoab}) for~$b_i$.
\end{proof}

In the $\bar n=1/2$ case the integration procedure for proper gDH systems
based on Theorem~\ref{thm:integration} specializes to the following
procedure for proper DH systems, when enhanced by Lemma~\ref{lem:suppl}.
(The formulas for $\mu_i,\mu'_i$ are as in~(\ref{eq:pair}).)

\begin{theorem}
  Any proper Darboux--Halphen system\/ ${\rm
    DH}(\alpha_1,\alpha_2,\alpha_3\,|\,c)$ of the form\/ {\rm(\ref{eq:DH})} can
  be integrated parametrically as follows.  For any exponent offset
  vector\/ $\kappa\in\mathbb{C}^3$ with\/ $\kappa_1+\kappa_2+\kappa_3=1$,
  let
  \begin{displaymath}
    \mu_i = (1-\alpha_i)/2 - \kappa_i,\qquad
    \mu'_i = (1+\alpha_i)/2 - \kappa_i,
  \end{displaymath}
  and let\/ $f^{(1)},f^{(2)}$ be any pair of independent analytic local
  solutions of a Papperitz equation\/~{\rm(\ref{eq:PE}\rm)} with
  exponents\/ $(\mu_1,\nobreak \mu_1'; \allowbreak \mu_2,\nobreak \mu_2';
  \allowbreak \mu_3,\nobreak \mu_3')$, and {\rm(}distinct\/{\rm)} singular
  points\/ $t_1,t_2,t_3\in\mathbb{P}^1$.  The independent and dependent
  variables, $\tau$ and\/~$x_i$, $i=1,2,3$, are then defined as functions
  of the local parameter\/~$t$ by
  \begin{align*}
    \tau(t) &= f^{(1)}(t) / f^{(2)}(t),\\
    x_i(t)&=(2/c)\,\frac{{\rm d}}{{\rm d}\tau}\log f_i,
  \end{align*}
  where\/ $f_i=\Delta_if^{(2)}$ and
    \begin{displaymath}
      \Delta_i(t) \defeq  
\left[-\,\frac{(t_i-t)(t_j-t_k)}{(t-t_j)(t_k-t_i)}\right]^{\mu_j}
\left[-\,\frac{(t-t_i)(t_j-t_k)}{(t_k-t)(t_i-t_j)}\right]^{\mu_k}.
    \end{displaymath}
  In this formula, $i,j,k$ is any cyclic permutation of\/ $1,2,3$.
\label{thm:integration0}
\end{theorem}
\begin{remark*}
The factors $\Delta_i$ and the functions~$f_i$ are defined even if one of
$t_1,t_2,t_3$ equals~$\infty$.  As in Theorem~\ref{thm:integration}, if
$(t_1,t_2,t_3)=(0,1,\infty)$ and
$\mu_1+\mu_2+\mu_3\eqdef\allowbreak\rho^{-1}$, i.e.,
$\rho=2/\allowbreak(1-\nobreak\alpha_1-\nobreak\alpha_2-\nobreak\alpha_3)$,
the~$f_i$ will satisfy
$\left({f_1}^\rho,{f_2}^\rho,{f_3}^\rho\right)=\allowbreak{f_1}^\rho\times\allowbreak
\bigl(1,\,(1-\nobreak t)/t,\,-1/t\bigr)$.
\end{remark*}

This parametric integration of any proper DH system is complete, since by
examination it includes three free parameters.  There are three (rather
than four) in the choice of $f^{(1)},f^{(2)}$, and hence in
$t\mapsto(\tau;x_1,x_2,x_3)$, as multiplying $f^{(1)},f^{(2)}$ by a common
factor affects neither $\tau=\tau(t)$ nor $x=x(t)$.

The integration can be `customized' in the sense that the offset
vector~$\kappa$ can be chosen to make the exponents $(\mu_1,\nobreak
\mu_1'; \allowbreak \mu_2,\nobreak \mu_2'; \allowbreak \mu_3,\nobreak
\mu_3')$ agree with those of any of the three normal-form
P\nobreakdash-symbols: (\ref{eq:GHPsymbol}), (\ref{eq:SAHPsymbol}),
and~(\ref{eq:thirdHPsymbol}).  That is, the Papperitz equation can be
chosen to agree with any of the normal forms (\ref{eq:GHE}),
(\ref{eq:SAHE}), and~(\ref{eq:thirdHE}) of the~PE, if one also chooses
$(t_1,t_2,t_3)=(0,1,\infty)$.  By examination, the first two of the
respective offset vectors are $(1-\nobreak\alpha_1, \allowbreak
1-\nobreak\alpha_2, \allowbreak \alpha_1+\nobreak\alpha_2)/2$
and~$(0,0,1)$.  But when the solutions $f^{(1)},f^{(2)}$ of the normal-form
PE are expressed in~terms of the Gauss hypergeometric function~${}_2F_1$,
the same formulas for $\tau=\tau(t)$ and $x=x(t)$ in~terms of~${}_2F_1(t)$
will always result.  Details are left to the reader.

Like the parametric integration in~\S\,\ref{subsec:fromPEtogDH} of a proper
gDH system, this integration of a proper DH system constructs the general
solution but is not adapted to constructing nongeneric, special solutions:
coincident ones, in the sense of Definition~\ref{def:singular}.  The ray
solutions (\ref{eq:raysoln}) and~(\ref{eq:newscale}), each proportional to
$(\tau-\tau_*)^{-1}$ and having coincident components, have been mentioned.
However, since any proper DH system is a proper gDH system with
$\nu_i=-\alpha_i/2$, $\nu_i'=\alpha_i/2$, and~$\bar n=1/2$, a
specialization of Theorem~\ref{thm:gDHbasecoin} on the proper ${\rm
  gDH}\leftrightarrow{\rm gSE}$ correspondence will hold.  In the DH case
the gSE specializes to what will be called the~SE\null.

\begin{definition}
The (triangular) Schwarzian equation, called the SE and denoted by ${\rm
  S}_{t_1,t_2,t_3}(\alpha_1,\alpha_2,\alpha_3)$, is the following nonlinear
third-order equation for $t=t(\tau)$:
\begin{equation}
\{\tau,t\} \defeq -\,\frac{\dddot t}{\dot t^3}+\frac32\left(\frac{\ddot t}{\dot t^2}\right)^2\!
=
\frac12\sum_{i=1}^3\left[\frac{1-\alpha_i^2}{(t-t_i)^2}+\frac{(1-\alpha_i^2)-(1-\alpha_j^2)-(1-\alpha_k^2)}{(t-t_j)(t-t_k)}\right],
\label{eq:SE}
\end{equation}
the subscripts $j,k$ in the summand being the elements of $\{1,2,3\}$ other
than~$i$.  Here $t_1,t_2,t_3\in\mathbb{P}^1$ are distinct singular values,
$(\alpha_1,\alpha_2,\alpha_3)\in\mathbb{C}^3$ is a vector of angular
parameters, and the left side $\{\tau,t\}=-\{t,\tau\}/\dot t^2$ is the
standard Schwarzian derivative of $\tau$ with respect to~$t$.  The SE is
invariant under each negation $\alpha_i\mapsto\nobreak-\alpha_i$.
\end{definition}

Also, in the case of a proper DH (rather than merely a proper gDH) system,
the $t(\cdot)\mapsto\allowbreak x(\cdot)$ map~(\ref{eq:zerothclaim})
specializes to
\begin{equation}
\begin{split}
  x_i & =c^{-1}\,\frac{\rm d}{{\rm d}\tau}\log
  \left[
    \frac{\dot t}{(t-t_i)^{\alpha_j+\alpha_k}\,(t-t_j)^{1-\alpha_j}\,(t-t_k)^{1-\alpha_k}}
    \right]\\
     &= c^{-1}\,
  \left[ \frac{\ddot t}{\dot t} - \sum_{l=1}^3(1-\alpha_l)\,\frac{\dot
      t}{t-t_l}\right] + c^{-1}(1-\alpha_1-\alpha_2-\alpha_3)
  \,\frac{\dot t}{t-t_i}.
\end{split}
\label{eq:reduction}
\end{equation}
The following is the specialization of Theorem~\ref{thm:gDHbasecoin} to the
DH case.
\begin{theorem}
  For any proper DH system\/ ${\rm DH}(\alpha_1,\alpha_2,\alpha_3\,|\,c)$,
  at any point in the\/ $\tau$\nobreakdash-plane there is a bijection\/
  $x\leftrightarrow t$ between\/ {\rm(i)} its meromorphic local solutions
  $x=x(\tau)$ that are noncoincident and\/ {\rm(ii)} the meromorphic local
  solutions\/ $t=t(\tau)$ of an associated SE of the form\/
  {\rm(\ref{eq:SE})}.  The maps\/ $t(\cdot)\mapsto\allowbreak x(\cdot)$,
  $x(\cdot)\mapsto\allowbreak t(\cdot)$ of the bijection are given by
  Eqs.\ {\rm(\ref{eq:reduction})}, {\rm(\ref{eq:secondclaim})}.  In all of
  this, the choice of {\rm(}distinct\/{\rm)} $t_1,t_2,t_3\in\mathbb{P}^1$
  is arbitrary.  {\rm(}E.g., if\/ $(t_1,t_2,t_3)$ is\/ $(0,1,\infty)$
  then\/ $t=-(x_2-\nobreak x_3)/\allowbreak (x_1-\nobreak x_2)$.{\rm)}
\label{thm:DHbasecoin}
\end{theorem}

The proper $\textrm{gDH}\leftrightarrow\textrm{gSE}$ correspondence thus
specializes to a proper $\textrm{DH}\leftrightarrow\textrm{SE}$
correspondence.

The SE (\ref{eq:SE}) is familiar from conformal mapping
theory~\cite{Nehari52}.  It is satisfied by any meromorphic function
$t=t(\tau)$ that maps the interior of a hyperbolic triangle with angles
$\pi(\alpha_1,\alpha_2,\alpha_3)$ in the complex $\tau$\nobreakdash-plane,
where $\alpha_i\ge\nobreak0$ for all~$i$ with
$\alpha_1+\nobreak\alpha_2+\allowbreak\alpha_3<\nobreak1$, to a disk in the
$t$\nobreakdash-plane that is bounded by the circle through $t_1,t_2,t_3$.
It is most familiar in the case when $t_1,t_2,t_3$ are taken to be the
collinear points $0,1,\infty$, and the disk becomes the upper half
$t$\nobreakdash-plane.

When $(\alpha_1,\alpha_2,\alpha_3)=\allowbreak
(\frac1{N_1},\frac1{N_2},\frac1{N_3})$, where each~$N_i$ is a~positive
integer or~$\infty$, with $\alpha_1+\nobreak \alpha_2+\allowbreak
\alpha_3<\nobreak 1$, a standard result from conformal mapping theory
applies.  In this case the function $t$~(and hence each~$x_i$,
by~(\ref{eq:reduction})) can be extended from the hyperbolic triangle in
the $\tau$\nobreakdash-plane by the Schwarz reflection principle, i.e.\ by
reflecting the triangle and its images repeatedly through their sides,
thereby tessellating a larger domain: the interior of a disk or half-plane.
By definition, the group of reflections acting on this domain is a triangle
group $\Delta(N_1,N_2,N_3)<{\it PSL}(2,\mathbb{C})$, under which
$t$~(though not $x_1,x_2,x_3$) will be automorphic, i.e., invariant.

As will be explained 
in Part~II,
in the upper half-plane case
the logarithmic derivatives $x_1,x_2,x_3$ can be viewed as quasi-modular
forms, with an affine-linear transformation law under
$\Delta(N_1,N_2,N_3)<{\it PSL}(2,\mathbb{R})$ that resembles
Eq.~(\ref{eq:affinelinear}) below.  For several choices of $(N_1,N_2,N_3)$
for which $\Delta(N_1,N_2,N_3)$ is an arithmetic group, it will be possible
to express each~$x_i$ as an explicit $q$\nobreakdash-series.

\smallskip
The Schwarzian derivative $\{\tau,t\}$ is invariant under M\"obius
transformations of~$\tau$, which implies that the SE~(\ref{eq:SE}) is
invariant under a larger group than is the gSE~(\ref{eq:gSE}): not merely
under affine transformations of the independent variable, but under any
invertible point transformation $\tau\mapsto\tau'=(A\tau+\nobreak
B)/\allowbreak(C\tau+\nobreak D)$.  By the $t(\cdot)\mapsto\allowbreak
x(\cdot)$ map~(\ref{eq:reduction}), this $\tau\mapsto\tau'$ induces a
transformation of the solution $x=x(\tau)$ of the associated proper DH
system ${\rm DH}(\alpha_1,\alpha_2,\alpha_3\,|\,c)$.  But the
transformation can be investigated without reference to the
SE~(\ref{eq:SE}), as follows.

\begin{theorem}
  Any {\rm(}not necessarily proper\/{\rm)} Darboux--Halphen system\/ $\dot
  x=Q(x)$, denoted by\/ ${\rm DH}(a_1,a_2,a_3;c)$, has a\/
  {\rm3}\nobreakdash-parameter group of Lie point symmetries: for any\/
  $\pm\left(\begin{smallmatrix} A & B \\ C & D
  \end{smallmatrix}\right)\in{\it PSL}(2,\mathbb{C})$, the system is
  stabilized by a point transformation of\/
  $\mathbb{C}\times\mathbb{C}^3\ni (\tau;x)$ that faithfully represents\/
  $\pm\left(\begin{smallmatrix} A & B \\ C & D
  \end{smallmatrix}\right)$, namely
\begin{equation}
\label{eq:affinelinear}
\begin{split}
  \tau&\mapsto\tau'=(A\tau+B)/(C\tau+D),\\
  x_i&\mapsto x_i'=  (C\tau+D)^2 x_i + (c/2)\,C(C\tau+D), \qquad i=1,2,3.
\end{split}
\end{equation}
The corresponding infinitesimal statement is that the normalizer\/
$\mathcal{N}(v)$ of the dynamical vector field\/ $v=\partial_\tau +
\sum_{i=1}^3 Q_i(x_1,x_2,x_3)\partial_{x_i}$ on\/
$\mathbb{C}\times\mathbb{C}^3$ contains the\/
    {\rm3}\nobreakdash-dimensional Lie algebra of vector fields generated
    by\/ $v_0,v_1,v_2$, given by
\begin{displaymath}
  \partial_\tau,
  \qquad 
  \tau\partial_\tau - \sum_{i=1}^3 x_i\partial_{x_i},
  \qquad 
  (c/2)\,\tau^2\partial_\tau - \sum_{i=1}^3(c\,\tau x_i+1)\partial_{x_i}
\end{displaymath}
and having commutators
\begin{displaymath}
  [v_0,v_1]=v_0,\qquad [v_0,v_2]=c\,v_1,\qquad [v_1,v_2]=v_2,
\end{displaymath}
which for\/ $c\neq0$, e.g.\ for a \emph{proper} DH system, is isomorphic
to\/ $\mathfrak{sl}(2,\mathbb{C})$.
\label{cor:DH}
\end{theorem}
\begin{proof}
  Invariance of any DH system, i.e., any gDH system~(\ref{eq:gDH})
  satisfying $b_1=b_2=b_3=c/2$, under the stated action
  $(x,\tau)\mapsto(x',\tau')$ of any $\pm\left(\begin{smallmatrix} A & B
    \\ C & D
  \end{smallmatrix}\right)\in{\it PSL}(2,\mathbb{C})$, can be verified
  by direct computation.  A closely related fact should be mentioned: any
  parametric solution $(\tau,x)=(\tau(t),x(t))$ of ${\rm
    DH}(\alpha_1,\alpha_2,\alpha_3\,|\,c)$, constructed by the integration
  scheme of Theorem~\ref{thm:integration0}, is mapped to another such
  solution, as one sees by considering the effects of
  \begin{displaymath}
    \left(
    \begin{array}{c}
      f^{(1)} \\ f^{(2)}
    \end{array}
    \right)
    \mapsto 
    \pm
    \left(
    \begin{array}{cc}
      A & B \\ C & D
    \end{array}
    \right)
    \left(
    \begin{array}{c}
      f^{(1)} \\ f^{(2)}
    \end{array}
    \right),
  \end{displaymath}
    taking into account that $\tau(t)=f^{(1)}(t)/f^{(2)}(t)$.  

    Each of the vector fields $v_0,v_1,v_2$ is an infinitesimal symmetry on
    $\mathbb{C}\times\mathbb{C}^3\ni (\tau;x)$, because each generates a
    1\nobreakdash-parameter subgroup of ${\it PSL}(2,\mathbb{C})$.
    Alternatively, that $v_i\in\mathcal{N}(v)$ for each~$i$ can be verified
    directly: the Lie bracket (i.e., commutator) of each~$v_i$ with~$v$
    turns~out to be a scalar multiple of~$v$.
\end{proof}

Any gDH system is invariant under $v_0$ and~$v_i$, the first two of the
three infinitesimal Lie point transformations above, because it is
stabilized by translation of~$\tau$, and by $\tau\mapsto\tau/\lambda$ when
accompanied by $x\mapsto\lambda x$.  That~is, it is covariant under affine
transformations.  The point of the theorem is that in the DH case, the
system is also covariant under \emph{projective} transformations of~$\tau$
(infinitesimal or otherwise).

What is remarkable is that although a non\nobreakdash-DH gDH system may have only a
2\nobreakdash-parameter group of Lie point symmetries, if proper it will
nonetheless be integrable by the Papperitz-based integration scheme
of~\S\,\ref{subsec:fromPEtogDH}.

For any proper DH system, the existence of a Lie symmetry algebra
isomorphic to $\mathfrak{sl}(2,\mathbb{C})$ is reminiscent of the Chazy-III
equation.  The 3\nobreakdash-dimensionality of the Chazy-III symmetry
algebra facilitates integration, though $\mathfrak{sl}(2,\mathbb{C})$ is
not solvable and the usual Lie integration technique cannot be used.  (See
\cite{Clarkson96}, and also~\cite{Ibragimov94}.)  The integration of ${\rm
  DH}(\alpha_1,\alpha_2,\alpha_3\,|\,c)$ extends that of Chazy-III, as will
be seen
in Part~II\null.

\subsection{Alternative DH systems}

It is useful to compare the standardized proper DH system~(\ref{eq:DH}),
which appeared for the first time in~\cite{Maier15}, to other proper DH
systems in the literature.  In 1881, Halphen~\cite{Halphen1881c} used
${}_2F_1$ (or equivalently, theta constants) to integrate a
3\nobreakdash-dimensional HQDS that is identical to~(\ref{eq:DH}), up~to a
trivial reparametrization.  Halphen's system was a generalization of a
system introduced in 1878, in Darboux's analysis of triply orthogonal
surfaces.  (Darboux's system had
$\alpha_1=\alpha_2=\allowbreak\alpha_3=0$.)  Halphen's system has been the
basis of several recent papers, e.g.,~\cite{Harnad2000}.

An alternative form of Halphen's system has appeared in other recent papers
(e.g.,~\cite{Ablowitz2003,Ohyama96,Valls2006b}).  This form is
\begin{gather}
\label{eq:altDH}
\left\{
\begin{aligned}
  \dot y_1 = y_1^2 + (y_1-y_2)(y_3-y_1) - T^2,\\
  \dot y_2 = y_2^2 + (y_2-y_3)(y_1-y_2) - T^2,\\
  \dot y_3 = y_3^2 + (y_3-y_1)(y_2-y_3) - T^2,\\
\end{aligned}
\right.
\\
T^2 \defeq  \sum_{(i,j,k)}
\alpha_i^2(y_i-y_j)(y_k-y_i), \nonumber
\end{gather}
the summation being over cyclic permutations of~$(1,2,3)$.  The
system~(\ref{eq:altDH}) arises naturally in the symmetry reduction of the
self-dual Yang--Mills equations~\cite{Ablowitz2003}, and differs
from~(\ref{eq:DH}) except in the Darboux case
$\alpha_1=\alpha_2=\alpha_3=0$.  But the two systems are equivalent under
${\it GL}(3,\mathbb{C})$.  The following is an exercise in linear algebra.

\begin{proposition}
  For any\/ $(k_1,k_2,k_3)\in{\mathbb{C}}^3$ with\/ $k_1+k_2+k_3\neq1$,
  define\/ $y=(y_1,y_2,y_3)$ in terms of\/ $x=(x_1,x_2,x_3)$, and vice
  versa, by
  \begin{displaymath}
    \left\{
    \begin{aligned}
      y_1 &= (1-k_2-k_3) x_1 + k_2 x_2 + k_3 x_3,\\
      y_2 &= k_1 x_1 + (1-k_3-k_1) x_2 + k_3 x_3,\\
      y_3 &= k_1 x_1 + k_2 x_2 + (1-k_1-k_2) x_3
    \end{aligned}
    \right.
  \end{displaymath}
  and
  \begin{displaymath}
    \left\{
    \begin{aligned}
      x_1 &= (1-k_1-k_2-k_3)^{-1} \left[+(1-k_1) y_1 - k_2 y_2 - k_3 y_3\right],\\
      x_2 &= (1-k_1-k_2-k_3)^{-1} \left[-k_1 y_1 + (1-k_2) y_2 - k_3 y_3\right],\\
      x_3 &= (1-k_1-k_2-k_3)^{-1} \left[-k_1 y_1 - k_2  y_2 + (1-k_3) y_3\right].\\
    \end{aligned}
    \right.
  \end{displaymath}
  Then\/ $x=x(\tau)$ will satisfy a proper Darboux--Halphen system\/ ${\rm
    DH}(\alpha_1,\alpha_2,\alpha_3)$ in the form\/ {\rm(\ref{eq:DH})}, with
  the normalization\/ $c=2$, if and only if\/ $y=y(\tau)$ satisfies
  \begin{displaymath}
    \left\{
    \begin{aligned}
      \dot y_1 = y_1^2 + (y_1-y_2)(y_3-y_1) - T^2,\\
      \dot y_2 = y_2^2 + (y_2-y_3)(y_1-y_2) - T^2,\\
      \dot y_3 = y_3^2 + (y_3-y_1)(y_2-y_3) - T^2,\\
    \end{aligned}
    \right.
    \end{displaymath}
  where, with\/ $\rho\defeq 2/(1-\alpha_1-\alpha_2-\alpha_3)$ as usual,
    \begin{displaymath}
    T^2 \defeq  \sum_{(i,j,k)} 
    \frac{-k_i(k_i+\rho\alpha_i)}{(1-k_1-k_2-k_3)^2}(y_i-y_j)(y_k-y_i).
  \end{displaymath}
\label{prop:nontrad}
\end{proposition}

The alternative DH system~(\ref{eq:altDH}) comes from this lemma by letting
$(k_1,k_2,k_3)=-(\rho/2)(\alpha_1,\alpha_2,\alpha_3)$.  

Our choice of~(\ref{eq:DH}) as the standardized proper DH system is to an
extent arbitrary.  In principle one could choose the
system~(\ref{eq:altDH}), or any system produced by
Prop.~\ref{prop:nontrad}.  (The corresponding non-associative
algebras~$\mathfrak{A}$ are of~course isomorphic.)  But (\ref{eq:DH})~has
several advantages.  It is uncomplicated and easy to write; it extends to
the gDH system~(\ref{eq:gDH}), which is itself a slight extension of a
well-known system of Bureau~\cite{Bureau72,Bureau87}; it is naturally
associated to the (nice) GHE normal form of the PE, rather than to any
other form; and moreover it appears in the theory of ODE's satisfied by
modular forms
(see~\cite{Maier15} and Part~II)\null.

\section{Rational Morphisms of gDH Systems}
\label{sec:transformations}

\subsection{Goals and context}

The chief result obtained in this section is Theorem~\ref{thm:gDHsystems},
accompanied by Tables \ref{tab:sigmas} and~\ref{tab:restrictions}, and
Figure~\ref{fig:only}.  There is a rich collection of nonlinear but
\emph{rational} solution-preserving maps between gDH systems of the
form~(\ref{eq:gDH}), generally with different parameters.  The maps follow
from the PE\nobreakdash-based gDH integration scheme
of~\S\,\ref{subsec:fromPEtogDH}.  They come from liftings of PE's to other
PE's, or equivalently, from liftings of generalized Schwarzian equations
(gSE's) to other gSE's.

Each such rational map, from a gDH system
$\bigl({\mathbb{C}}^{3}\!,\ \dot{\tilde x}=\tilde Q(\tilde x)\bigr)$ to
another gDH system $\bigl({\mathbb{C}}^{3}\!,\ \dot x= Q( x)\bigr)$, is a
rational function $x=\Phi(\tilde x)$.  As such, it has a singular locus
in~$\mathbb{C}^3\ni\tilde x$, on which at~least one of $x_1,x_2,x_3$
diverges.  The map~$\Phi$ takes each local solution $\tilde x = \tilde
x(\tau)$ of $\dot {\tilde x}=\tilde Q(\tilde x)$ (beginning off the
singular locus, and not crossing~it) to a solution $x=x(\tau)$ of $\dot
x=Q(x)$.

The simplest example is the quadratic transformation $x=\Phi(\tilde x)$
given by
\begin{subequations}
\label{eq:quadraticlift}
\begin{equation}
\label{eq:quadraticliftmap}
x_1=\frac12(\tilde x_1 + \tilde x_3),
\qquad
 x_2 = \tilde x_2,
\qquad
x_3 = \frac{\tilde x_2(\tilde x_1+\tilde x_3) - 2\,\tilde x_1\tilde x_3}{2\,\tilde x_2 - (\tilde x_1+\tilde x_3)}, 
\end{equation}
the singular locus of which is the plane $2\tilde x_2-(\tilde x_1+\tilde
x_3)=0$.  This map~$\Phi$ is solution-preserving when the parameter vector
of $\dot{\tilde x}=\tilde Q(\tilde x)$, namely $(\tilde a_1,\nobreak \tilde
a_2,\nobreak \tilde a_3;\allowbreak \tilde b_1,\nobreak \tilde b_2,\nobreak
\tilde b_3;\allowbreak \tilde c)$, satisfies $\tilde a_3=\tilde a_1$,
$\tilde b_3=\tilde b_1$, and the parameter vector of $\dot x=Q(x)$, say
$(a_1,\nobreak a_2,\nobreak a_3;\allowbreak b_1,\nobreak b_2,\nobreak
b_3;\allowbreak c)$ with $c=\tilde c$, is computed from~it by
\begin{equation}
  \begin{alignedat}{5}
        \qquad  a_1 &= 2\,\tilde a_1, & \qquad & a_2&&=\tilde a_2, & \qquad & a_3&&= 2\,\tilde a_1 + \tilde a_2 - \tilde c,\\
     b_1 &= \tilde c - \tilde b_2, & \qquad & b_2&&=\tilde b_2, & \qquad & b_3&&=2\,\tilde b_1  + \tilde b_2- \tilde c.
  \end{alignedat}
\end{equation}
\end{subequations}
One can view (\ref{eq:quadraticliftmap}) as a degree\nobreakdash-2 map from
the projective plane~$\mathbb{P}^2_{\tilde x}$, on which $\tilde x_1,\tilde
x_2,\tilde x_3$ are homogeneous coordinates, to the projective
plane~$\mathbb{P}^2_{x}$, i.e., as
\begin{sizemultline}{\small}
{}\Phi\colon[\tilde x_1: \tilde x_2: \tilde x_3]\mapsto [x_1:x_2:x_3]\defeq{}\\
\bigl[(\tilde x_1+\tilde x_3)(\tilde x_1+\tilde x_3 - 2\,\tilde x_2) :
  2\,\tilde x_2(\tilde x_1+\tilde x_3-2\,\tilde x_2) : -2\bigl(\tilde
  x_2(\tilde x_1+\tilde x_3)-2\,\tilde x_1\tilde x_3\bigr)\bigr].\nonumber
\end{sizemultline}
A notable feature of~(\ref{eq:quadraticliftmap}) is that if no two of
$\tilde x_1,\tilde x_2,\tilde x_3$ coincide, then the quadruple $\tilde
x_1,\nobreak\tilde x_2;\allowbreak\tilde x_3,\nobreak x_3$ will be
equianharmonic: it will have cross-ratio~$2$.

It is difficult to construct solution-preserving maps between HQDS's,
though they can be treated abstractly: R\"ohrl~\cite[\S\,5]{Rohrl77}
defines a category of which the objects are HQDS's and the morphisms are
solution-preserving maps that are germs of analytic functions.  By
definition, any solution-preserving~$\Phi$ from
$\bigl({\mathbb{C}}^{3}\!,\ \dot{\tilde x}=\tilde Q(\tilde x)\bigr)$ to
$\bigl({\mathbb{C}}^{3}\!,\ \dot x= Q( x)\bigr)$ satisfies $Q=\Phi_*\tilde
Q$, i.e., the `determining equations'
\begin{equation}
\label{eq:determining}
   Q(\Phi(\tilde x)) = \left[D\Phi(\tilde x)\right]\tilde Q(\tilde x),
\end{equation}
a~system of coupled nonlinear PDE's for which there exists no~solution
algorithm.  

For $\tilde Q,Q$ of gDH type with suitably constrained parameter values, we
shall construct rational solutions~$\Phi$ of~(\ref{eq:determining}), each
of which can be viewed as a $\mathbb{P}^2_{\tilde x}\to\mathbb{P}^2_{x}$
map, by exploiting the Papperitz-based parametric integration of gDH
systems, as developed in~\S\,\ref{subsec:fromPEtogDH}.  Each~$\Phi$ will
come from a \emph{hypergeometric transformation}: a~lifting of a Papperitz
equation (PE) on $\mathbb{P}^1_t$ to a PE on a covering
$\mathbb{P}^1_{\tilde t}$, along a rational covering map $\tilde t\mapsto
t$.  (See Tables \ref{tab:coveringmapsclassical}
and~\ref{tab:coveringmapsnonclassical}.)  Such transformations, viewed as
taking ${}_2F_1$'s to ${}_2F_1$'s with changed parameters, originated with
Gauss and Kummer.  The quadratic, cubic, quartic, and sextic ones are
classical: they were worked~out in detail by Goursat~\cite{Goursat1881}.
(Post-19th century expositions are few; Ref.~\cite{Poole36} is worth
consulting.)  A~final collection of ${}_2F_1$~transformations, nonclassical
in having no~free parameter, was discovered only
recently~\cite{Vidunas2005,Vidunas2009}.

The existence of solution-preserving maps between certain gDH systems,
including DH systems, explains the DH representations of Chazy-III
solutions recently found by Chakravarty and
Ablowitz~\cite{Chakravarty2010}.  Any Chazy-III solution $u=u(\tau)$ can be
viewed as the first component of a solution $x=x(\tau)$ of the system ${\rm
  DH}(0,\frac13,\frac12)$, and by lifting this system to a system
$\dot{\tilde x}=\tilde Q(\tilde x)$, one can generate additional
DH~representations.  For instance, the lifting~(\ref{eq:quadraticlift})
yields $u=({\tilde x}_1+\nobreak{\tilde x}_3)/2$, where $\dot{\tilde
  x}=\tilde Q(\tilde x)$ is the system ${\rm DH}(0,\frac23,0)$.  This is
one of their representations.  
In Part~II, besides extending their representations to Chazy\nobreakdash-I,
II, VII, XI, and~XII, we shall show that the recently discovered
nonclassical transformations of~${}_2F_1$ yield novel and beautiful
DH~representations of the solutions of the $N=7,8,9$ cases of Chazy-XII.

In principle, a 3\nobreakdash-dimensional HQDS $\dot x=Q(x)$ such as a gDH
system may have a solution-preserving \emph{self-map}: a map that
transforms among solutions.  A~family of self-maps will be generated by any
polynomial vector field $q\colon\mathbb{C}^3\to\mathbb{C}^3$
in~$\mathcal{C}(Q)$, the centralizer of~$Q$; i.e., by any vector field
$\sum_{i=1}^3 q_i(x)\partial_{x_i}$ in the Lie symmetry algebra of the gDH
system.  It is known that if the non-associative algebra~$\mathfrak{A}$
coming from~$Q$ is unital, with a multiplicative identity, then $q$~will be
either linear in~$x$, in~which case it generates automorphisms
of~$\mathfrak{A}$, or quadratic; and the possible quadratic~$q$ are
bijective with the elements of a certain Jordan sub-algebra
of~$\mathfrak{A}$.  (See \cite{Koecher77}
and~\cite[Chaps.~9,10]{Walcher91}.)  This abstract result is made relevant
by being combined with the following: the algebra~$\mathfrak{A}$ associated
to any DH system is unital, and in~fact (see Ohyama~\cite{Ohyama2001}), any
3\nobreakdash-dimensional non-associative algebra over~$\mathbb{C}$ that is
unital is generically isomorphic to one coming from a DH system.  It may
therefore be possible to classify all solution-preserving self-maps of DH
systems.\footnote{The Lie symmetries of non\nobreakdash-DH HQDS's may
  include polynomial vector fields of degree~$>2$.  For Lotka--Volterra
  HQDS's, several examples occur in the tables of Ref.~\cite{Almeida95}.}
But here, we focus largely on the task of finding solution-preserving
maps~$\Phi$ between \emph{non-isomorphic} gDH systems, which maps are
required to be rational.  (In~\S\,\ref{subsec:equivs}, we do find an
order\nobreakdash-48 Coxeter group of isomorphisms between distinct gDH
systems.)

Among our rational $\Phi$, the `classical' ones are essentially the maps of
Harnad and McKay~\cite{Harnad2000}, to whom we owe much.  But the context
is different.  Their maps~$\Phi$ were between HQDS's satisfied by
3\nobreakdash-vectors of quasi-modular forms.  The HQDS's were specific DH
systems, with no~free parameters.  Our transformations apply more broadly:
the HQDS's mapped between are gDH systems with free parameters.  We
nonetheless share with Harnad--McKay a focus on the automorphism group of
the `upper,' or lifted algebra~$\tilde{\mathfrak{A}}$.  If the associated
ring of polynomial invariants in~$\tilde x_1,\tilde x_2,\tilde x_3$ has a
finite set of generators $\chi_1,\dots,\chi_r$, it is natural to express
$x=\Phi(\tilde x)$ in~terms of them.  (An~example is the quadratic
map~(\ref{eq:quadraticliftmap}): the only nontrivial automorphism of the
system $\dot{\tilde x}=\tilde Q(\tilde x)$, which satisfies $\tilde
a_3=\tilde a_1$ and $\tilde b_3=\tilde b_1$ by assumption, interchanges
$\tilde x_1,\tilde x_3$; and $\tilde x_2$, $\tilde x_1+\nobreak\tilde x_3$,
$\tilde x_1\tilde x_3$ can be chosen as the generators.)  In~general
$(\tilde x_1,\tilde x_2,\tilde x_3)\mapsto(\chi_1,\dots,\chi_r)$ will be
solution-preserving from the HQDS $\bigl({\mathbb{C}}^{3}\!,\ \dot{\tilde
  x}=\tilde Q(\tilde x)\bigr)$ to a dynamical system on~$\mathbb{C}^r$.
(See~\cite{Kinyon97}.)  The striking thing is that if one chooses
$x_1,x_2,x_3$ to be appropriate rational functions of the generators, this
system will become $\left({\mathbb{C}}^{3}\!,\ \dot{x}=Q(x)\right)$, i.e.,
another HQDS\null.

Interestingly, the quadratic map~(\ref{eq:quadraticliftmap}), denoted
by~$\bf2$ below, and also the cyclic cubic map~${\bf3}_{\bf c}$, are
solution-preserving even when applied to many HQDS's that are not of the
gDH form, and are not linearly equivalent to any gDH system.  (See Theorems
\ref{thm:full1} and~\ref{thm:full2}.)

\subsection{Liftings of PE's}
\label{subsec:liftings}

By employing the calculus of Riemann P\nobreakdash-symbols, one can derive
and explain many identities satisfied by the Gauss function~${}_2F_1$.
Each such identity (a~${}_2F_1$ transformation) comes from lifting a Gauss
hypergeometric equation (GHE) to another GHE, or more generally a PE to
another~PE, along a rational covering $R\colon \mathbb{P}^1_{\tilde
  t}\to\mathbb{P}^1_t$.  Usually $R$~has an additional interpretation, from
the theory of conformal mapping: it maps between Schwarzian triangle
functions~\cite{Nehari52}.

The quadratic transformations of~${}_2F_1$, which are especially well
known~\cite{Andrews99}, arise as follows.  Consider a PE~(\ref{eq:PE})
on~$\mathbb{P}^1_t$, for simplicity with singular points $t_1,t_2,t_3$
equal to $0,1,\infty$.  Consider the degree\nobreakdash-2 map
\begin{equation}
\label{eq:canon2}
t=R(\tilde t)=4\tilde t/(1+\tilde t)^2,
\end{equation}
which takes $\tilde t=0,\infty$ to~$t=0$; and also $\tilde t=1$ to $t=1$
and $\tilde t=-1$ to $t=\infty$, both with multiplicity~2.  In terms of the
characteristic exponents $(\mu_1,\nobreak \mu_1'; \allowbreak
\mu_2,\nobreak \mu_2'; \allowbreak \mu_3,\nobreak \mu_3')$ and their
differences $\alpha\defeq \allowbreak \mu_1'-\nobreak\mu_1$, $\beta\defeq
\allowbreak \mu_2'-\nobreak \mu_2$, and $\gamma\defeq \allowbreak
\mu_3'-\nobreak \mu_3$, one can write
\begin{sizeequation}{\small}
  \left\{
  \begin{array}{ccc}
    0 & 1 & \infty \\
    \hline
    \mu_1 & \mu_2 & \mu_3 \\
    \mu_1+\alpha & \mu_2+\beta & \mu_3+\gamma
  \end{array}
  \right\}(t)
  =
  \left\{
  \begin{array}{cccc}
    0 & 1 & -1 & \infty \\
    \hline
    \mu_1 & 2\mu_2 & 2\mu_3  & \mu_1\\
    \mu_1+\alpha & 2\mu_2+2\beta & 2\mu_3+2\gamma & \mu_1+\alpha
  \end{array}
  \right\}(\tilde t).
\end{sizeequation}
Specializing to the case $\mu_3=0$, $\gamma=\frac12$ yields
\begin{sizeequation}{\small}
  \left\{
  \begin{array}{ccc}
    0 & 1 & \infty \\
    \hline
    \mu_1 & \mu_2 & 0 \\
    \mu_1+\alpha & \mu_2+\beta & \frac12
  \end{array}
  \right\}(t)
  =
  \left\{
  \begin{array}{ccc}
    0 & 1 & \infty \\
    \hline
    \mu_1 & 2\mu_2 & \mu_1\\
    \mu_1+\alpha & 2\mu_2+2\beta & \mu_1+\alpha
  \end{array}
  \right\}(\tilde t),
\label{eq:Psymbollifting}
\end{sizeequation}
because any lifted singular point with exponents~$0,1$ (e.g., $\tilde t=-1$
here) is no~singular point at~all, but rather an ordinary point.

Equation~(\ref{eq:Psymbollifting}) describes the lifting of the PE
on~$\mathbb{P}^1_t$ to one on~$\mathbb{P}^1_{\tilde t}$, under the change
of variables $t=R(\tilde t)$.  The singular points $\tilde t_1,\tilde
t_2,\tilde t_3$ on~$\mathbb{P}^1_{\tilde t}$ are equal to $0,1,\infty$,
like those on~$\mathbb{P}^1_{t}$.  The lifting acts on exponent differences
by $(\alpha,\beta,\frac12)\leftarrow\allowbreak(\alpha,2\beta,\alpha)$,
which is a parametric restriction.  In~solution terms, the meaning
of~(\ref{eq:Psymbollifting}) is that if $f(t)$~is a local solution of the
original~PE, $\tilde f(\tilde t)\defeq f(R(\tilde t))$ will be one of
the lifted~PE\null.  By setting $\mu_1=0$ and comparing
(\ref{eq:Psymbollifting}) with the GHE
P\nobreakdash-symbol~(\ref{eq:GHPsymbol}) and the
definition~(\ref{eq:abc}) of the ${}_2F_1$ parameters $(a,b;c)$, one
readily derives
\begin{equation}
\label{eq:quadexample}
(1-t)^{-\mu_2} \, {}_2F_1(\mu_2,\mu_2+\tfrac12;\,1-\alpha;\,t) =
(1-\tilde t)^{-2\mu_2} \, {}_2F_1(2\mu_2,2\mu_2+\alpha;\,1-\alpha;\,\tilde t).
\end{equation}
Here the left and right prefactors serve to shift one of the exponents at
each of $t=1$ and $\tilde t=1$ to zero, in agreement
with~(\ref{eq:GHPsymbol}).

Equation~(\ref{eq:quadexample}) is one of many quadratic transformations
of~${}_2F_1$, each with two free parameters.  The others can be generated
by permuting the singular points $\tilde t=0,1,\infty$ and/or
$t=0,1,\infty$, i.e., by pre- and post-composing the covering map~$R$ with
M\"obius transformations that permute them.  Each modified~$R$ will yield a
PE lifting, but to yield a transformation of~${}_2F_1$ as~well, $R(0)=0$
must be satisfied.  This is because ${}_2F_1$ is only defined near the
origin, as the sum of the series~(\ref{eq:hypseries}).

The features of the PE lifting induced by the canonical
degree\nobreakdash-2 map~$R$ of~(\ref{eq:canon2}), which will be denoted
by~${\bf 2}$, are indicated in its \emph{branching schema}, which is
$1+\nobreak1=2={\it2}$.  The three `slots,' separated by equals signs,
indicate the multiplicities with which points on~$\mathbb{P}^1_{\tilde t}$
are mapped to $t=0,1,\infty$ respectively; so a value greater than~$1$
indicates a critical point of the map.  The italicization of the
second~`$2$' means that in the PE lifting, the corresponding critical point
($\tilde t=-1$) is kept from being a singular point by restricting the
exponents at the critical value `beneath'~it ($t=1$) to be $0,\frac12$.  In
the schema of any map used for lifting a PE to a PE, there will be exactly
three non-italicized multiplicities, pertaining to $\tilde t=0,1,\infty$
(in some order).

Another covering yielding a PE lifting is the degree\nobreakdash-4 map
\begin{equation}
  t=R(\tilde t) = \frac{16\,\tilde t(1-\tilde t)}{(1+4\tilde t-4\tilde t^2)^2}
  = 1 - \frac{(1-2\tilde t)^4}{(1+4\tilde t-4\tilde t^2)^2},
\end{equation}
which will be denoted by ${\bf4}_{\bf bq}$.  The associated branching
schema is $1+1+2={\it4}={\it2(2)}$, meaning
$1+\nobreak1+\nobreak2={\it4}={\it 2}+\nobreak{\it 2}$, since (for
instance) this~$R$ takes $\tilde t=0,1,\infty$ to $t=0$, with respective
multiplicities $1,1,2$; and $\tilde t=\frac12$ to $t=1$ with
multiplicity~$4$.  By examination, the associated lifting acts as
$(\alpha,\frac14,\frac12)\leftarrow(\alpha,\alpha,2\alpha)$ on the vector
of constrained exponent differences.  The resulting quartic ${}_2F_1$
transformation is left as an exercise to the reader.  The `{\bf bq}' stands
for \emph{biquadratic}, as this quartic covering turns~out to factor into
two quadratic ones.  This will be indicated by ${\bf 4}_{\bf bq}\sim {\bf
  2}\circ{\bf 2}$.  The symbol~$\sim$ is used because this is not a
composition of the above map~${\bf 2}$ with itself, a permutation of
singular points being involved.  Specifically,
\begin{equation}
\frac{16\,\tilde t(1-\tilde t)}{(1+4\tilde t-4\tilde t^2)^2} =
\frac{4\,\tilde t}{(1+\tilde t)^2} \circ \frac{\tilde t}{\tilde t-1} \circ \frac{4\,\tilde
  t}{(1+\tilde t)^2}\circ \frac{\tilde t}{\tilde t-1}.
\end{equation}
In symbolic terms, ${\bf 4}_{\bf
  bq}={\bf2}\circ\sigma_{(23)}\circ{\bf2}\circ\sigma_{(23)}$, in which each
M\"obius transformation $\tilde t\mapsto \sigma_{(23)}(\tilde
t)\defeq\tilde t/(\tilde t-1)$ interchanges the points~$1,\infty$.

The quadratic, cubic, quartic and sextic PE liftings (and the resulting
${}_2F_1$ transformations, each with at~least one free parameter), were
worked~out in 1881 by Goursat~\cite{Goursat1881}, and the ones yielding
transformations with no~free parameter much more
recently~\cite{Vidunas2005,Vidunas2009}.  In Tables
\ref{tab:coveringmapsclassical} and~\ref{tab:coveringmapsnonclassical} a
canonical covering map for each type of ${}_2F_1$ transformation is given,
chosen in an \emph{ad~hoc} way (there is no~generally accepted convention
for choosing a representative unique up~to isomorphism).  Each map~$\tilde
t\mapsto t=R(\tilde t)$ is written as $-{{P}}_1(\tilde
t)/{{P}}_3(\tilde t)$, a~quotient of polynomials, so that
\begin{equation}
  t=-{{P}}_1/{{P}}_3, \qquad (t-1)/t=-{{P}}_2/{{P}}_1, \qquad 1/(1-t)=-{{P}}_3/{{P}}_2,
\end{equation}
where ${{P}}_1+{{P}}_2+{{P}}_3=0$.  According to our convention each
canonical~$R$ must map $\tilde t=0,\infty$ to $t=0$; also $R^{-1}(\infty)$
must contain at~least one ordinary point, so the third ($t=\infty$) slot in
the branching schema must include at~least one italicized integer~$>1$.
(All italicized integers in any single slot must be equal.)  These
requirements are admittedly arbitrary, but loosening them will merely
permute the points $\tilde t=0,1,\infty$ and/or $t=0,1,\infty$, and modify
${{P}}_1,{{P}}_2,{{P}}_3$ correspondingly.

\afterpage{\clearpage\rm
\begin{landscape}
\ 
\vfill
\begin{table}[h]
  \caption{Standardized covering maps $t=-{{P}}_1(\tilde
    t)/{{P}}_3(\tilde t)$ appearing in the classical and semiclassical
    transformations of~${}_2F_1$.  
    Each triple
    ${{P}}_1,{{P}}_2,{{P}}_3$ satisfies ${{P}}_1 + {{P}}_2 +
    {{P}}_3 = 0$.
    In the ${\bf 4}_{\bf bq},{\bf 12}_{\bf
      bq}$ maps, $r$~signifies $\tilde t-1/2$, and in the ${\bf 3}_{\bf
      c},{\bf 6}_{\bf c},{\bf 12}_{\bf c}$ maps, $s,\bar s$~signify $\tilde
    t+\nobreak\omega$, $\tilde t+\nobreak\bar\omega$, for $\omega$ a
    primitive cube root of unity.}
  \begin{center}
    \begin{tabular}{|l||l|l|l|}
      \hline
       \hfil covering map & \hfil branching schema &
      $(\alpha_1,\alpha_2,\alpha_3)\leftarrow(\tilde\alpha_1,\tilde\alpha_2,\tilde\alpha_3)$
      & \hfil $\bigl[{{P}}_1, {{P}}_2, {{P}}_3\bigr](\tilde   t)$ \\
      \hline
      \hline
      ${\bf2}$ & $1+1=2={\it 2}$ & $(\alpha,\beta,\frac12)\leftarrow(\alpha,2\beta,\alpha)$ 
      & $\bigl[-4{\tilde t},\, -(1-{\tilde t})^2, (1+{\tilde t})^2\bigr]$ \\
      \hline
      ${\bf3}$ & $1+2={\it 3}=1+{\it 2}$ & $(\alpha,\frac13,\frac12)\leftarrow(\alpha,\frac12,2\alpha)$
      &$\bigl[ -27{\tilde t},\, -(1-4{\tilde t})^3,\, (1-{\tilde t})(1+8{\tilde
           t})^2 \bigr]$ \\
      \hline
      ${\bf4}$ & $1+3=1+{\it 3}={\it (2)2}$ & $(\alpha,\frac13,\frac12)\leftarrow(\alpha,\frac13,3\alpha)$
      & $\bigl[-64{\tilde t},\, -(1-t)(1-9t)^3,$ \\
      & & & \hfill $(1+18t-27t^2)^2\bigr]$ \\
      \hline
      ${\bf6}\sim{\bf3}\circ{\bf2}$ & $1+1+4={\it(2) 3}={\it (3)2}$ & $(\alpha,\frac13,\frac12)\leftarrow(\alpha,\alpha,4\alpha)$
      & $\bigl[-108{\tilde t}(1-\tilde t),\, -(1-16{\tilde t}+16{\tilde t}^2)^3,$ \\
      & & & \hfill $(1-2{\tilde t})^2(1+32{\tilde t}-32{\tilde t}^2)^2\bigr]$ \\
      \hline
      ${\bf6}_{\bf c}\sim{\bf3}\circ{\bf2}$ & $2+2+2={\it(2) 3}={\it (3)2}$ & $(\alpha,\frac13,\frac12)\leftarrow(2\alpha,2\alpha,2\alpha)$
      & $\bigl[27{\tilde t}^2(1-{\tilde t})^2,\, -4(1-{\tilde t}+{\tilde t}^2)^3, $ \\
      $\quad{}\sim{\bf2}\circ{\bf3_{\bf c}}$ & & & \hfill $(1+{\tilde t})^2(2-{\tilde t})^2(1-2{\tilde t})^2\bigr]$ \\
      & & & $\quad{}=\bigl[-(s^3-\bar s^3)^2,\, -4s^3\bar s^3,\, (s^3+\bar s^3)^2\bigr]$ \\
      \hline
      \hline
      ${\bf 3_{\bf c}}$ & $1+1+1={\it 3}={\it 3}$ & $(\alpha,\frac13,\frac13)\leftarrow(\alpha,\alpha,\alpha)$
      & $\bigl[ 3(\omega-\bar\omega)\tilde t (1-\tilde t),\, -(\tilde
        t+\bar\omega)^3,\, (\tilde t+\omega)^3  \bigr]$  \\
      & & & $\quad{}=\bigl[\bar s^3 - s^3,\, -\bar s^3,\, s^3\bigr]$ \\
\hline
      ${\bf 4_{\bf bq}}\sim{\bf 2\circ2}$ & $1+1+2={\it 4}={\it (2)2}$ & $(\alpha,\frac14,\frac12)\leftarrow(\alpha,\alpha,2\alpha)$
      & $\bigl[-16 \tilde t(1-\tilde t),\, -(1-2\tilde t)^4,$ \\
      & & & \hfill $(1+4\tilde t-4\tilde t^2)^2\bigr]$\\
      & & & $\quad{}=\bigl[ -4(1-4r^2),\,-16r^4,\,4(1-2r^2)^2  \bigr]$ \\
      \hline
      \hline
      ${\bf 12}_{\bf bq}\sim{\bf6}\circ{\bf2}$
      & $1+1+2+{\it 8}$ & $(\frac18,\frac13,\frac12)\leftarrow(\frac18,\frac18,\frac14)$
      & $\bigl[27r^8(1-4r^2),\, -4(1-4r^2+r^4)^3,$ \\
      $\quad{}\sim{\bf3}\circ{\bf2}\circ{\bf2}$ & $\quad\qquad{}={\it(4) 3}={\it (6)2}$ & & \hfill $(1-2r^2)^2(2-8r^2-r^4)^2\bigr]$ \\
      $\quad{}\sim{\bf3}\circ{\bf4}_{\bf bq}$ & & & \\
      \hline
      ${\bf 12}_{\bf c}\sim{\bf4}\circ{\bf3_{\bf c}}$ & $1+1+1+{\it 9}$ & $(\frac19,\frac13,\frac12)\leftarrow(\frac19,\frac19,\frac19)$
      & $\bigl[64 s^9 (\bar s^3 - s^3),\, -\bar s^3(9\bar s^3 - 8 s^3)^3, $ \\
      & $\quad\qquad{}={\it(4) 3}={\it (6)2}$ & & \hfill $(27\bar s^6 -
        36 s^3 \bar s^3 + 8 s^6)^2\bigr]$ \\
      \hline
    \end{tabular}
  \end{center}
\label{tab:coveringmapsclassical}
\end{table}
\vfill
\ 
\end{landscape}

\begin{landscape}
\ 
\vfill
\begin{table}[h]
  \caption{Standardized covering maps $t=-{{P}}_1(\tilde
    t)/{{P}}_3(\tilde t)$ appearing in the nonclassical transformations
    of~${}_2F_1$ that are defined over~$\mathbb{Q}$.  Each triple
    ${{P}}_1,{{P}}_2,{{P}}_3$ satisfies ${{P}}_1 + {{P}}_2 +
    {{P}}_3 = 0$.}
  \begin{center}
    \begin{tabular}{|l||l|l|}
      \hline
      covering map & branching schema, & $\bigl[{{P}}_1,{{P}}_2,{{P}}_3\bigr](\tilde t)$\\
      & $(\alpha_1,\alpha_2,\alpha_3)\leftarrow(\tilde\alpha_1,\tilde\alpha_2,\tilde\alpha_3)$ & \\
      \hline
      \hline
      ${\bf10}$ & $1+2+{\it 7}$ & $\bigl[{\tilde t}(32-81{\tilde t})^7$,   \\
      & $\quad\qquad{}= 1+{\it3(3)} = {\it5(2)}$ & $-4(1-{\tilde t})(256+17280{\tilde t}+5832{\tilde t}^2-6561{\tilde t}^3)^3$, \\
      & $(\frac17,\frac13,\frac12)\leftarrow(\frac17,\frac13,\frac27)$ & $(8192 -1271808{\tilde t} - 6127488{\tilde t}^2 + 7453296{\tilde t}^3 - 1948617{\tilde t}^4 + 1062882{\tilde t}^5)^2\bigr]$\\
      \hline
      ${\bf24_{\bf c}\sim{\bf8}\circ{\bf 3}_{\bf c}}$ & $1+1+1+{\it3(7)}$ 
      & $\bigl[-1728{\tilde t}(1-{\tilde t})(1+5{\tilde t}-8{\tilde t}^2+{\tilde t}^3)^7$,  \\
      & $\quad\qquad{}={\it8(3)}={\it12(2)}$  &
      $-(1-{\tilde t}+{\tilde t}^2)^3 (1-235{\tilde t} + 1430{\tilde t}^2 - 1695{\tilde t}^3 + 270{\tilde t}^4 + 229{\tilde t}^5 + {\tilde t}^6)^3$, \\
      & $(\frac17,\frac13,\frac12)\leftarrow(\frac17,\frac17,\frac17)$ 
      & $ (1+510{\tilde t}-14631{\tilde t}^2 + 80090{\tilde t}^3 - 218058{\tilde t}^4 + 316290{\tilde t}^5 - 253239{\tilde t}^6 $\\
      & & \hfill $ {}+131562{\tilde t}^7 -70998{\tilde t}^8 + 37950{\tilde t}^9 - 8955{\tilde t}^{10} - 522{\tilde t}^{11} + {\tilde t}^{12} )^2\bigr]$ \\
      \hline
    \end{tabular}
  \end{center}
\label{tab:coveringmapsnonclassical}
\end{table}
\vfill
\ 
\end{landscape}
}  

The tabulated maps split neatly into four classes.  The ones in the first
and second, denoted by ${\bf2},{\bf3},{\bf4},{\bf6},{\bf6}_{\bf c}$,
resp.\ ${\bf3}_{\bf c},{\bf4}_{\bf bq}$, were known to Goursat.  The ones
in the first can lift PE's with exponent differences
$(\alpha_1,\alpha_2,\alpha_3)$ equal to $(0,\frac13,\frac12)$, so they
arise naturally in the context of modular forms.
(See Part~II; also~\cite{Maier15}.)  
The maps ${\bf4},{\bf4}_{\bf bq}$ are distinct, as are ${\bf3},{\bf3}_{\bf
  c}$ and ${\bf6},{\bf6}_{\bf c}$.  The cubic map~${\bf3}_{\bf c}$ is
defined not over~$\mathbb{Q}$ but over~$\mathbb{Q}(\omega)$, where
$\omega$~is a primitive cube root of unity.  The subscript in~${\bf 3}_{\bf
  c}$ stands for `cyclic,' and is used because up to composition with
M\"obius transformations, ${\bf3}_{\bf c}$~is equivalent to the cyclic
covering $\tilde t\mapsto \tilde t^3$.  The map~${\bf 6}_{\bf c}$ is not
cyclic but merits the~${\bf c}$, because unlike the other sextic map~${\bf
  6}$, it has ${\bf3}_{\bf c}$ as a right factor.

The degree\nobreakdash-12 maps in the third class, denoted by ${\bf12}_{\bf
  bq}$ and~${\bf12}_{\bf c}$ in Table~\ref{tab:coveringmapsclassical},
could be called `semiclassical': they are compositions of maps that appear
in Goursat's transformations of~${}_2F_1$, and in consequence appear in
certain degree\nobreakdash-12 ${}_2F_1$ transformations (which turn~out to
have no~free parameter).  But the transformations seem not to have appeared
in the literature before~\cite{Vidunas2005,Vidunas2009}.  The subscripts on
${\bf12}_{\bf bq}$ and~${\bf12}_{\bf c}$ indicate the respective presence
of ${\bf4}_{\bf bq}$ and~${\bf3}_{\bf c}$ as right factors.  It may be that
${\bf12}_{\bf bq}$ and~${\bf12}_{\bf c}$ were missed because the
transformations of~${}_2F_1$ based on ${\bf4}_{\bf bq}$ and~${\bf3}_{\bf
  c}$ (unlike, say, those based on ${\bf2}$ and~${\bf3}$) are rather
obscure, though they can be found in~\cite{Goursat1881}.

The maps in the fourth class, denoted by ${\bf10}$ and~${\bf 24}_{\bf c}$
in Table~\ref{tab:coveringmapsnonclassical}, appear in the remaining
${}_2F_1$ transformations with no~free parameter; or more accurately, in
the only two that are defined over~$\mathbb{Q}$.  There are five additional
ones, the maps in which could be denoted by
${\bf6}',\nobreak{\bf8},\allowbreak{\bf9},
\allowbreak{\bf10'},\nobreak{\bf18}$, after their
degrees~\cite{Vidunas2005,Vidunas2009}; but being rather exotic, they will
not be considered here in any detail.  (The map~$\bf8$ is defined
over~$\mathbb{Q}(\omega)$, but the others are defined over such algebraic
number fields as $\mathbb{Q}({\mathrm{i}})$.)  The subscript in~${\bf
  24}_{\bf c}$ serves to indicate the presence of a ${\bf 3}_{\bf c}$ right
factor, and the consequent cyclic symmetry.  Of the four `${\bf c}$'~maps,
${\bf24}_{\bf c}$ like~${\bf6}_{\bf c}$ is defined over~$\mathbb{Q}$,
though ${\bf12}_{\bf c}$ like~${\bf3}_c$ is only defined
over~$\mathbb{Q}(\omega)$.

In a certain sense, the most symmetrical of the eleven canonical maps in
Tables \ref{tab:coveringmapsclassical}
and~\ref{tab:coveringmapsnonclassical} are ${\bf2}$, ${\bf3}_{\bf c}$,
and~${\bf6}_{\bf c}$.  This follows by considering $\mathbb{C}(\tilde
t),\mathbb{C}(t)$, the fields of rational functions on
$\mathbb{P}^1_{\tilde t},\mathbb{P}^1_{t}$.  For each of ${\bf2}$,
${\bf3}_{\bf c}$, and~${\bf6}_{\bf c}$, the field extension
$\mathbb{C}(\tilde t)/\mathbb{C}(t)$ is Galois, the Galois group
$G=G(\mathbb{C}(\tilde t)/\mathbb{C}(t))$ (i.e., the group of automorphisms
of~$\mathbb{C}(\tilde t)$ that fix~$\mathbb{C}(t)$) being isomorphic
respectively to the cyclic groups $\mathfrak{Z}_2$, $\mathfrak{Z}_3$, and
the symmetric group $\mathfrak{S}_3$ on three letters.  The last of these
comprises all M\"obius transformations of~$\mathbb{P}^1_{\tilde t}$ that
permute $\tilde t=0,1,\infty$, and the first two are the cyclic subgroups
$\tilde t\mapsto\allowbreak\{\tilde t,1/\tilde t\}$ and $\tilde
t\mapsto\allowbreak\{\tilde t,\allowbreak 1/(1-\nobreak\tilde
t),\allowbreak(\tilde t-\nobreak 1)/\tilde t\}$.  

The field extensions $\mathbb{C}(\tilde t)/\mathbb{C}(t)$ coming from most
of the other eight maps in Tables \ref{tab:coveringmapsclassical}
and~\ref{tab:coveringmapsnonclassical} are non-Galois: for each, some
element of~$\mathbb{C}(\tilde t)$ not in~$\mathbb{C}(t)$ is fixed by all
automorphisms of~$\mathbb{C}(\tilde t)$ that fix~$\mathbb{C}(t)$.  It can
be shown (cf.~\cite{Harnad2000}) that the extensions coming from the maps
${\bf3},{\bf4},{\bf6}$, though non-Galois, have the symbolic
representations $\mathfrak{S}_3/\mathfrak{Z}_2$,
$\mathfrak{A}_4/\mathfrak{Z}_3$, $\mathfrak{S}_4/\mathfrak{Z}_4$.  Here,
representing $\mathbb{C}(\tilde t)/\mathbb{C}(t)$ by $G_1/G_2$ indicates
that $\mathbb{C}(\tilde t),\mathbb{C}(t)$ have a common
extension~$\mathbb{C}(\tilde{\tilde t})$, with $\mathbb{C}(\tilde{\tilde
  t})/\mathbb{C}(t)$ and $\mathbb{C}(\tilde{\tilde t})/\mathbb{C}(\tilde
t)$ having respective Galois groups $G_1,G_2$.  But in each non-Galois
case, with $G_2$ not normal in~$G_1$, there is no naturally associated
group of M\"obius transformations of~$\mathbb{P}^1_{\tilde t}$.

\smallskip
It must be mentioned that each map $\tilde t=R(t)$ of Table
\ref{tab:coveringmapsclassical} or~\ref{tab:coveringmapsnonclassical} that
is defined over~$\mathbb{Q}$, except for~$\bf10$, has a conformal mapping
interpretation along the lines indicated in Ref.~\cite{Nehari52}.  Let
$\tau=\tau(t)$ be a (multivalued) Schwarzian triangle function, each branch
of which takes the upper-half $t$\nobreakdash-plane to a hyperbolic
triangle in the $\tau$\nobreakdash-plane with angles
$\pi(\alpha_1,\alpha_2,\alpha_3)$, the vertices being the images of
$(0,1,\infty)$.  Then if $\deg R=d$, the multivalued function $\tilde
\tau=\tilde \tau(\tilde t)\defeq \tau(R(\tilde t))$ will act similarly on
the upper-half $\tilde t$\nobreakdash-plane, each image triangle in the
$\tilde \tau$\nobreakdash-plane having angles
$\pi(\tilde\alpha_1,\tilde\alpha_2,\tilde\alpha_3)$ and being naturally
partitioned into $d$~sub-triangles with angles
$\pi(\alpha_1,\alpha_2,\alpha_3)$.

The map~$\bf2$ is an example.  It yields a partition of a hyperbolic
triangle with angles $\pi(\alpha,2\beta,\alpha)$ into two triangles with
angles $\pi(\alpha,\beta,\frac12)$; cf.~\cite[Fig.~43]{Nehari52}.  Some
discussion of the partitions produced by the remaining maps can be found in
Refs.~\cite{Hodgkinson18,Hodgkinson20}.  For instance, the map
${\bf24}_{\bf c}$ yields an elegant dissection of an equilateral hyperbolic
triangle with angles $\pi(\frac17,\frac17,\frac17)$ into 24~triangles with
angles $\pi(\frac17,\frac13,\frac12)$.

\subsection{Maps between gDH systems}
\label{subsec:mapsbetween}

With the aid of Theorem~\ref{thm:integration}, a rich collection of
solution-preserving maps between gDH systems can now be derived.
\begin{theorem}
  To each classical PE\nobreakdash-lifting covering map\/ $t=R(\tilde t)$
  {\rm(}of degree\/~$d$\,{\rm)} listed in
  Table\/~{\rm\ref{tab:coveringmapsclassical}}, i.e., to each of\/
  ${\bf2},{\bf3},{\bf4},{\bf6},{\bf6}_{\bf c}$ and\/ ${\bf3}_{\bf
    c},{\bf4_{\bf bq}}$, there is associated a rational map\/
  $x=\Phi(\tilde x)$ from a parametrized gDH system\/ $\dot{\tilde
    x}=\tilde Q(\tilde x)$ to a parametrized gDH system\/ $\dot x=Q(x)$.
  It can be viewed as a map\/ $\Phi\colon\mathbb{P}^2_{\tilde
    x}\to\mathbb{P}^2_{x}$, and is given by
  \begin{gather*}
    x_1 = \Sigma_1/d, \\
    \begin{aligned}
      x_1-x_2 &= -\Sigma_3/d\,\Sigma_2, \\
      x_2-x_3 &= -\Sigma_6/d\,\Sigma_2\Sigma_3, \\
      x_3-x_1 &= -\Sigma_2^2/d\,\Sigma_3,
    \end{aligned}
  \end{gather*}
  where\/ $\Sigma_k$, $k=1,2,3,6$, are certain homogeneous polynomials in\/
  $\tilde x_1,\tilde x_2,\tilde x_3$ of degree\/~$k$, which are listed in
  Table\/~{\rm\ref{tab:sigmas}} and satisfy\/ $\Sigma_2^3 + \Sigma_3^2 +
  \Sigma_6=0$.  A sufficient condition for\/ $\Phi$ to be
  solution-preserving from\/ $\dot{\tilde x}=\tilde Q(\tilde x)$, i.e.,
  ${\rm gDH}(\tilde a_1, \nobreak\tilde a_2, \nobreak\tilde a_3;
  \allowbreak\tilde b_1,\nobreak\tilde b_2,\nobreak\tilde
  b_3;\nobreak\tilde c),$ to\/ $\dot x=Q(x)$, i.e., to\/ ${\rm
    gDH}(a_1,\nobreak a_2,\nobreak a_3;\allowbreak b_1,\nobreak
  b_2,\nobreak b_3;\nobreak c)$, is that the respective parameter vectors
  be restricted and related as specified in
  Table\/~{\rm\ref{tab:restrictions}}.
\label{thm:gDHsystems}
\end{theorem}

\afterpage{\clearpage\rm
\begin{landscape}
\ 
\vfill
\begin{table}[h]
  \caption{Polynomials $\Sigma_k$ in the gDH transformations $x=\Phi(\tilde
    x)$ derived from classical ${}_2F_1$ transformations.  
    Each triple $\Sigma_2,\Sigma_3,\Sigma_6$
    satisfies $\Sigma_2^3 + \Sigma_3^2 + \Sigma_6=0$.
    In ${\bf 3}_{\bf
      c}$ and~${\bf 6}_{\bf c}$, $z,\bar z$~signify $\tilde x_1+\omega
    \tilde x_2+\bar\omega\tilde x_3$, $\tilde x_1+\bar\omega \tilde
    x_2+\omega\tilde x_3$.}
  \begin{center}
    \begin{tabular}{|l||l|l|}
      \hline
      {\rm covering map} & $\Sigma_1$ & $\Sigma_2,\Sigma_3,\Sigma_6$ \\
      \hline
      \hline
      ${\bf2}$ & ${\tilde x}_1+{\tilde x}_3$ 
      & $-({\tilde x}_3-{\tilde x}_1)^2$, \\
      & & $({\tilde x}_3-{\tilde x}_1)^2 ({\tilde x}_1-2{\tilde x}_2+{\tilde x}_3)$, \\
      & & $4({\tilde x}_1-{\tilde x}_2)({\tilde x}_2-{\tilde x}_3)({\tilde x}_3-{\tilde x}_1)^4$ \\
      \hline
      ${\bf3}$ & ${\tilde x}_1+2{\tilde x}_3$ 
      & $-({\tilde x}_3-{\tilde x}_1)(-{\tilde x}_1-3{\tilde x}_2+4{\tilde x}_3)$, \\
      & & $({\tilde x}_3-{\tilde x}_1)^2 ({\tilde x}_1-9{\tilde x}_2+8{\tilde x}_3)$, \\
      & & $-27({\tilde x}_1-{\tilde x}_2)^2({\tilde x}_2-{\tilde x}_3)({\tilde x}_3-{\tilde x}_1)^3$ \\
      \hline
      ${\bf4}$ & ${\tilde x}_1+3{\tilde x}_3$ 
      & $-({\tilde x}_3-{\tilde x}_1)(-{\tilde x}_1-8{\tilde x}_2+9{\tilde x}_3)$, \\
      & & $({\tilde x}_3-{\tilde x}_1)(-{\tilde x}_1^2 + 8{\tilde x}_2^2 + 27{\tilde x}_3^2 + 20{\tilde x}_1{\tilde x}_2 -36{\tilde x}_2{\tilde x}_3 - 18{\tilde x}_3{\tilde x}_1)$, \\
      & & $64({\tilde x}_1-{\tilde x}_2)^3({\tilde x}_2-{\tilde x}_3)({\tilde x}_3-{\tilde x}_1)^2$ \\
      \hline
      ${\bf6}\sim{\bf3}\circ{\bf2}$ & ${\tilde x}_1+{\tilde x}_2+4{\tilde x}_3$ 
      & $-({\tilde x}_1^2+{\tilde x}_2^2+16{\tilde x}_3^2 +14{\tilde x}_1{\tilde x}_2 - 16{\tilde x}_2{\tilde x}_3 - 16{\tilde x}_3{\tilde x}_1)$,  \\
      & & $({\tilde x}_1+{\tilde x}_2-2{\tilde x}_3)({\tilde x}_1^2 + {\tilde x}_2^2 -32{\tilde x}_3^2 -34{\tilde x}_1{\tilde x}_2 + 32{\tilde x}_2{\tilde x}_3 + 32{\tilde x}_3{\tilde x}_1)$, \\
      & & $-108({\tilde x}_1-{\tilde x}_2)^4({\tilde x}_2-{\tilde x}_3)({\tilde x}_3-{\tilde x}_1)$ \\
      \hline
      ${\bf6}_{\bf c}\sim{\bf3}\circ{\bf2}$ & $2{\tilde x}_1+2{\tilde x}_2+2{\tilde x}_3$
      & $-4({\tilde x}_1^2+{\tilde x}_2^2+{\tilde x}_3^2 -{\tilde x}_1{\tilde x}_2 - {\tilde x}_2{\tilde x}_3 - {\tilde x}_3{\tilde x}_1) = -4z\bar z$,  \\
      $\quad{}\sim{\bf2}\circ{\bf3_{\bf c}}$ & & $4(2{\tilde x}_1-{\tilde x}_2-{\tilde x}_3)(2{\tilde x}_2-{\tilde x}_3-{\tilde x}_1)(2{\tilde x}_3-{\tilde x}_1-{\tilde x}_2) = 4(z^3+\bar z^3)$, \\
      & & $432({\tilde x}_1-{\tilde x}_2)^2({\tilde x}_2-{\tilde
        x}_3)^2({\tilde x}_3-{\tilde x}_1)^2 = -16(\bar z^3 - z^3)^2 $ \\
      \hline
      \hline
      ${\bf3}_{\bf c}$ & ${\tilde x}_1+{\tilde x}_2+{\tilde x}_3$ 
      & $-({\tilde x}_1^2+{\tilde x}_2^2+{\tilde x}_3^2 -{\tilde x}_1{\tilde x}_2 - {\tilde x}_2{\tilde x}_3 - {\tilde x}_3{\tilde x}_1) = -({\tilde x}_1+\omega {\tilde x}_2
      + \bar\omega {\tilde x}_3)({\tilde x}_1+\bar\omega {\tilde x}_2 + \omega {\tilde x}_3) \eqdef -z\bar z$, \\
      & & $({\tilde x}_1+\omega {\tilde x}_2 + \bar\omega {\tilde x}_3)^3 \eqdef z^3$, \\
      & & $3(\omega-\bar\omega)({\tilde x}_1-{\tilde x}_2)({\tilde x}_2-{\tilde x}_3)({\tilde x}_3-{\tilde x}_1)\cdot({\tilde x}_1+\omega
      {\tilde x}_2 + \bar\omega {\tilde x}_3)^3  = (\bar z^3 -  z^3) z^3  $ \\
      \hline
      ${\bf 4_{\bf bq}}\sim{\bf 2\circ2}$ & ${\tilde x}_1+{\tilde x}_2+2{\tilde x}_3$ 
      & $-({\tilde x}_1+{\tilde x}_2 - 2{\tilde x}_3)^2$, \\
      & & $({\tilde x}_1+{\tilde x}_2-2{\tilde x}_3)({\tilde x}_1^2 + {\tilde x}_2^2 -4{\tilde x}_3^2 - 6{\tilde x}_1{\tilde x}_2 + 4{\tilde x}_2{\tilde x}_3 + 4{\tilde x}_3{\tilde x}_1)$, \\
      & & $-16({\tilde x}_1-{\tilde x}_2)^2({\tilde x}_2-{\tilde x}_3)({\tilde x}_3-{\tilde x}_1)\cdot({\tilde x}_1+{\tilde x}_2-2{\tilde x}_3)^2$ \\
      \hline
    \end{tabular}
  \end{center}
  \label{tab:sigmas}
\end{table}
\vfill
\ 
\end{landscape}

\begin{landscape}
\ 
\vfill
\begin{table}[h]
  \caption{Constraints on the parameters of the systems ${\rm gDH}(\tilde
    a_1,\tilde a_2,\tilde a_3;\tilde b_1,\tilde b_2,\tilde b_3;\tilde c)$
    and ${\rm gDH}(a_1,\nobreak a_2,\nobreak a_3;\allowbreak b_1,\nobreak
    b_2,\nobreak b_3;\nobreak c)$, under which $x=\Phi(\tilde x)$ is a
    solution-preserving map.  (In all cases, $c=\tilde c$.)}
  \begin{center}
    \begin{tabular}{|l||l|l||l|l|}
      \hline 
      {\rm covering map}$\vphantom{\tilde{\tilde A}}$ 
      & $(\tilde a_1,\tilde a_2,\tilde a_3)$
      & $(\tilde b_1,\tilde b_2,\tilde b_3)$ 
      & $(a_1,a_2,a_3)$ 
      & $(b_1,b_2,b_3)$ 
      \\
      \hline
      \hline
      ${\bf2}\vphantom{\tilde{\tilde A}}$
      & $\bigl(\tilde a_1,\tilde a_2, \tilde a_1\bigr)$
      & $\bigl(\tilde b_1,\tilde b_2, \tilde b_1\bigr)$
      & $\bigl(2\tilde a_1, \tilde a_2, 2\tilde a_1+\tilde a_2-\tilde c\bigr)$
      & $\bigl(\tilde c-\tilde b_2, \tilde b_2, 2\tilde b_1+\tilde b_2-\tilde c\bigr)$
      \\
      \hline
      ${\bf3}\vphantom{\tilde{\tilde A}}$
      & $\bigl(\tilde a,3\tilde a-\tilde c, 2\tilde a\bigr)$
      & $\bigl(\tilde b,3\tilde b-\tilde c, \tilde c-\tilde b\bigr)$
      & $\bigl(3\tilde a, 2(3\tilde a-\tilde c), 3(3\tilde a-\tilde c)\bigr)$
      & $\bigl(2\tilde c-3\tilde b, 3\tilde b-\tilde c, 3\tilde b-\tilde c\bigr)$
      \\
      \hline
      ${\bf4}\vphantom{\tilde{\tilde A}}$
      & $\bigl(\tilde a,\frac12(4\tilde a-\tilde c), 3\tilde a\bigr)$
      & $\bigl(\tilde b,\frac12(4\tilde b-\tilde c), \tilde c-\tilde b\bigr)$
      & $\bigl(4\tilde a, 2(4\tilde a-\tilde c), 3(4\tilde a-\tilde c)\bigr)$
      & $\bigl(\frac12(3\tilde c-4\tilde b), \frac12(4\tilde b-\tilde c), \frac12(4\tilde b-\tilde c)\bigr)$
      \\
      \hline
      ${\bf6}\sim{\bf3}\circ{\bf 2}\vphantom{\tilde{\tilde A}}$
      & $\bigl(\tilde a,\tilde a, 4\tilde a\bigr)$
      & $\bigl(\tilde b,\tilde b, \frac14(3\tilde c-2\tilde b)\bigr)$
      & $\bigl(6\tilde a, 2(6\tilde a-\tilde c), 3(6\tilde a-\tilde c)\bigr)$
      & $\bigl(\frac14(5\tilde c-6\tilde b), \frac14(6\tilde b-\tilde c), \frac14(6\tilde b-\tilde c)\bigr)$
      \\
      \hline
      ${\bf6}_{\bf c}\sim{\bf3}\circ{\bf 2}\sim{\bf2}\circ{\bf 3}_{\bf c}\vphantom{\tilde{\tilde A}}$
      & $\bigl(\tilde a,\tilde a, \tilde a\bigr)$
      & $\bigl(\tilde b,\tilde b, \tilde b\bigr)$
      & $\bigl(3\tilde a, 2(3\tilde a-\tilde c), 3(3\tilde a-\tilde c)\bigr)$
      & $\bigl(2\tilde c-3\tilde b, 3\tilde b-\tilde c, 3\tilde b-\tilde c\bigr)$
      \\
      \hline
      \hline
      ${\bf3}_{\bf c}\vphantom{\tilde{\tilde A}}$
      & $\bigl(\tilde a,\tilde a, \tilde a\bigr)$
      & $\bigl(\tilde b,\tilde b, \tilde b\bigr)$
      & $\bigl(3\tilde a, 3\tilde a-\tilde c, 3\tilde a-\tilde c\bigr)$
      & $\bigl(2\tilde c-3\tilde b, 3\tilde b-\tilde c, 3\tilde b-\tilde c\bigr)$
      \\
      \hline
      ${\bf4}_{\bf bq}\sim{\bf2}\circ{\bf 2}\vphantom{\tilde{\tilde A}}$
      & $\bigl(\tilde a,\tilde a, 2\tilde a\bigr)$
      & $\bigl(\tilde b,\tilde b, \frac12\tilde c\bigr)$
      & $\bigl(4\tilde a, 4\tilde a-\tilde c, 2(4\tilde a-\tilde c)\bigr)$
      & $\bigl(\frac12(3\tilde c-4\tilde b), \frac12(4\tilde b-\tilde c), \frac12(4\tilde b-\tilde c)\bigr)$
      \\
      \hline
      \hline
      ${\bf12}_{\bf bq}\sim{\bf6}\circ{\bf2}\vphantom{\tilde{\tilde A}}$
      & $\bigl(-\frac14,-\frac14,-\frac12\bigr)\tilde c$
      & $\bigl(\frac12,\frac12,\frac12\bigr)\tilde c$
      & $(-3,-8,-12)\tilde c$ 
      & $\bigl(\frac12,\frac12,\frac12\bigr)\tilde c$
      \\
      $\quad{}\sim{\bf3}\circ{\bf2}\circ{\bf2}\sim{\bf3}\circ{\bf4}_{\bf bq}$
      &
      &
      &
      &
      \\
      \hline
      ${\bf12}_{\bf c}\sim{\bf4}\circ{\bf 3_{\bf c}}\vphantom{\tilde{\tilde A}}$
      & $\bigl(-\frac16,-\frac16,-\frac16\bigr)\tilde c$
      & \quad\texttt{"}
      & $(-2,-6,-9)\tilde c$ 
      & \quad\texttt{"}
      \\
      \hline
      \hline
      ${\bf10}\vphantom{\tilde{\tilde A}}$
      & $\bigl(-\frac38,-\frac78,-\frac38\bigr)\tilde c$
      & \quad\texttt{"}
      & $(-6,-14,-21)\tilde c$ 
      & \quad\texttt{"}
      \\
      \hline
      ${\bf24}_{\bf c}\sim{\bf8}\circ{\bf 3_{\bf c}}\vphantom{\tilde{\tilde A}}$
      & $\bigl(-\frac14,-\frac14,-\frac14\bigr)\tilde c$
      & \quad\texttt{"}
      & $(-6,-14,-21)\tilde c$ 
      & \quad\texttt{"}
      \\
      \hline
    \end{tabular}
  \end{center}
\label{tab:restrictions}
\end{table}
\vfill
\ 
\end{landscape}
}                               

\begin{remark*}
  The seemingly \emph{ad hoc} convention for standardizing the maps~$R$ of
  Tables \ref{tab:coveringmapsclassical}
  and~\ref{tab:coveringmapsnonclassical} was chosen to regularize the
  resulting formulas for the rational functions $x_i$ of $\tilde x_1,\tilde
  x_2,\tilde x_3$.  Choosing another would permute $\tilde x_1,\tilde
  x_2,\tilde x_3$ and/or $x_1,x_2,x_3$.
\end{remark*}

\begin{proof}
  Consider (as it turns out, without loss of generality) the case when each
  gDH system is proper in the sense of Definition~\ref{def:proper}, and the
  solutions being mapped between, $x=x(\tau)$, $\tilde x=\tilde
  x(\tilde\tau)$ with $\tau=\tilde\tau$, are noncoincident in the sense of
  Definition~\ref{def:singular}.  Then the systems and their solutions will
  come via Theorem~\ref{thm:integration} from a pair of PE's with solutions
  $f(t),\tilde f(\tilde t)$.  Their singular points $(t_1,t_2,t_3)$,
  $(\tilde t_1,\tilde t_2,\tilde t_3)$ will both be taken to be
  $(0,1,\infty)$, and the exponent offset vectors $\kappa,\tilde \kappa$ to
  be $(0,0,1)$.

  The map $x=\Phi(\tilde x)$, and the restrictions on parameters that must
  be imposed for it to be solution-preserving, will be derived from the
  requirement that $t,\tilde t$ be related as in \S\,\ref{subsec:liftings}
  by the covering $t=R(\tilde t)=-{{P}}_1(\tilde t)/{{P}}_3(\tilde t)$.
  For the duration, let $i,j,k$ be any cyclic permutation of~$1,2,3$, with
  $y_k\defeq \allowbreak x_i-\nobreak x_j$, $\tilde
  y_k\defeq \allowbreak \tilde x_i-\nobreak \tilde x_j$.

  For each map~$R$ in Table~\ref{tab:coveringmapsclassical}, the
  corresponding polynomials $\Sigma_1,\Sigma_2,\Sigma_3,\Sigma_6$ in
  Table~\ref{tab:sigmas} will come from a formula for the
  solution-preserving map $x=\Phi(\tilde x)$,
  \begin{equation}
    \label{eq:fund}
    x_i = \tilde x_3 + d^{-1}\,\tilde y_3\left[\tilde t(1-\tilde
      t)({{P}}'_i/{{P}}_i)(\tilde t)\right]\big|_{\tilde t=-\tilde
      y_1/\tilde y_3},
  \end{equation}
  where $P_1+P_2+P_3=0$ as usual.  In fact, $\Sigma_2,\Sigma_3,\Sigma_6$
  will come from
  \begin{equation}
    \label{eq:fundcor}
    y_k = d^{-1}\,\tilde y_3\Bigl\{\tilde t(1-\tilde
    t)\left[({{P}}'_i/{{P}}_i)-({{P}}'_j/{{P}}_j)\right](\tilde
    t)\Bigr\}\big|_{\tilde t=-\tilde y_1/\tilde y_3} ,
  \end{equation}
  a corollary of~(\ref{eq:fund}) that will be derived first.  The proof is
  long and is therefore divided into three parts: (I)~the derivation of the
  parameter relationships of Table~\ref{tab:restrictions}, (II)~the
  derivation of~(\ref{eq:fundcor}), and (III)~the derivation
  of~(\ref{eq:fund}).  Of~the three parts the first is the longest.  Only
  in~(III) will the facts $\kappa,\tilde \kappa=(0,0,1)$ be used.

  \emph{Part I of Proof}.  What must first be elucidated are the relation
  between~$f,\tilde f$, and that between the exponents of the PE's that
  $f,\tilde f$ satisfy; and hence, the relation between the parameters of
  the two induced gDH systems.  It is not the case that $f(t)=\tilde
  f(\tilde t)$.  (The reader should glance at the
  identity~(\ref{eq:Psymbollifting}), which shows the P\nobreakdash-symbols
  of two PE's related by the quadratic transformation~${\bf 2}$, the
  solutions $f,\tilde f$ of which \emph{do} satisfy $f(t)=\tilde f(\tilde
  t)$.)  That $f(t),\tilde f(\tilde t)$ must differ is evident from the
  parametrizations $\tau=\tau(t)$, $\tilde\tau=\tilde\tau(\tilde t)$
  produced by Theorem~\ref{thm:integration}, which satisfy
  \begin{equation}
    \frac{{\rm d}\tau}{{\rm d}t} =  K^{-2}(t) f^{-1/\bar n}(t), \qquad
    \frac{{\rm d}\tilde\tau}{{\rm d}\tilde t} =  \tilde K^{-2}(\tilde t)
    \tilde f^{-1/\bar n}(t),
  \end{equation}
  for some chosen $\bar n\in\mathbb{P}^1\setminus\{0,\infty\}$ and
  functions $K,\tilde K$ determined by the chosen offset
  vectors~$\kappa,\tilde\kappa$.  If $t=R(\tilde t)$ and $\tau=\tilde \tau$
  then
  \begin{equation}
    \label{eq:ready}
    \tilde f(\tilde t)/f(t) = 
    \left[ K^2(t)/\tilde K^2(\tilde t) R'(\tilde t)\right]^{\bar n},
  \end{equation}
  the prime indicating differentiation with respect to~$\tilde t$.  It
  follows readily from~(\ref{eq:ready}) that if a point
  $t^*\in\mathbb{P}^1_t$ is mapped by~$R$ with multiplicity~$m$ from a
  point $\tilde t^*\in\mathbb{P}^1_{\tilde t}$, so that the local behavior
  of~$R$ is $(t-t^*)\sim\textrm{const}\times\allowbreak(\tilde t -\nobreak
  \tilde t^*)^m$, then the respective exponents $(\mu,\mu')$ and~$(\tilde
  \mu,\tilde \mu')$ are related not by $(\tilde \mu,\tilde
  \mu')=m(\mu,\mu')$ (as~is the case, say, in~(\ref{eq:Psymbollifting})),
  but rather by
  \begin{equation}
    \label{eq:multadd}
    \bigl[(\tilde\mu,\tilde\mu') + (2\tilde \kappa-1){\bar{n}}\,(1,1)\bigr] = m
     \bigl[(\mu,\mu') + (2\kappa-1){\bar{n}}\,(1,1)\bigr].
  \end{equation}
  Here, $\kappa,\tilde \kappa$ are specific to the points~$t^*,\tilde t^*$
  in the sense that $\kappa$~is $\kappa_i$ if $t^*=t_i$, and
  $\tilde\kappa$~is $\tilde\kappa_i$ if $\tilde t^*=\tilde t_i$; and
  otherwise they are zero.

  Equation~(\ref{eq:multadd}) takes on a simple form when written in~terms
  of \emph{offset} exponents
  $\nu\defeq \allowbreak\mu+\allowbreak(2
  \kappa-\nobreak1)\bar n$, etc., which were already introduced
  in~(\ref{eq:zetas}), in the statement of Theorem~\ref{thm:integration}.
  It says that
  \begin{equation}
    \label{eq:usemults}
    (\tilde\nu,\tilde\nu')=m(\nu,\nu')
  \end{equation}
  relates the (offset) exponents at any point $\tilde t=\tilde t^*$ to
  those at its image $t=t^*$.  Informally, when lifting a PE
  on~$\mathbb{P}_t^1$ that induces a gDH system (with independent
  variable~$\tau$) to a PE on~$\mathbb{P}_{\tilde t}^1$ that induces
  another gDH system (with the \emph{same} independent variable, i.e., with
  $\tilde\tau=\tau$), one should work in~terms of~$\nu$'s (offset
  exponents) rather than~$\mu$'s (unoffset ones).

  An example is the map~${\bf 2}$, i.e., $t=R(\tilde t)=4\tilde
  t/(1+\nobreak\tilde t)^2$.  Equation~(\ref{eq:usemults}) holding, the
  P\nobreakdash-symbol identity (\ref{eq:Psymbollifting}) must be replaced
  by the pair of P\nobreakdash-symbols
\begin{subequations}
\label{eq:Psymab}
\begin{sizealign}{\small}
\label{eq:Psyma}
  f(t) & = 
  \left\{
  \begin{array}{ccc}
    0 & 1 & \infty \\
    \hline
    \nu_1-(2\kappa_1-1){\bar{n}} & \nu_2-(2\kappa_2-1){\bar{n}} & 0-\left(2\kappa_3-\tfrac12\right){\bar{n}} \\
    \nu'_1-(2\kappa_1-1){\bar{n}} & \nu'_2-(2\kappa_2-1){\bar{n}} & \tfrac12-\left(2\kappa_3-\tfrac12\right){\bar{n}}
  \end{array}
  \right\}(t),\\
  \tilde f(\tilde t) &=
  \left\{
  \begin{array}{ccc}
    0  & 1 & \infty \\
    \hline
    \nu_1-(2\tilde\kappa_1-1){\bar{n}} & 2\nu_2-(2\tilde\kappa_2-1){\bar{n}} & \nu_1-(2\tilde\kappa_3-1){\bar{n}}\\
    \nu'_1-(2\tilde\kappa_1-1){\bar{n}} & 2\nu'_2-(2\tilde\kappa_2-1){\bar{n}} & \nu'_1-(2\tilde\kappa_3-1){\bar{n}}
  \end{array}
  \right\}(\tilde t),
\label{eq:Psymb}
\end{sizealign}
\end{subequations}
which are parametrized by the offset exponents $(\nu_1,\nobreak \nu_1';
\allowbreak \nu_2,\nobreak \nu_2'; \allowbreak \nu_3,\nobreak \nu_3')$,
together with $\kappa,\tilde \kappa$ and~$\bar n$.  Most of the
dependencies that one sees between exponents in
(\ref{eq:Psyma}),\allowbreak(\ref{eq:Psymb}) are attributable to the
unitalicized multiplicities in the branching schema $1+\nobreak1=2={\it2}$
of the map~$\bf2$, which imply
\begin{equation}
\label{eq:zetaprimeconstraints}
  (\tilde\nu_1,\tilde\nu_1')=(\nu_1,\nu'_1),\qquad
  (\tilde\nu_2,\tilde\nu_2')=2(\nu_2,\nu'_2),\qquad
  (\tilde\nu_3,\tilde\nu_3')=(\nu_1,\nu'_1).\qquad
\end{equation}
In~(\ref{eq:Psyma}), the exponents $(\mu_3,\mu_3')$ at~$t=t_3=\infty$,
which is doubly mapped from $\tilde t=-1$, come instead from the fact that
at any point $t^*\in\mathbb{P}^1_t$ mapped with multiplicity~$m$ from an
\emph{ordinary} point $\tilde t^*\in\mathbb{P}^1_{\tilde t}$ (with
exponents $(\tilde\mu,\tilde\mu')$ equal to $(0,1)$ by definition), the
offset exponents $(\nu,\nu')$ and exponents $(\mu,\mu')$ are given by
\begin{subequations}
\begin{align}
(\nu,\nu')&=(0,1/m)-(1/m){\bar{n}}\,(1,1),\\ 
(\mu,\mu')&=(0,1/m)-(2\kappa-1+1/m){\bar{n}}\,(1,1),
\label{eq:muunderordinary}
\end{align}
\end{subequations}
due to~(\ref{eq:usemults}).  It should be noted that
$\nu_1,\nobreak\nu_2,\nobreak\nu_3;
\allowbreak\nu_1',\nobreak\nu_2',\nobreak\nu_3'$ are always
constrained by the Fuchsian condition
\begin{equation}
\label{eq:nonfree}
  \sum_{i=1}^3 (\nu_i+\nu'_i)=1-2{\bar{n}},
\end{equation}
which follows from Fuchs's relation $\sum_{i=1}^3(\mu_i+\nobreak\mu_i')=1$.
In~(\ref{eq:Psymab}), each of $\nu_1,\nobreak\nu_1';
\allowbreak\nu_2,\nobreak\nu_2'$ can be viewed as a free parameter, $\bar
n$~being determined by~(\ref{eq:nonfree}).

It can now be seen how the gDH parameter vectors $( a_1, \nobreak a_2,
\nobreak a_3;\allowbreak b_1, \nobreak b_2, \nobreak b_3; c)$ and $(\tilde
a_1, \nobreak \tilde a_2, \nobreak \tilde a_3;\allowbreak \tilde b_1,
\nobreak \tilde b_2, \nobreak\tilde b_3;\tilde c)$ are related, subject to
constraints.  To begin with, the `upper' offset exponents $(\tilde\nu_1,
\nobreak \tilde\nu_2, \nobreak \tilde\nu_3;\allowbreak \tilde\nu_1',
\nobreak \tilde\nu_2', \nobreak \tilde\nu_3')$ will satisfy one or more
constraints.  (For instance, for the map~$\bf2$ these are
$\tilde\nu_3=\tilde\nu_1$ and $\tilde\nu_3'=\tilde\nu_1'$,
by~(\ref{eq:zetaprimeconstraints}).)  Let $(\tilde a_1, \nobreak \tilde
a_2, \nobreak \tilde a_3;\allowbreak \tilde b_1, \nobreak \tilde b_2,
\nobreak\tilde b_3;\tilde c)$ be computed from $(\tilde\nu_1, \nobreak
\tilde\nu_2, \nobreak \tilde\nu_3;\allowbreak \tilde\nu_1', \nobreak
\tilde\nu_2', \nobreak \tilde\nu_3')$, by the formulas
(\ref{eq:abtozetasn}) of Theorem~\ref{thm:integration}.  Then $(\tilde a_1,
\nobreak \tilde a_2, \nobreak \tilde a_3;\allowbreak \tilde b_1, \nobreak
\tilde b_2, \nobreak\tilde b_3;\tilde c)$ will also be constrained;
for~$\bf2$ the constraints turn~out to be that $\tilde a_3=\tilde a_1$,
$\tilde b_3=\tilde b_1$.  Let the `lower' offset exponents $(\nu_1,\nobreak
\nu_1'; \allowbreak \nu_2,\nobreak \nu_2'; \allowbreak \nu_3,\nobreak
\nu_3')$ be computed from $(\tilde\nu_1,\nobreak \tilde\nu_1'; \allowbreak
\tilde\nu_2,\nobreak \tilde\nu_2'; \allowbreak \tilde\nu_3,\nobreak
\tilde\nu_3')$ by~(\ref{eq:usemults}), using the unitalicized
multiplicities in the branching schema of~$R$, and let $( a_1, \nobreak
a_2, \nobreak a_3;\allowbreak b_1, \nobreak b_2, \nobreak b_3; c)$ be
computed from $(\nu_1,\nobreak \nu_1'; \allowbreak \nu_2,\nobreak \nu_2';
\allowbreak \nu_3,\nobreak \nu_3')$ by the formulas (\ref{eq:zetasntoab})
of Theorem~\ref{thm:integration}.  What results is an expression for the
gDH parameter vector $( a_1, \nobreak a_2, \nobreak a_3;\allowbreak b_1,
\nobreak b_2, \nobreak b_3; c)$, as a function of the (constrained) gDH
parameter vector $(\tilde a_1, \nobreak \tilde a_2, \nobreak \tilde
a_3;\allowbreak \tilde b_1, \nobreak \tilde b_2, \nobreak\tilde b_3;\tilde
c)$.

The parametric constraints and relations in Table~\ref{tab:restrictions}
were computed by this technique.  (Each line of the table was normalized by
taking $c=\tilde c$.)  In most cases the constraints on $(\tilde a_1,
\nobreak \tilde a_2, \nobreak \tilde a_3;\allowbreak \tilde b_1, \nobreak
\tilde b_2, \nobreak\tilde b_3;\tilde c)$ have an interpretation in~terms
of symmetry.  For instance, for each `${\bf c}$'~map the constraints are
that $\tilde a_1=\tilde a_2=\tilde a_3$ and $\tilde b_1=\tilde b_2=\tilde
b_3$, because the singular points $\tilde t=0,1,\infty$ are mapped by~$R$
to the same point ($t=0$), with equal multiplicities.

As was noted in~\S\,\ref{subsec:fromPEtogDH}, it is implicit in
Theorem~\ref{thm:integration} that the gDH system that it produces is
proper.  Here, this translates to a hidden assumption that the two gDH
systems being mapped between do \emph{not} satisfy (i) $\tilde c=0$
or~$c=0$, (ii) $2\tilde c-\nobreak\tilde b_1-\nobreak\tilde
b_2-\nobreak\tilde b_3=0$ or $2 c-\nobreak b_1-\nobreak b_2-\nobreak
b_3=0$, or (iii) $\tilde c-\nobreak\tilde a_1-\nobreak\tilde
a_2-\nobreak\tilde a_3=0$ or $c-\nobreak a_1-\nobreak a_2-\nobreak a_3=0$.
But by continuity of each gDH system in its parameters, this assumption can
safely be dropped.

\smallskip
\emph{Part II of Proof}.  Now that the relation between the two PE's, the
relation between their solutions $f(t),\tilde f(\tilde t)$, and that
between the two induced gDH systems $\dot x=Q(x)$, $\dot{\tilde x}=\tilde
Q(\tilde x)$ are all understood, the solution-preserving map $x=\Phi(\tilde
x)$ can be computed from the covering $t=R(\tilde t)$.  First, from
$f,\tilde f$, construct as functions of~$t,\tilde t$ the dependent
variables $x,\tilde x$ and independent variables $\tau,\tilde\tau$
(satisfying $\tau=\tilde\tau$), as in Theorem~\ref{thm:integration}.
Define $\rho,\tilde \rho$ by
\begin{alignat}{2}
\label{eq:rhodef}
  \rho^{-1}&=\mu_1+\mu_2+\mu_3 &\qquad\tilde\rho^{-1}&=\tilde\mu_1+\tilde\mu_2+\tilde\mu_3\\
&=\frac{2\,c-b_1-b_2-b_3}{c-a_1-a_2-a_3},&\qquad&=\frac{2\,\tilde c-\tilde b_1-\tilde b_2-\tilde b_3}{\tilde c-\tilde a_1-\tilde a_2-\tilde a_3},\nonumber
\end{alignat}
the lower expressions following from the formulas in that theorem; and note
that $\rho/\tilde\rho=d$, as $R$~is a degree\nobreakdash-$d$ covering.
(The case when $\rho$ or~$\tilde\rho$ is undefined can be handled by
continuity, as above.)  As $(t_1,t_2,t_3)$ and $(\tilde t_1,\tilde
t_2,\tilde t_3)$ equal $(0,1,\infty)$, the formula (\ref{eq:secondclaim})
of Lemma~\ref{lem:usedlater} becomes
\begin{equation}
\label{eq:320}
  t=-\left(\frac{x_2-x_3}{x_1-x_2}\right) = -\,\frac{y_1}{y_3}, \qquad
  \tilde t=-\left(\frac{\tilde x_2-\tilde x_3}{\tilde x_1-\tilde
    x_2}\right) = -\,\frac{\tilde y_1}{\tilde y_3},
\end{equation}
it being assumed that in the gDH solutions $x=x(\tau)$ and $\tilde x=\tilde
x(\tilde\tau)$ are noncoincident.  (The case of coinciding components can
be handled by another continuity argument.)  Also,
Eq.~(\ref{eq:fifthclaim}) of the lemma yields
\begin{equation}
\label{eq:321}
  \dot{\tilde t} = c\,{\bar{n}}\,\tilde\rho
  \left[\frac{(\tilde x_2-\tilde x_3)(\tilde x_3-\tilde x_1)}{\tilde
      x_1-\tilde x_2}\right]
  = c\,{\bar{n}}\,\tilde\rho
  \left(\frac{\tilde y_1\tilde y_2}{\tilde y_3}\right).
\end{equation}
As $\dot t=\dot{\tilde t}R'(\tilde t)$, it then follows from
Eq.~(\ref{eq:firstclaim}) of the lemma that
\begin{align}
  y_k=x_i-x_j &= c^{-1}\bar n^{-1}\rho^{-1}\left[\frac{\dot t}{t-t_i}-\frac{\dot t}{t-t_j}\right]\nonumber\\
  &=c^{-1}\bar n^{-1}\rho^{-1}\left[\frac1{t-t_i}-\frac1{t-t_j}\right]\dot{\tilde t}\,R'(\tilde  t)\\
  &=d^{-1}\left[\frac1{R(\tilde
      t)-t_i}-\frac1{R(\tilde t)-t_j}\right]\left(\frac{\tilde
    y_1\tilde y_2}{\tilde y_3}\right)R'(\tilde t).\nonumber
\end{align}
But $R(\tilde t)=-{{P}}_1(\tilde t)/{{P}}_3(\tilde t)$ with
${{P}}_1+{{P}}_2+{{P}}_3=0$, and $(t_1,t_2,t_3)=(0,1,\infty)$.  Using
(\ref{eq:321}) and~(\ref{eq:320}) to rewrite $\dot{\tilde t}$ in~terms of
$\tilde y_3$ and~$\tilde t$, one obtains rational expressions for
$y_1,y_2,y_3$ in~terms of $\tilde y_1,\tilde y_2,\tilde y_3$, which by
examination are
\begin{equation}
\label{eq:prefundcor}
  (y_1,y_2,y_3)/\tilde y_3 = d^{-1}\Bigl[\tilde
  t(1-\tilde t)R'(\tilde t)\cdot\bigl({{P}}_2^{-1}{{P}}_3,
  {{P}}_1^{-1}{{P}}_3, {{P}}_1^{-1}{{P}}_2^{-1}{{P}}_3^2\bigr)\Bigr]\big|_{\tilde t=-\tilde y_1/\tilde y_3}.
\end{equation}
Further rearrangement using $R' =
({{P}}_3'{{P}}_1-{{P}}_1'{{P}}_3)/{{P}}_3^2$ and
${{P}}_1'+{{P}}_2'+{{P}}_3'=0$ yields
  \begin{equation}
    \label{eq:fundcor2}
    y_k = d^{-1}\,\tilde y_3\Bigl\{\tilde t(1-\tilde t)\left[({{P}}'_i/{{P}}_i)-({{P}}'_j/{{P}}_j)\right](\tilde t)\Bigr\}\big|_{\tilde t=-\tilde y_1/\tilde y_3}
  \end{equation}
for $k=1,2,3$, which was claimed as Eq.~(\ref{eq:fundcor}) above.

The expressions for $\Sigma_2,\Sigma_3,\Sigma_6$ in Table~\ref{tab:sigmas}
can be computed from Eq.~(\ref{eq:fundcor2}), which gives each
$y_k=x_i-\nobreak x_j$ as a rational function of $\tilde y_1,\tilde
y_2,\tilde y_3$, and hence of $\tilde x_1,\tilde x_2,\tilde x_3$.
Actually, it is more illuminating to use~(\ref{eq:prefundcor}).  Writing
$y_1,y_2,y_3$ in~terms of $\Sigma_2,\Sigma_3,\Sigma_6$, and using
(\ref{eq:prefundcor}) to solve for the latter, yields
\begin{subequations}
\label{eq:Sigma236}
\begin{align}
  \Sigma_2 / \tilde y_3^2 &=
  \left\{\left[\tilde t(\tilde t-1)R'\right]^2{{P}}_1^{-2}{{P}}_2^{-1}{{P}}_3^3\right\}|_{\tilde t=-\tilde
    y_1/\tilde y_3}, \label{eq:Sigma2}\\
  \Sigma_3 / \tilde y_3^3 &=
  \left\{\left[\tilde t(\tilde t-1)R'\right]^3{{P}}_1^{-3}{{P}}_2^{-2}{{P}}_3^5\right\}|_{\tilde t=-\tilde
    y_1/\tilde y_3}, \label{eq:Sigma3}\\
  \Sigma_6 / \tilde y_3^6 &=
  \left\{\left[\tilde t(\tilde t-1)R'\right]^6{{P}}_1^{-5}{{P}}_2^{-4}{{P}}_3^9\right\}|_{\tilde t=-\tilde y_1/\tilde y_3}. \label{eq:Sigma6}
\end{align}
\end{subequations}
By considering the order of vanishing of each right side
in~(\ref{eq:Sigma236}) at each zero of ${{P}}_1$, ${{P}}_2$, or~${{P}}_3$,
it is not difficult to see that irrespective of the choice of classical
PE\nobreakdash-lifting map $t=R(\tilde t)$, each of
$\Sigma_2,\Sigma_3,\Sigma_6$ must be \emph{polynomial} in $\tilde
x_1,\tilde x_2,\tilde x_3$, and not merely rational.  (This is essentially
because in each case, $R^{-1}\{0\}\subset\{0,1,\infty\}$.)  Polynomiality
also follows by direct computation.

\smallskip
\emph{Part III of Proof}.  What remains to be derived is the strengthening
(\ref{eq:fund}) of~(\ref{eq:fundcor2}), i.e., the useful formula
  \begin{equation}
    \label{eq:fund2}
    x_i = \tilde x_3 + d^{-1}\,\tilde y_3\left[\tilde t(1-\tilde t)({{P}}'_i/{{P}}_i)(\tilde t)\right]\big|_{\tilde t=-\tilde
      y_1/\tilde y_3},
  \end{equation}
from which the remaining polynomials~$\Sigma_1$ of Table~\ref{tab:sigmas}
can be computed.  Equation~(\ref{eq:fund2}) is a corollary of an even
simpler formula,
\begin{equation}
  \label{eq:simplerfund}
  f_i^{\rho}(t=R(\tilde t))/\tilde f_3^\rho(\tilde t) = C\times P_i(\tilde t)|_{\tilde t=-\tilde y_1/\tilde
    y_3},
 \end{equation}
where as in Theorem~\ref{thm:integration}, $f_i=\Delta_if$ and $\tilde
f_i=\tilde \Delta_i\tilde f$, and where the prefactor $C\neq0$ is $\tilde
t$-independent.  To see that (\ref{eq:fund2}) follows
from~(\ref{eq:simplerfund}), recall (see Theorem~\ref{thm:integration})
that $x_i,\tilde x_i$ are defined as logarithmic derivatives with respect
to~$\tau$, i.e.,
\begin{equation}
   x_i=c^{-1}{\bar{n}^{-1}}\,\dot{f}_i/{f}_i,\qquad
  \tilde x_i=c^{-1}{\bar{n}^{-1}}\,\dot{\tilde f}_i/{\tilde f}_i.
\end{equation}
By taking the logarithmic derivative of both sides
of~(\ref{eq:simplerfund}), and using (\ref{eq:321}) and~(\ref{eq:320}) to
rewrite $\dot{\tilde t}$ in~terms of $\tilde y_3$ and~$\tilde t$, one
obtains~(\ref{eq:fund2}).

To prove that Eq.~(\ref{eq:simplerfund}) holds for each classical
PE\nobreakdash-lifting map, reason as follows.  It suffices to prove the
$i=3$ case, i.e., that
\begin{equation}
\label{eq:consistent1}
\left[f_3(t)/\tilde f_3(\tilde t)\right]^\rho = C\times P_3(\tilde t).
\end{equation}
But the relation between $f,\tilde f$ is known; see Eq.~(\ref{eq:ready})
above.  The case when $\kappa,\tilde\kappa=(0,0,1)$ is especially simple,
as then $K^2,\tilde K^2=1$ (see the second remark after
Theorem~\ref{thm:integration}).  The relation becomes $\tilde
f/f=\left[R'(\tilde t)\right]^{\bar n}$, so that
\begin{align}
\label{eq:consistent2}
\left[f_3(t)/\tilde f_3(\tilde t)\right]^\rho &= \left[R'(\tilde
  t)\right]^{-{\bar{n}}\rho} \left[ \frac{\Delta_3(t)}{\tilde
    \Delta_3(\tilde t)}\right]^\rho\\ &=\left[R'(\tilde
  t)\right]^{-{\bar{n}}\rho} \left[
  \frac{(-t)^{-\mu_1}(t-1)^{-\mu_2}}{(-\tilde t)^{-\tilde\mu_1}(\tilde
    t-1)^{-\tilde\mu_2}}\right]^\rho \nonumber\\ &=\left[R'(\tilde
  t)\right]^{-{\bar{n}}\rho} \left[
  \frac{(P_1/P_3)^{-\mu_1}(P_2/P_3)^{-\mu_2}}{(-\tilde
    t)^{-\tilde\mu_1}(\tilde t-1)^{-\tilde\mu_2}}\right]^\rho. \nonumber
\end{align}
The expressions for $\Delta_3,\tilde\Delta_3$ are taken
from~(\ref{eq:Deltas}), and use has been made of the fact that $t=R(\tilde
t)=-P_1(\tilde t)/P_3(\tilde t)$ with $P_1+P_2+P_3=0$.  After a bit more
manipulation, making use of
$\rho^{-1}=\allowbreak\mu_1+\nobreak\mu_2+\nobreak\mu_3$
(see~(\ref{eq:rhodef}) above), one finds that (\ref{eq:consistent1}) will
follow from~(\ref{eq:consistent2}) if
\begin{equation}
\label{eq:lambdadef}
  \Lambda(\tilde t)\defeq (-\tilde t)^{\tilde\mu_1}(\tilde t-1)^{\tilde\mu_2}\,
P_1^{-\mu_1}P_2^{-\mu_2}P_3^{-\mu_3}\left[R'=(-P_1/P_3)'=(P_2/P_3)'\right]^{\bar n}
\end{equation}
is a nonzero constant function of~$\tilde t$.

To prove this last, it suffices to show that at all~$\tilde t$ (or~merely,
all finite~$\tilde t$), $\Lambda$~has zero order of vanishing.  But, the
only points on~$\mathbb{P}^1_{\tilde t}$ at which it can have nonzero order
of vanishing are those in $R^{-1}(\{0,1,\infty\})$, i.e., (i)~the singular
points $\tilde t=0,1,\infty$, and (ii)~ordinary (nonsingular) points mapped
by~$R$ to $t=0,1,\infty$.

(i)~As an example, consider $\tilde t=0$, which by convention is mapped
by~$R$ to $t=0$.  Suppose this is with multiplicity~$m$, i.e., $t=R(\tilde
t)\sim\textrm{const}\times\nobreak\tilde t^m$, so that $P_1(\tilde
t)\sim\textrm{const}\times\nobreak\tilde t^m$ and $P_2(0),P_3(0)\neq0$.
By~(\ref{eq:lambdadef}), the order of vanishing of~$\Lambda$ at $\tilde
t=0$ will be $\tilde\mu_1-m\mu_1+\allowbreak(m-\nobreak1)\bar n$.  But by
the remarks on lifting of exponents in Part~I of this proof, the lifted
exponent~$\tilde\mu_1$ at~$\tilde t=0$ will equal
$m\mu_1-\allowbreak(m-\nobreak1)\bar n$.  (See~(\ref{eq:multadd});
$\kappa_1=0$ and $\tilde\kappa_1=0$ are used here.)  Hence the order of
vanishing at $\tilde t=0$ is zero.  The singular point $\tilde t=1$ is
handled identically.

(ii)~As an example, consider $\tilde t=\tilde t^*$, some ordinary point
mapped by~$R$ to $t=\infty$.  Being an ordinary (nonsingular) point, it has
exponents $(\tilde\mu,\tilde\mu')=(0,1)$.  Suppose the mapping is with
multiplicity~$m$, so that $P_3(\tilde
t)\sim\textrm{const}\times\nobreak(\tilde t-\tilde t^*)^m$ and $P_2(\tilde
t^*),P_3(\tilde t^*)\neq0$.  Then by the above remarks on lifting of
exponents, the first exponent~$\mu_3$ at the mapped point $t=\infty$ must
be $0-\allowbreak(2\cdot1-\nobreak1+\nobreak1/m)\bar n$.
(See~(\ref{eq:muunderordinary}); $\kappa_3=1$ is used here.)  It follows
immediately from~(\ref{eq:lambdadef}) that $\Lambda$~has zero order of
vanishing at~$\tilde t=\tilde t^*$.  Ordinary points
on~$\mathbb{P}^1_{\tilde t}$ that are mapped by~$R$ to the other two
singular points, i.e., to $t=1,\infty$, are handled similarly.

\smallskip
Now that Eq.~(\ref{eq:simplerfund}) and its corollary~(\ref{eq:fund2}) have
been derived, one can readily compute $x_1\eqdef\Sigma_1/d$ as a function
of $\tilde x_1,\tilde x_2,\tilde x_3$ for each classical
PE\nobreakdash-lifting map $t=R(\tilde t)=-(P_1/P_3)(\tilde t)$.  In each
case $\Sigma_1$ turns~out to be linear in $\tilde x_1,\tilde x_2,\tilde
x_3$.  The resulting linear polynomials are listed in
Table~\ref{tab:sigmas}.
\end{proof}

\begin{remark*}
  There is an alternative proof of Theorem~\ref{thm:gDHsystems}, which
  would go as follows.  (I)~Each covering $t=R(\tilde t)$ induces a lifting
  of a gSE of the form~(\ref{eq:gSE}) satisfied by $t=t(\tau)$ to one for
  $\tilde t=\tilde t(\tau)$, and an accompanying lifting of (suitably
  constrained) parameter vectors, $(\nu_1,\nobreak \nu_1'; \allowbreak
  \nu_2,\nobreak \nu_2'; \allowbreak \nu_3,\nobreak \nu_3';\bar
  n)\leftarrow\allowbreak (\tilde\nu_1,\nobreak \tilde\nu_1'; \allowbreak
  \tilde\nu_2,\nobreak \tilde\nu_2'; \allowbreak \tilde\nu_3,\nobreak
  \tilde\nu_3';\bar n)$.  (II)~By Theorem~\ref{thm:gDHbasecoin}, each such
  yields a rational solution-preserving map of (suitably constrained)
  proper gDH systems.  Details are left to the reader.
\end{remark*}

One should note that for each classical PE\nobreakdash-lifting map~$R$, the
coefficients of the polynomial~$\Sigma_1$ in Table~\ref{tab:sigmas}, which
is linear in $\tilde x_1,\tilde x_2,\tilde x_3$, sum to the degree~$d$ of
the map.  In fact the coefficients are the multiplicities with which
$\tilde t=0,1,\infty$ are taken to~$t=0$ by the map (cf.~the branching
schemata of Table~\ref{tab:coveringmapsclassical}).

In the cases ${\bf2},{\bf3}_{\bf c},{\bf6}_{\bf c}$, in which the field
extension $\mathbb{C}(\tilde t)/\mathbb{C}(t)$ is Galois, there is a simple
interpretation of the polynomials $\Sigma_1,\Sigma_2,\Sigma_3$.  On the
level of the rational solution-preserving map $(\tilde x_1,\tilde
x_2,\tilde x_3)\mapsto\allowbreak(x_1,x_2,x_3)$, the respective Galois
groups $\mathfrak{Z}_2,\mathfrak{Z}_3,\mathfrak{S}_2$ comprise interchanges
$\tilde x_1\leftrightarrow \tilde x_3$, cyclic permutations $\tilde x_1\to
\tilde x_2\to \tilde x_3\to \tilde x_1$, and arbitrary permutations of
$\tilde x_1,\tilde x_2,\tilde x_3$.  In each case
$\Sigma_1,\Sigma_2,\Sigma_3$ form an algebraic basis for an associated ring
of polynomial invariants.  For instance, in the case~${\bf 6}_{\bf c}$ they
generate all symmetric polynomials in $\tilde x_1,\tilde x_2,\tilde x_3$.
This interpretation of $\Sigma_1,\Sigma_2,\Sigma_3$ has an extension to the
non-Galois cases (see~\cite{Harnad2000}).

The solution-preserving maps between gDH systems induced by the
semiclassical and nonclassical PE\nobreakdash-lifting maps ${\bf12}_{\bf
  bq},{\bf12}_{\bf c}$ and ${\bf10},{\bf24}_{\bf c}$ of Tables
\ref{tab:coveringmapsclassical} and~\ref{tab:coveringmapsnonclassical} are
rather complicated and will not be given in their entirety.  The following
will suffice.

\begin{theorem}
  To each of\/ ${\bf12}_{\bf bq},{\bf12}_{\bf c}$ and\/
  ${\bf10},{\bf24}_{\bf c}$ {\rm(}of degree\/~$d$\,{\rm)} there is
  associated a rational map\/ $x=\Phi(\tilde x)$ between gDH systems\/
  $\dot{\tilde x}=\tilde Q(\tilde x)$ and\/ $\dot x=Q(x)$ of the form\/
  {\rm(\ref{eq:gDH})}.  It can be viewed as a map\/
  $\Phi\colon\mathbb{P}^2_{\tilde x}\to\mathbb{P}^2_{x}$, and is expressed
  as in Theorem\/~{\rm\ref{thm:gDHsystems}} but with\/
  $\Sigma_1,\Sigma_2,\Sigma_3,\Sigma_6$ rational rather than polynomial,
  and with
  \begin{displaymath}
    x_1 = \Sigma_1/d = \frac1d\left[\hat\Sigma_1 + \frac{d-d_0}{\deg\Upsilon}\,
\frac{\sum_{i=1}^3 \tilde
        x_j\tilde x_k\, \partial\Upsilon/\partial\tilde x_i}{\Upsilon}\right],
  \end{displaymath}
  in which\/ $j,k$ are the elements of\/ $1,2,3$ other than\/~$i$.  Here\/
  $\hat\Sigma_1,\Upsilon$ are certain homogeneous polynomials in\/ $\tilde
  x_1,\tilde x_2,\tilde x_3$ that are listed in
  Table\/~{\rm\ref{tab:upsilon}}, the first being linear, and\/ $d_0<d$ is
  the sum of the coefficients of\/~$\hat\Sigma_1$.  A~sufficient condition
  for\/ $\Phi$ to be solution-preserving from the system\/ $\dot{\tilde
    x}=\tilde Q(\tilde x)$, i.e., ${\rm gDH}(\tilde a_1, \nobreak\tilde
  a_2, \nobreak\tilde a_3; \allowbreak\tilde b_1,\nobreak\tilde
  b_2,\nobreak\tilde b_3;\nobreak\tilde c)$, to the system\/ $\dot x=Q(x)$,
  i.e., to\/ ${\rm gDH}(a_1,\nobreak a_2,\nobreak a_3;\allowbreak
  b_1,\nobreak b_2,\nobreak b_3;\nobreak c)$, is that the respective
  parameter vectors be restricted and related as specified in
  Table\/~{\rm\ref{tab:restrictions}}.
  \label{thm:gDHsystems2}
\end{theorem}

\begin{proof}
  \emph{Mutatis mutandis}, the same as that of
  Theorem~\ref{thm:gDHsystems}.  In each of the cases ${\bf12}_{\bf
    bq},{\bf12}_{\bf c}$ and ${\bf10},{\bf24}_{\bf c}$, the useful
  formula~(\ref{eq:fund}) can be used to compute $\Phi$ from~$R$, and in
  particular the rational function $x_1=\Sigma_1/d$ of $\tilde x_1,\tilde
  x_2,\tilde x_3$.
\end{proof}

In fact, the coefficients of each linear polynomial~$\hat\Sigma_1$ of
Table~\ref{tab:upsilon} are the multiplicities with which $\tilde
t=0,1,\infty$ are taken by the respective covering map~$R$ to $t=0$, much
as with the polynomials~$\Sigma_1$ of Table~\ref{tab:sigmas}.  But for
these semiclassical and nonclassical covering maps, a~new feature enters:
one or more ordinary points on~$\mathbb{P}^1_{\tilde t}$ are also mapped to
$t=0$.  For instance, for~${\bf10}$ one has the factor $32-81\tilde t$
in~${{P}}_1$ (see Table~\ref{tab:coveringmapsnonclassical}), which is zero
at the ordinary point $\tilde t=32/81$.  This factor is responsible for
$\Upsilon$ equaling $32\tilde x_1 + 49\tilde x_2 - 81\tilde x_3$ in
Table~\ref{tab:upsilon}.  The polynomial~$\Upsilon$ is of degree~1 for
${\bf12}_{\bf bq},{\bf12}_{\bf c}$ and ${\bf10}$, but of degree~3
for~${\bf24}_{\bf c}$, due to $R^{-1}(0)\subset\mathbb{P}^1_{\tilde t}$
comprising three ordinary points, i.e.\ the roots of $1+5\tilde t - 8\tilde
t^2 + \tilde t^3$, rather than only one.

\begin{table}[t]
  \begin{center}
    \begin{tabular}{|l||l|l|}
      \hline
      \hfil covering map & $\hat\Sigma_1$ & $\Upsilon$ \\
      \hline  
      \hline  
      ${\bf12}_{\bf bq}\sim{\bf6}\circ{\bf2}$ & $\tilde x_1+2\tilde x_2+\tilde x_3$ & $\tilde x_1-2\tilde x_2+\tilde x_3$ \\
      $\quad{}\sim{\bf3}\circ{\bf2}\circ{\bf2}$ & & \\
      $\quad{}\sim{\bf3}\circ{\bf4}_{\bf bq}$ & & \\
      \hline  
      ${\bf12}_{\bf c}\sim{\bf4}\circ{\bf3}_{\bf c}$ & $\tilde x_1+\tilde x_2+\tilde x_3$ &
      $\tilde x_1+\omega \tilde x_2+\bar\omega \tilde x_3 \eqdef z$ \\
      \hline  
      \hline
      ${\bf10}$ & $\tilde x_1 + 2\tilde x_3$  & $32\tilde x_1 + 49\tilde x_2 - 81\tilde x_3$  \\
      \hline  
      ${\bf24}_{\bf c}\sim{\bf8}\circ{\bf3}_{\bf c}$ & $\tilde x_1+\tilde x_2+\tilde x_3$ & 
      $(\tilde x_1^3 + \tilde x_2^3 + \tilde x_3^3) + 5(\tilde x_1^2\tilde x_2 + \tilde x_2^2\tilde x_3 + \tilde x_3^2\tilde x_1)$ \\
      & & $\quad{}-8(\tilde x_1\tilde x_2^2 + \tilde x_2\tilde x_3^2 + \tilde x_3\tilde x_1^2)+6\tilde x_1\tilde x_2\tilde x_3$   \\
      \hline  
    \end{tabular}
    \bigskip
  \caption{Polynomials $\hat\Sigma_1,\Upsilon$ in the gDH transformations
    $x=\Phi(\tilde x)$ derived from semiclassical and nonclassical
    ${}_2F_1$ transformations.}
  \label{tab:upsilon}  
  \end{center}
\end{table}

\begin{figure}[tb]
\centering\includegraphics[width=5in]{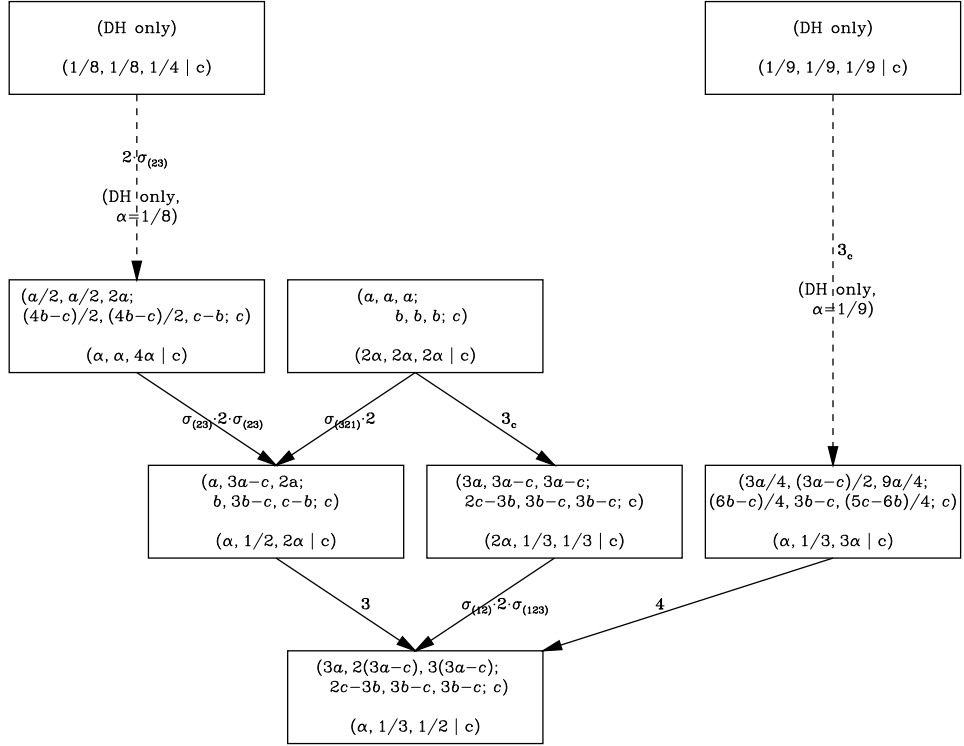}
\medskip
\caption{A directed graph of gDH systems with free parameters $(a,b,c)$,
  which are related by solution-preserving maps and specialize to proper DH
  systems with free parameters $\alpha$ and~$c$.}
\medskip
\label{fig:only}
\end{figure}

Another observation is that up~to scaling by $\tilde c=c$, in the cases
${\bf12}_{\bf bq},{\bf12}_{\bf c}$ and ${\bf10},{\bf24}_{\bf c}$ the
parameter vectors of the two gDH systems $\dot x=Q(x)$, $\dot{\tilde
  x}=\tilde Q(\tilde x)$ are completely specified.  (See
Table~\ref{tab:restrictions}.)  In each case $\tilde b_1=\tilde b_2=\tilde
b_3=\tilde c/2$ and $b_1=b_2=b_3=c/2$, so the solution-preserving maps of
Theorem~\ref{thm:gDHsystems2} are maps between DH systems, \emph{only}.
The correspondences
$(\alpha_1,\alpha_2,\alpha_3)\leftarrow(\tilde\alpha_1,\tilde\alpha_2,\tilde\alpha_3)$
between the angular parameters of the DH systems, if proper, have already
appeared in Tables \ref{tab:coveringmapsclassical}
and~\ref{tab:coveringmapsnonclassical}.  For instance, ${\bf24}_{\bf
  c}$~yields a rational solution-preserving map from ${\rm
  DH}\left(\frac17,\frac17,\frac17\right)$ to ${\rm
  DH}\left(\frac17,\frac13,\frac12\right)$.

\smallskip

The relationships among the solution-preserving maps of Theorems
\ref{thm:gDHsystems} and~\ref{thm:gDHsystems2} are displayed in
Figure~\ref{fig:only}, a composition graph.  It reveals how a single gDH
system (the `root' at the bottom, parametrized by~$a,b,c$) can give rise to
many others, by being lifted along the classical covering maps.  Each node
is labeled by the parameter vector $(\alpha_1,\alpha_2,\alpha_3\,|\,c)$ of
a proper DH system, and in all but two cases by the parameter vector
$(a_1,\nobreak a_2,\nobreak a_3;\allowbreak b_1,\nobreak b_2,\nobreak
b_3;c)$ of a gDH system that specializes to the proper DH system if the
root is constrained by $b=c/2\neq0$ and $c-\nobreak3a\neq0$, i.e., if the
root is specialized to ${\rm DH}(\alpha,\frac13,\frac12\,|\,c)$.  Each
directed edge of the graph is a solution-preserving map based on
${\bf2},{\bf3},{\bf4}$, or~${\bf3}_{\bf c}$.  The map may include M\"obius
transformations that permute singular points, and hence solution
components.

For each node in the graph, the DH parameters
$(\alpha_1,\alpha_2,\alpha_3)$ were taken from
Table~\ref{tab:coveringmapsclassical}, and the gDH parameter vector
$(a_1,\nobreak a_2,\nobreak a_3;\allowbreak b_1,\nobreak b_2,\nobreak
b_3;c)$ if~any was computed from Table~\ref{tab:restrictions}.  For
instance, the node atop the `diamond' is the ${\bf6}_{\bf c}$ node, for
which $(a_1,\nobreak a_2,\nobreak a_3;\allowbreak b_1,\nobreak b_2,\nobreak
b_3;c)$ is $(a,\nobreak a,\nobreak a;\allowbreak b,\nobreak b,\nobreak
b;c)$.  When $b=c/2\neq0$ and $c-\nobreak3a\neq0$, the gDH system with this
parameter vector specializes to ${\rm DH}(2\alpha,2\alpha,2\alpha\,|\,c)$,
where (cf.~(\ref{eq:computealphas})) the angular parameter~$\alpha$ is
defined by $2\alpha=-a/(c-\nobreak3a)$ and must satisfy
$\alpha\neq\frac16$ if the DH system is to be proper.  The left and right
sides of the diamond come from ${\bf 6}_{\bf
  c}\sim{\bf3}\circ{\bf2}\sim{\bf2}\circ{\bf3}_{\bf c}$, which with
M\"obius transformations included is
\begin{equation}
\label{eq:permutesps}
{\bf 6}_{\bf c}={\bf3}\circ\sigma_{(321)}\circ{\bf2}=\sigma_{(12)}\circ{\bf2}\circ\sigma_{(123)}\circ{\bf3}_{\bf c}.
\end{equation}
Here $\sigma_{(321)}$, $\sigma_{(12)}$, $\sigma_{(123)}$ permute the
singular points $0,1,\infty$, and hence the gDH components.

The upper left, resp.\ right nodes are the ${\bf12}_{\bf bq},{\bf12}_{\bf
  c}$ nodes, and the ${\bf12}_{\bf bq},{\bf12}_{\bf c}$ maps are functional
compositions leading down from them to the root.  Those two nodes are
`DH~only,' as indicated, and moreover are present only if $\alpha=\frac18$,
resp.\ $\alpha=\frac19$.  The decompositions ${\bf12}_{\bf
  bq}\sim{\bf3}\circ{\bf2}\circ{\bf2}$ and ${\bf12}_{\bf
  c}\sim{\bf4}\circ{\bf3}_{\bf c}$ can be seen.

Nodes $(\alpha_1,\alpha_2,\alpha_3\,|\,c)$ coming from the nonclassical
coverings ${\bf10},{\bf24}_{\bf c}$ could be added, but are omitted due to
lack of space.  The ${\bf10},{\bf24}_{\bf c}$ nodes are DH-only and require
$\alpha=\frac17$.  Their angular parameters
$(\alpha_1,\alpha_2,\alpha_3)=\allowbreak(\frac17,\frac13,\frac27)$ and
$(\frac17,\frac17,\frac17)$ were given in
Table~\ref{tab:coveringmapsnonclassical}.  From each, a directed edge
extends to the root node
$(\alpha,\frac13,\frac12\,|\,c)=\allowbreak(\frac17,\frac13,\frac12\,|\,c)$.

It was mentioned in~\S\,\ref{subsec:liftings} that there are five
nonclassical coverings that are not defined over $\mathbb{Q}$, and will not
be considered in any detail.  They can be denoted by
${\bf6}',{\bf8},\allowbreak{\bf9},\allowbreak{\bf10'},{\bf18}$, after their
degrees.  The final four of these add DH-only nodes to
Figure~\ref{fig:only}, which are present only if
$\alpha=\frac17,\frac17,\frac18,\frac17$, respectively.  The angular
parameters of these nodes are $(\frac17,\frac13,\frac13)$,
$(\frac17,\frac12,\frac17)$, $(\frac18,\frac13,\frac18)$,
$(\frac17,\frac17,\frac27)$.  From each, an edge extends to the root node
$(\alpha,\frac13,\frac12\,|\,c)$.  The ${\bf6}'$~map is anomalous: it maps
${\rm DH}(\frac14,\frac15,\frac14\,|\,c)$ to ${\rm
  DH}(\frac14,\frac15,\frac12\,|\,c)$, and does not fit into the framework
of the figure.

\subsection{Maps between non-gDH HQDS's}

The focus thus far in \S\,\ref{sec:transformations} has been on gDH
systems, and on the consequences of the hypergeometric integration scheme
implicit in Theorem~\ref{thm:integration} for the existence of rational
solution-preserving maps~$\Phi$ between them.  But the same maps~$\Phi$,
applied to certain HQDS's that are not of the gDH form, can yield images
that are at~least HQDS's.

Suppose that $x=\Phi(\tilde x)$ is solution-preserving from a gDH system
$\dot{\tilde x}=\tilde Q(\tilde x)$ to a HQDS $\dot{x}=Q(x)$.  Suppose also
that a HQDS $\dot{\tilde x}'= \tilde Q'(\tilde x')$, not necessarily a gDH
system, is linearly equivalent to $\dot{\tilde x} = \tilde Q(\tilde x)$,
i.e., that there is some $\tilde T\in{\it GL}(3,\mathbb{C})$ such that
$\tilde x'=\tilde T\tilde x$ and $\tilde x=\tilde T^{-1}\tilde x'$, by
which is meant that the quadratic vector fields $\tilde Q',\tilde Q$
satisfy $\tilde Q'=\tilde T\circ \tilde Q\circ\tilde T^{-1}$.  The
corresponding non-associative algebras
$\tilde{\mathfrak{A}}',\tilde{\mathfrak{A}}$ will then be isomorphic.
Clearly $x'=(\Phi\circ T^{-1})(\tilde x')$ will be solution-preserving from
$\dot{\tilde x}'=\tilde Q'(\tilde x')$ to $\dot{x}=Q(x)$.  Moreover if
$\Phi$~satisfies $\Phi=T\circ\Phi\circ \tilde T^{-1}$ for some $T,\tilde
T\in{\it GL}(3,\mathbb{C})$, then $\Phi$~itself will map the HQDS
$\dot{\tilde x}'=\tilde Q'(\tilde x')$ to a HQDS $\dot{x}'=Q'(x')$, the
vector field~$Q'$ being defined by $Q' = T\circ Q\circ T^{-1}$.

The following two theorems give examples of non-gDH HQDS's on which the
maps~$\Phi$ coming from the coverings $\bf2$, ${\bf3}_{\bf c}$ and~${\bf
  6}_{\bf c}$ are solution-preserving to other HQDS's.  In some cases this
is because they are linearly equivalent to the gDH systems of
Theorem~\ref{thm:gDHsystems} by $\tilde x'=\tilde T\tilde x$ for
some~$\tilde T$, and $\Phi$~satisfies $\Phi=T\circ\Phi\circ \tilde T^{-1}$
for some~$T$.  The map $x=\Phi(\tilde x)$ coming from~$\bf2$ was given
in~(\ref{eq:quadraticliftmap}), and the `cyclic' ones coming from ${\bf
  3}_{\bf c},{\bf6}_{\bf c}$ can be read~off from Table~\ref{tab:sigmas}.
Explicitly, the ${\bf 3}_{\bf c}$ map is
\begin{alignat}{3}
  \label{eq:3cmap}
  x_1 &= \frac13S_1,&\qquad x_2&=\frac13\left(S_1-\frac{z^2}{\bar z}\right), &\quad
  x_3&=\frac13\left(S_1-\frac{\bar z^2}{z}\right),
  \intertext{and the ${\bf 6}_{\bf c}$ map is}
  \label{eq:6cmap}
  x_1&=\frac13S_1, &\qquad
  x_2&= \frac13\left(S_1 - \frac{z^3+\bar z^3}{2\,z\bar z}  \right), &\quad
  x_3&= \frac13\left(S_1- \frac{2(z\bar z)^2}{z^3+\bar z^3}\right),\\
  &  &\quad
  &= \frac{9\,S_3-S_1S_2}{2(3\,S_2-S_1^2)}, &\quad
  &= \frac{9\,S_1S_3 - 6\,S_2^2 + S_1^2S_2}{27\,S_3-9\,S_1S_2+2\,S_1^3}.
  \nonumber
\end{alignat}
Here
\begin{equation}
S_1\defeq\tilde x_1+\tilde x_2+\tilde x_3,\qquad S_2\defeq\tilde x_1\tilde
x_2+\tilde x_2\tilde x_3+\tilde x_3\tilde x_1,\qquad S_2\defeq\tilde x_1\tilde
x_2\tilde x_3
\end{equation}
are the elementary symmetric polynomials in $\tilde x_1,\tilde x_2,\tilde
x_3$, and
\begin{equation}
z\defeq\tilde x_1 + \omega\,\tilde x_2 + \bar\omega\,\tilde x_3,\qquad \bar
z\defeq\tilde x_1 + \bar\omega\,\tilde x_2 + \omega\,\tilde x_3
\end{equation}
are the cyclic relative invariants, $\omega$~being a primitive cube root of
unity.  By $\mathfrak{Z}_2,\mathfrak{Z}_3,\mathfrak{S}_3<{\it
  GL}(3,\mathbb{C})$ will be meant the subgroups generated by $\tilde
x_1\leftrightarrow \tilde x_3$ and $\tilde x_1\to\tilde x_2\to \tilde
x_3\to \tilde x_1$, and the symmetric group on $\tilde x_1,\tilde
x_2,\tilde x_3$.

A general HQDS $\bigl(\mathbb{C}^3,\ \dot{\tilde x}=\tilde Q(\tilde
x)\bigr)$, with 18~parameters, can be written as
\begin{equation}
  \label{eq:HQDS}
  \left\{
  \begin{aligned}
    \dot{\tilde x}_1 &= a_{11}\tilde x_1^2 + a_{12}\tilde x_2^2 + a_{13}\tilde x_3^2  + b_{11}\tilde x_2\tilde x_3 + b_{12}\tilde x_3\tilde x_1 + b_{13}\tilde x_1\tilde x_2,\\
    \dot{\tilde x}_2 &= a_{21}\tilde x_1^2 + a_{22}\tilde x_2^2 + a_{23}\tilde x_3^2  + b_{21}\tilde x_2\tilde x_3 + b_{22}\tilde x_3\tilde x_1 + b_{23}\tilde x_1\tilde x_2,\\
    \dot{\tilde x}_3 &= a_{31}\tilde x_1^2 + a_{32}\tilde x_2^2 + a_{33}\tilde x_3^2  + b_{31}\tilde x_2\tilde x_3 + b_{32}\tilde x_3\tilde x_1 + b_{33}\tilde x_1\tilde x_2.
  \end{aligned}
  \right.
\end{equation}
The special case $a_{ij}=a_i\delta_{ij}$,
$b_{ij}=(2a_i-c)\delta_{ij}-a_i+b_j$ is the gDH case, as was mentioned in
the Introduction; cf.~(\ref{eq:gDH}).

\begin{theorem}
  Let\/ $\bigl(\mathbb{C}^3,\ \dot{\tilde x}=\tilde Q(\tilde x)\bigr)$ be a
  HQDS of the form\/ {\rm(\ref{eq:HQDS})} that is invariant under\/
  $\mathfrak{Z}_2$, resp.\ $\mathfrak{Z}_3$, resp.\ $\mathfrak{S}_3$, i.e.,
  satisfying\/ $a_{\pi(i),\pi(j)}=a_{ij}$ and\/ $b_{\pi(i),\pi(j)}=b_{ij}$
  for each supported permutation\/~$\pi$ of\/~$1,2,3$.  In the\/
  $\mathfrak{Z}_2$ case it is also required that
    \begin{sizealign}{\small}
      &\left[(a_{11} + a_{12} + a_{13}) - 2 (a_{21} + a_{22} + a_{23})
      + (a_{31} + a_{32} + a_{33})\right]\nonumber\\
      &\quad{}+\left[(b_{11} + b_{12} + b_{13}) - 2 (b_{21} + b_{22} + b_{23})
      + (b_{31} + b_{32} + b_{33})\right]=0.\nonumber
    \end{sizealign}
    Then, the\/ $\bf2$~map\/ {\rm(\ref{eq:quadraticliftmap})}, resp.\ the\/
    ${\bf3}_{\bf c}$~map\/ {\rm(\ref{eq:3cmap})}, resp.\ the\/ ${\bf6}_{\bf
      c}$~map\/ {\rm(\ref{eq:6cmap})}, will take\/
    $\bigl(\mathbb{C}^3,\ \dot{\tilde x}=\tilde Q(\tilde x)\bigr)$ to some
    HQDS\/ $\bigl(\mathbb{C}^3,\ \dot{ x}=Q( x)\bigr)$.
\label{thm:full1}
\end{theorem}

These claims are readily confirmed with the aid of a computer algebra
system.  The image HQDS's $\dot x=Q(x)$ in the ${\bf3}_{\bf c}$
and~${\bf6}_{\bf c}$ cases are made explicit in the following; the
$\bf2$~case is left to the reader.

\begin{theorem}
  Let a cyclically symmetric\/ {\rm(}i.e.,
  $\mathfrak{Z}_3$-invariant\/{\rm)} HQDS be defined by
  \begin{equation}
    \label{eq:cyclicalqds}
  \begin{split}
    \dot {\tilde x}_i &= -\tilde a(\tilde x_i-\tilde x_j)(\tilde x_k-\tilde x_i)
    +\tilde b(\tilde x_2\tilde x_3 + \tilde x_3\tilde x_1 + \tilde x_1\tilde x_2)
    -\tilde c\,\tilde x_j\tilde x_k\\
    &\qquad{}+\tilde d(\tilde x_1 + \tilde x_2 + \tilde x_3)^2 
    +\tilde e\,\tilde x_i(\tilde    x_j-\tilde x_k) 
    +\tilde f(\tilde x_j^2-\tilde x_k^2),
  \end{split}
  \end{equation}
  for\/ $i,j,k$ equal to each cyclic permutation of\/ $1,2,3$.  The\/
  ${\bf3}_{\bf c}$~map\/ {\rm(\ref{eq:3cmap})} will take this HQDS to the
  image HQDS
  \begin{displaymath}
    \left\{
  \begin{aligned}
    \dot x_1&= -3\,\tilde a(x_1-x_2)(x_3-x_1)
    \hphantom{(-\tilde c)}\, + \mathbf{B}(x_1,x_2,x_3)
      -\tilde c\,x_2x_3  +9\tilde d\,x_1^2,\\
    \dot x_2&= -(3\tilde a-\tilde c)(x_2-x_3)(x_1-x_2) + \mathbf{B}(x_1,x_2,x_3)
      -\tilde c\,x_3x_1  +9\tilde d\,x_1^2\\
      &\qquad +(\omega-\bar\omega)(x_1-x_2)
      \bigl[\tilde e(x_2+2\,x_3)+\tilde f(9\,x_1-x_2-2\,x_3)\bigr],\\
    \dot x_3&= -(3\tilde a-\tilde c)(x_3-x_1)(x_2-x_3)  + \mathbf{B}(x_1,x_2,x_3)
      -\tilde c\,x_1x_2  +9\tilde d\,x_1^2\\
      &\qquad +(\omega-\bar\omega)(x_3-x_1)
      \bigl[\tilde e(2\,x_2+x_3)+\tilde f(9\,x_1-2\,x_2-x_3)\bigr].\\
  \end{aligned}
  \right.
  \end{displaymath}
  If\/ $\tilde e=\tilde f=0$ {\rm(}expanding the invariance group from\/
  $\mathfrak{Z}_3$ to\/ $\mathfrak{S}_3${\rm)}, the\/ ${\bf6}_{\bf
    c}$~map\/ {\rm(\ref{eq:6cmap})} can also be applied, yielding the image
  HQDS
  \begin{displaymath}
    \left\{
  \begin{alignedat}{3}
    \dot x_1&= -3\,\tilde a(x_1-x_2)(x_3-x_1)&& + \mathbf{B}(x_1,x_2,x_3)
      &&-\tilde c\,x_2x_3  +9\tilde d\,x_1^2,\\
    \dot x_2&= -2(3\tilde a-\tilde c)(x_2-x_3)(x_1-x_2)&& + \mathbf{B}(x_1,x_2,x_3)
      &&-\tilde c\,x_3x_1  +9\tilde d\,x_1^2,\\
    \dot x_3&= -3(3\tilde a-\tilde c)(x_3-x_1)(x_2-x_3)&&  + \mathbf{B}(x_1,x_2,x_3)
      &&-\tilde c\,x_1x_2  +9\tilde d\,x_1^2.
  \end{alignedat}
  \right.
  \end{displaymath}
  In both image HQDS's,
  \begin{displaymath}
    \mathbf{B}(x_1,x_2,x_3)\defeq (2\,\tilde c-3\,\tilde b)x_2x_3 + (3\,\tilde b-\tilde c)x_3x_1
        +(3\,\tilde b-\tilde c)x_1x_2.
  \end{displaymath}
\label{thm:full2}
\end{theorem}
\begin{remark*}
  Setting $\tilde d=\tilde e=\tilde f=0$ reduces the initial HQDS to a gDH
  system with $\tilde a_1=\tilde a_2=\tilde a_3=\tilde a$ and $\tilde
  b_1=\tilde b_2=\tilde b_3=\tilde b$; cf.~{(\ref{eq:gDH})}.  Hence the
  terms proportional to $\tilde a,\tilde b,\tilde c$ in each image HQDS can
  be verified by examining Table~{\ref{tab:restrictions}}.  It is only the
  terms proportional to the gDH deformation parameters $\tilde d,\tilde
  e,\tilde f$ that are new.
\end{remark*}
\begin{remark*}
  Setting $\tilde e=\tilde f=0$ removes the terms from the image
  under~${\bf3}_{\bf c}$ that break invariance under $x_2\leftrightarrow
  x_3$, thereby allowing ${\bf 3}_{\bf c}$ to be composed with a further
  solution-preserving map of quadratic type.  But for the map ${\bf 6}_{\bf
    c}$ to be obtained thus, by composing ${\bf 3}_{\bf c}$ with~$\bf2$
  (which is based on symmetry under $x_1\leftrightarrow x_3$, not
  $x_2\leftrightarrow x_3$), two permutations of components are needed;
  cf.~(\ref{eq:permutesps}).
\end{remark*}

It follows from the discussion in \S\,\ref{subsec:liftings} that in each of
the `nice' (i.e.\ Galois) cases ${\bf2},{\bf3}_{\bf c},{\bf 6}_{\bf c}$,
there is a simple interpretation of the polynomials
$\Sigma_1,\Sigma_2,\Sigma_3$ that appear, according to
Theorem~\ref{thm:gDHsystems}, in the solution-preserving map $x=\Phi(\tilde
x)$.  Associated to the respective Galois group $G=
\mathfrak{Z}_2,\mathfrak{Z}_3,\mathfrak{S}_3$ there is a ring $I(G)$ of
polynomial invariants in $\tilde x_1,\tilde x_2,\tilde x_3$; and
$\Sigma_1,\Sigma_2,\Sigma_3$ generate~$I(G)$.  (The case $G=\mathfrak{Z}_2$
is a bit special: a~fourth invariant $\Sigma_1'\defeq x_2$ must be added.)
When $G=\mathfrak{S}_3$, say, this is equivalent to saying that the ring of
symmetric polynomials in $\tilde x_1,\tilde x_2,\tilde x_3$ is generated by
the elementary symmetric polynomials $S_1,S_2,S_3$.  Any HQDS that is
invariant under~$\mathfrak{S}_3$ can be taken to a polynomial differential
system satisfied by the invariants $S_1,S_2,S_3$, via a solution-preserving
map.

The significance of Theorems \ref{thm:full1} and~\ref{thm:full2} is that
they reveal that the differential system satisfied by $x_1,x_2,x_3$ (each a
suitably chosen rational function of $S_1,S_2,S_3$, or more generally of
$\Sigma_1,\Sigma_2,\Sigma_3$) will be a HQDS, like the original system for
$\tilde x_1,\tilde x_2,\tilde x_3$.  In the area of
3\nobreakdash-dimensional differential systems, this may be a new
discovery.  In Ref.~\cite{Kinyon97}, which introduced and exploited the
ring~$I(G)$, several HQDS's were solved, in each case by mapping the HQDS
onto a differential system satisfied by a vector of cyclic invariants.  But
since the invariants were $S_1,S_2,S_3$ rather than the above
$x_1,x_2,x_3$, the system was not a HQDS\null.

Cyclically symmetric HQDS's have arisen many times in dynamical systems
theory.  If $\tilde a=\tilde d=\tilde f=0$, the HQDS~(\ref{eq:cyclicalqds})
of Theorem~\ref{thm:full2} becomes the Leonard--May model of cyclic
competition among three species~\cite{May75,Schuster79}.\footnote{In the
  model as usually defined, each $\dot {\tilde x}_i$ also includes a
  term~$r{\tilde x}_i$, $r$~being the common species growth rate.  The
  substitutions ${\tilde X}_i=e^{-rt}{\tilde x}_i$, ${\rm d}T=e^{rt}\,{\rm
    d}t$ remove these linear terms.}  If moreover ${\tilde b=\tilde c=0}$,
so that $\tilde e$~is its only nonzero parameter, the
HQDS~(\ref{eq:cyclicalqds}) becomes the periodic 3\nobreakdash-particle
Volterra model $\dot{\tilde x}_i=\tilde e\,\tilde x_i(\tilde
x_j-\nobreak\tilde x_k)$, or equivalently the 3\nobreakdash-particle KM
(Kac--van~Moerbeke) system~\cite{Fernandes97}, which is integrable by
elliptic functions~\cite[\S\,11]{Bureau92}.  Such models have been studied
exhaustively, but it may not have been noticed before that in the case of
three particles (or~species), the cyclic symmetry group~$\mathfrak{Z}_3$
can be quotiented~out in such a way as to yield another HQDS\null.

Another noteworthy special case of the HQDS~(\ref{eq:cyclicalqds}) is the
Kasner system $\dot {\tilde x}_i=\allowbreak {\tilde x}_j{\tilde
  x}_k-\nobreak {\tilde x}_i^2$, which is invariant under~$\mathfrak{S}_3$.
(See~\cite{Kasner25,Kinyon97} and~\cite[\S\,5.3]{Walcher91}.)  Of~course
the Kasner system is an (improper) gDH system, with $(a_1,\nobreak
a_2,\nobreak a_3;\allowbreak b_1,\nobreak b_2,\nobreak
b_3;c)=\allowbreak-(1,\nobreak1,\nobreak1;\allowbreak
1,\nobreak1,\nobreak1;3)$.  But the ${\bf 3}_{\bf c}$ half of
Theorem~\ref{thm:full2} gives many examples of solution-preserving maps
from HQDS's that are not gDH systems, or linearly equivalent to them.  (For
instance, the periodic 3\nobreakdash-particle Volterra model is not.)  The
${\bf 6}_{\bf c}$ half is a bit less interesting, since when $\tilde
e=\tilde f=0$, the HQDS (\ref{eq:cyclicalqds}) turns~out always to be
gDH-equivalent.  But the ${\bf 6}_{\bf c}$~half can be generalized in the
following way.
\begin{theorem}
  For each classical PE\nobreakdash-lifting covering map\/ $t=R(\tilde t)$
  {\rm(}of degree\/~$d$\,{\rm)} listed in
  Table\/~{\rm\ref{tab:coveringmapsclassical}}, i.e., for each of\/
  ${\bf2},{\bf3},{\bf4},{\bf6},{\bf6}_{\bf c}$ and\/ ${\bf3}_{\bf
    c},{\bf4_{\bf bq}}$, the associated rational map\/ $x=\Phi(\tilde x)$
  given in Theorem\/~{\rm\ref{thm:gDHsystems}} satisfies\/
  $\Phi=T_{\epsilon}\circ\Phi\circ \tilde T_{\epsilon}^{-1}$, where\/
  $T_{\epsilon},\tilde T_{\epsilon}\in{\it GL}(3,\mathbb{C})$ are defined
  by
  \begin{sizealign}{\small}
    (\tilde{x}_1',\tilde{x}_2',\tilde{x}_3')=\tilde T_{\epsilon}(\tilde x_1,\tilde
    x_2,\tilde x_3) &\defeq (1-{\epsilon})
    (\tilde x_1,\tilde x_2,\tilde x_3) + {\epsilon}\left[\Sigma_1(\tilde
      x_1,\tilde x_2,\tilde x_3)/d\right](1,1,1),\nonumber\\
    (x_1',x_2',x_3')=T_{\epsilon}(x_1,x_2,x_3) &\defeq  (1-{\epsilon})(x_1,x_2,x_3) + {\epsilon}\,x_1(1,1,1),\nonumber
  \end{sizealign}
  ${\epsilon}\neq1$ being arbitrary.  That is, if\/ $\Phi$ maps a gDH
  system\/ $\dot{\tilde x}=\tilde Q(\tilde x)$ to a gDH system\/
  $\dot{x}=Q(x)$, it also maps a one-parameter deformation\/ $\dot{\tilde
    x}'=\tilde Q'(\tilde x')$ of the former to a one-parameter
  deformation\/ $\dot x'=Q'(x')$ of the latter, the deformed\/
  {\rm(}linearly equivalent\/{\rm)} vector fields being\/ $\tilde Q'\defeq
  \tilde T_{\epsilon}\circ \tilde Q\circ\tilde T_{\epsilon}^{-1}$ and\/
  $Q'\defeq T_{\epsilon}\circ Q\circ T_{\epsilon}^{-1}$.
\label{eq:firstdeformed}
\end{theorem}

\begin{proof}
  As in the proof of Theorem~\ref{thm:gDHsystems}, let $y_k\defeq
  \allowbreak x_i-\nobreak x_j$, etc.  What is to be shown is that if
  $\tilde x'=\tilde T\tilde x$ and $x'=T x$, then $x'=\Phi(\tilde x')$.
  But $\tilde y'_k=(1-{\epsilon})\tilde y_k$ and $y'_k=(1-{\epsilon}) y_k$.
  By the formula~(\ref{eq:fundcor}), each~$y_k$ is a homogeneous
  degree\nobreakdash-1 rational function of $\tilde y_1,\tilde y_2,\tilde
  y_3$; thus each~$y_k'$ can be represented as the \emph{same} function of
  $\tilde y_1',\tilde y_2',\tilde y_3'$.  Hence it suffices to prove, say,
  that $x_1'$~is the same function of $\tilde x_1',\tilde x_2',\tilde x_3'$
  as $x_1$~is of $\tilde x_1,\tilde x_2,\tilde x_3$; i.e.\ (see the
  statement of Theorem~\ref{thm:gDHsystems}), that given
  $x_1=\Sigma_1(\tilde x_1,\tilde x_2,\tilde x_3)$, one also has
  $x_1'=\Sigma_1(\tilde x_1',\tilde x_2',\tilde x_3')$.  But by definition,
  $x_1'=x_1$; and for each covering map, $\Sigma_1$~is linear in its
  arguments with coefficients that sum to~$d$ (since they are
  multiplicities, as was remarked after Theorem~\ref{thm:gDHsystems}).  It
  follows that $x_1'=\Sigma_1(\tilde x_1',\tilde x_2',\tilde x_3')$.
\end{proof}

Some elementary further calculations reveal that for each choice of
covering map, each of the deformed quadratic vector fields $\tilde Q',Q'$
of Theorem~\ref{eq:firstdeformed} can be written as the vector field of a
gDH system, of the familiar type shown in~(\ref{eq:gDH}), plus a
deformation field of a simple form.  This is summarized in the following.

\begin{theorem}
  If a rational map\/ $x=\Phi(\tilde x)$ associated by
  Theorem\/~{\rm\ref{thm:gDHsystems}} to one of the classical covering
  maps\/ ${\bf2},{\bf3},{\bf4},{\bf6},{\bf6}_{\bf c}$ and\/ ${\bf3}_{\bf
    c},{\bf4_{\bf bq}}$ {\rm(}of degree\/~$d$\,{\rm)} is
  solution-preserving from\/ ${\rm gDH}(\tilde a_1, \nobreak\tilde a_2,
  \nobreak\tilde a_3; \allowbreak\tilde b_1,\nobreak\tilde
  b_2,\nobreak\tilde b_3;\nobreak\tilde c)$ to\/ ${\rm gDH}(a_1,\nobreak
  a_2,\nobreak a_3;\allowbreak b_1,\nobreak b_2,\nobreak b_3;\nobreak c)$,
  then the same will be true if the two gDH systems are deformed: a term\/
  $\delta \left[\Sigma_1(\tilde x_1,\tilde x_2,\tilde x_3)/d\right]^2$
  being added to the expression for each\/ $\dot{\tilde x}_i$, and a term\/
  $\delta x_1^2$ to that for each\/~$\dot{x}_i$.  The deformation\/
  parameter $\delta\in\mathbb{C}$ is arbitrary.
\end{theorem}

It should be stressed that the deformed systems of this theorem are
linearly equivalent to gDH ones.  The ${\bf 6}_{\bf c}$ case, in which
$d=6$ and $\Sigma_1=2\tilde x_1 + 2\tilde x_2 + 2\tilde x_3$ by
Table~\ref{tab:sigmas}, subsumes the ${\bf 6}_{\bf c}$~half of
Theorem~\ref{thm:full2}.  (In the statement of that theorem, the
deformation parameter~$\delta$ appeared as~$9\tilde d$.)

\subsection{Equivalences among gDH systems}
\label{subsec:equivs}

Distinct gDH systems $\dot x=Q(x)$ and $\dot {\bar x}=\bar Q(\bar x)$ of
the form~(\ref{eq:gDH}) may be linearly equivalent, i.e., may be related by
$\bar x=Tx$ for some $T\in{\it GL}(3,\mathbb{C})$.  That is, the vector
fields $Q,\bar Q$ may satisfy $\bar Q=T\circ\nobreak Q\circ\nobreak
T^{-1}$.  One case is when $T$~permutes the components $(x_1,x_2,x_3)$ into
$(\bar x_1,\bar x_2,\bar x_3)$, perhaps with scaling; but there are also
nontrivial equivalences, summarized in Theorem~\ref{thm:48} below.  The
existence of nontrivial equivalences has implications for the (as~yet
incomplete) classification up~to isomorphism of 3\nobreakdash-dimensional
commutative, non-associative algebras.  Even if one restricts oneself to
algebras~$\mathfrak{A}$ of gDH type, one must quotient~out a group action
on the space of gDH parameters $(a_1,\nobreak a_2,\nobreak a_3;\allowbreak
b_1,\nobreak b_2,\nobreak b_3;c)$.  A fundamental domain for this action
will need to be computed, as in the classification of
2\nobreakdash-dimensional non-associative algebras~\cite{Markus60}.

Nontrivial linear equivalences between gDH systems are an algebraic matter,
but interestingly, explicit formulas for such equivalences follow directly
from the PE-based integration scheme of~\S\,\ref{subsec:fromPEtogDH}.  The
scheme employed offset exponents $(\nu_1,\nobreak \nu_1'; \allowbreak
\nu_2,\nobreak \nu_2'; \allowbreak \nu_3,\nobreak \nu_3')$, rationally
related to $(a_1,\nobreak a_2,\nobreak a_3;\allowbreak b_1,\nobreak
b_2,\nobreak b_3)$ and~$c$ by formulas given in
Theorem~\ref{thm:integration}; and another parameter,~$\bar n$.  In~fact
the 7\nobreakdash-dimensional linear space of gDH systems can alternatively
be parametrized by $(\nu_1,\nobreak \nu_1'; \allowbreak \nu_2,\nobreak
\nu_2'; \allowbreak \nu_3,\nobreak \nu_3';\allowbreak \bar n; c )$.  By
Fuchs's relation, the six offset exponents are not independent;
see~(\ref{eq:nonfree}).

There is a natural group~$\mathfrak{G}$ of gDH equivalences, which is
generated by (1)~permutations of the 3\nobreakdash-set
$\left\{(\nu_1,\nobreak\nu_1'), \allowbreak(\nu_2,\nobreak\nu_2'),
\allowbreak(\nu_3,\nobreak\nu_3')\right\}$, and (2)~the tranpositions
$\nu_i\leftrightarrow\nu_i'$, $i=1,2,3$.  The group~$\mathfrak{G}$ is
clearly isomorphic to the group of \emph{signed} permutations of a
3\nobreakdash-set, which is the order\nobreakdash-48 Coxeter
group,~$\mathcal{B}_3$.  Any element of~$\mathcal{B}_3$ can be written in a
sign-annotated version of the disjoint cycle representation used for
elements of~$\mathfrak{S}_3$.  For instance, $[1_+2_-][3_-]$ signifies
$1\mapsto\nobreak2$, $2\mapsto\nobreak-1$, $3\mapsto\nobreak-3$.
(Positively signed 1\nobreakdash-cycles are typically omitted.)  The group
$\mathcal{B}_3$~is isomorphic to the wreath product $\mathcal{Z}_2
\wr\nobreak \mathfrak{S}_3$, i.e.\ to a semidirect product
$(\mathcal{Z}_2)^3\!\rtimes\nobreak \mathfrak{S}_3$, where the normal
subgroup~$(\mathcal{Z}_2)^3$ is generated by the involutions
$\nu_i\leftrightarrow\nu_i'$, $i=1,2,3$.  In~fact $\mathcal{B}_3$ turns~out
to be isomorphic to $\mathcal{Z}_2\times\nobreak \mathfrak{S}_4$, where the
$\mathcal{Z}_2$ factor is generated by the involution $[1_-][2_-][3_-]$,
i.e., by $\nu_i\leftrightarrow\nu_i'\ (\forall i)$; though this will not be
needed.

\begin{theorem}
  There is an order\/{\rm-48} gDH-stabilizing group\/
  $\mathfrak{G}\cong\allowbreak{\mathcal{B}}_3\cong\allowbreak
  \mathcal{Z}_2 \wr\nobreak \mathfrak{S}_3$, which acts on the gDH
  parameter space as follows.  Positive-signed elements
  of\/~$\mathfrak{G}$, i.e., pure permutations, act by permuting the
  pairs\/ $(a_1,b_1)$, $(a_2,b_2)$, $(a_3,b_3)$; and the associated linear
  equivalence\/ $\bar x=Tx$ permutes the components\/ $x_1,x_2,x_3$ into\/
  $\bar x_1,\bar x_2,\bar x_3$.  Also, for\/ $i=1,2,3$, the involution\/
  $[i_-]\in\mathfrak{G}$ performs the map\/ $(a_1,\nobreak a_2,\nobreak
  a_3;\allowbreak b_1,\nobreak b_2,\nobreak b_3;c)\mapsto (\bar
  a_1,\nobreak \bar a_2,\nobreak \bar a_3;\allowbreak \bar b_1,\nobreak
  \bar b_2,\nobreak \bar b_3;\bar c)$ given by
  \begin{equation}
    \label{eq:parammap}
    \left\{
    \begin{aligned}
      \bar a_i &= c-(c-a_i)\left( \frac{2\,c-b_1-b_2-b_3}{c-a_i-b_i}  \right),\\
      \bar a_j &= a_j \left( \frac{2\,c-b_1-b_2-b_3}{c-a_i-b_i}  \right);\\[\jot]
      \bar b_i &= -(c+a_i - b_1-b_2-b_3)\left( \frac{2\,c-b_1-b_2-b_3}{c-a_i-b_i}  \right),\\
      \bar b_j &= -(c-b_1-b_2-b_3)+(c-b_i-b_k)\left( \frac{2\,c-b_1-b_2-b_3}{c-a_i-b_i}  \right)
    \end{aligned}
    \right.
  \end{equation}
  with\/ $\bar c=c$, the associated linear equivalence\/ $\bar
  x=T_{[i_-]}x$ being
  \begin{equation}
    \label{eq:xmap}
    \left\{
    \begin{aligned}
      \bar x_i &= x_i,\\
      \bar x_j &= x_i + \left( \frac{c-a_i-b_i}{2\,c-b_1-b_2-b_3}  \right)(x_j-x_i).
    \end{aligned}
    \right.
  \end{equation}
  In the preceding, $j$~stands for each element of\/ $\{1,2,3\}$ other
  than\/~$i$, and\/ $k$ for the third element.
\label{thm:48}
\end{theorem}

\begin{proof}
  That the substitution $\bar x=T_{[i_-]}x$ of~(\ref{eq:xmap}) maps ${\rm
    gDH}(a_1,\nobreak a_2,\nobreak a_3;\allowbreak b_1,\nobreak
  b_2,\nobreak b_3;c)$ to ${\rm gDH}(\bar a_1,\nobreak \bar a_2,\nobreak
  \bar a_3;\allowbreak \bar b_1,\nobreak \bar b_2,\nobreak \bar b_3;c)$ can
  be checked by hand.  For a more comprehensible derivation of
  (\ref{eq:parammap}) and~(\ref{eq:xmap}) based on
  Theorem~\ref{thm:integration}, reason as follows.

  From $(a_1,\nobreak a_2,\nobreak a_3;\allowbreak b_1,\nobreak
  b_2,\nobreak b_3)$ and~$c$, compute $(\nu_1,\nobreak \nu_1'; \allowbreak
  \nu_2,\nobreak \nu_2'; \allowbreak \nu_3,\nobreak \nu_3';\bar n)$ by
  Eqs.\ (\ref{eq:abtozetasn}) of that theorem. Perform the transposition
  $\nu_i\leftrightarrow\nu_i'$, and by applying
  Eqs.\ (\ref{eq:zetasntoab}), reverse the direction of computation to
  obtain $(\bar a_1,\nobreak \bar a_2,\nobreak \bar a_3;\allowbreak \bar
  b_1,\nobreak \bar b_2,\nobreak \bar b_3)$.  After algebraic
  simplification of rational functions, one finds~(\ref{eq:parammap}).

  To derive (\ref{eq:xmap}), use the expressions given in
  Theorem~\ref{thm:integration} for $x_1,x_2,x_3$ as logarithmic
  derivatives of PE solutions.  With a common constant of proportionality,
  these are (see~(\ref{eq:2ndremark}))
  \begin{equation}
    \label{eq:easyconseq}
    x_i\propto ({\rm d}/{\rm d}\tau)\log(\Delta_if)
    =({\rm d}/{\rm d}\tau)\log\left[
(t-t_i)^{\mu_j+\mu_k}(t-t_j)^{-\mu_j}(t-t_k)^{-\mu_k}f
      \right],
  \end{equation}
  for $i,j,k$ any cyclic permutation of $1,2,3$.  Now, let~$i$
  once again mean the index at which the interchange $\nu_i\leftrightarrow\nu_i'$
  of offset exponents (i.e., the interchange $\mu_i\leftrightarrow\mu_i'$
  of exponents) takes place, producing $\bar x$ from~$x$.  It then follows
  from~(\ref{eq:easyconseq}) that
  \begin{equation}
    \label{eq:xmap2}
    \left\{
    \begin{aligned}
      \bar x_i &= x_i,\\ \bar x_j &= x_i + \left(
      \frac{\mu_i'+\mu_j+\mu_k}{\mu_i+\mu_j+\mu_k} \right)(x_j-x_i),
    \end{aligned}
    \right.
  \end{equation}
  where $j$~now stands for each element of $\{1,2,3\}$ other than~$i$, and
  $k$ for the third element.  Using $\mu_i'=\mu_i+\alpha_i$ and
  $\mu_l=\nu_l-(2\kappa_l-1){\bar{n}}$, $l=1,2,3$, one has
  \begin{equation}
    \frac{\mu_i'+\mu_j+\mu_k}{\mu_i+\mu_j+\mu_k} = 
    \frac{\alpha_i + \nu_1+\nu_2+\nu_3+\bar n}
         {\hphantom{\alpha_i +{} }\nu_1+\nu_2+\nu_3+\bar n},
  \end{equation}
  as $\kappa_1+\kappa_2+\kappa_3=1$.  Expressing $\bar n$ and
  $\nu_1,\nu_2,\nu_3$ and~$\alpha_i$ in~terms of $(a_1,\nobreak
  a_2,\nobreak a_3;\allowbreak b_1,\nobreak b_2,\nobreak b_3;c)$, by
  Eqs.\ (\ref{eq:abtozetasn}), yields (\ref{eq:xmap})
  from~(\ref{eq:xmap2}).
\end{proof}

In its action on the gDH parameter space, $\mathfrak{G}$~is an order-48
\emph{algebraic group}: a collection of rational maps that act on points
$(a_1,\nobreak a_2,\nobreak a_3;\allowbreak b_1,\nobreak b_2,\nobreak
b_3;c)$.  Hence any gDH system lies on an orbit of up~to 48~gDH systems.
Some orbits are of size strictly less than~48.  For instance, one may have
for some $i\neq\nobreak j$ that $a_i=a_j$ and $b_i=b_j$, so that
$(\nu_i,\nu_i')=\allowbreak (\nu_j,\nu_j')$, in which case
$[i_+j_+]\in\mathfrak{G}$, which performs the interchange
$(\nu_i,\nu_i')\leftrightarrow\allowbreak (\nu_j,\nu_j')$, will leave the
gDH system invariant.  Also, for some $g\in\mathfrak{G}$ the map
$(a_1,\nobreak a_2,\nobreak a_3;\allowbreak b_1,\nobreak b_2,\nobreak
b_3;c)\to\allowbreak (\bar a_1,\nobreak \bar a_2,\nobreak \bar
a_3;\allowbreak \bar b_1,\nobreak \bar b_2,\nobreak \bar b_3;\bar c)$ may
involve a division by zero.  That is, the transformed gDH system may be
undefined.

One should also note that in exceptional cases, applying the
element~$[i_-]\in\mathfrak{G}$, i.e., the transformation
$\nu_i\leftrightarrow\nu_i'$ or equivalently the negation
$\alpha_i\mapsto\allowbreak -\alpha_i$, may convert a proper gDH system to
an improper one, or vice versa.  (Recall Definition~\ref{def:proper}.)

To understand $\mathfrak{G}$ better, it is useful to consider its action on
DH systems, for which $(a_1,\nobreak a_2,\nobreak a_3;\allowbreak
b_1,\nobreak b_2,\nobreak b_3;c)$ equals $(a_1,\nobreak a_2,\nobreak
a_3;\allowbreak b,\nobreak b,\nobreak b;c)$ with $b=c/2$.  (Proper) DH
systems of the form~(\ref{eq:DH}) are parametrized by
$(\alpha_1,\alpha_2,\alpha_3\,|\,c)$, where the angular parameters satisfy
the condition that $\rho^{-1}\defeq \allowbreak (1-\nobreak
\alpha_1-\nobreak \alpha_2-\nobreak \alpha_3)/2$ is nonzero (and also,
$c\neq0$).  The group~$\mathfrak{G}$ acts on the DH parameter space as a
group of signed permutations: the orbit of
$(\alpha_1,\alpha_2,\alpha_3\,|\,c)$ consists of the 48~points
$(\pm\alpha_{1'},\pm\alpha_{2'},\pm\alpha_{3'}\,|\,c)$, where
$\alpha_{1'},\alpha_{2'},\alpha_{3'}$ is a permutation of
$\alpha_1,\alpha_2,\alpha_3$.  This is because the transposition
$\nu_i\leftrightarrow \nu_i'$ specializes to the negation
$\alpha_i\mapsto\nobreak -\alpha_i$.  Owing to (unsigned) angles appearing
multiple times, a $\mathfrak{G}$\nobreakdash-orbit may contain fewer than
48~distinct DH systems.  In exceptional cases a proper DH system can be
transformed to an improper one.

The following explicit formulas dealing with the involution
$[1_-][2_-][3_-]\in\mathfrak{G}$, which performs
$\nu_i\leftrightarrow\nu_i'\ (\forall i)$, will be used in the construction
of gDH representations of the solutions of Chazy equations.

\begin{proposition}
  The linear equivalence\/ $\bar x=T_{[1_-][2_-][3_-]}x$, acting on any
  solution\/ $x=x(\tau)$ of\/ ${\rm gDH}(a,\nobreak a,\nobreak a;\allowbreak
  b,\nobreak b,\nobreak b;c)$, is the circulant map
  \begin{displaymath}
    \bar x_i = x_i -\frac{c+a-2b}{2c-3b}\,(2\,x_i - x_j - x_k), \qquad i=1,2,3,
  \end{displaymath}
  in which\/ $j,k$ are the elements of\/ $\{1,2,3\}$ other than\/~$i$.  The
  corresponding transformation of gDH systems is\/ ${\rm gDH}(a,\nobreak
  a,\nobreak a;\allowbreak b,\nobreak b,\nobreak b;c)\mapsto\allowbreak{\rm
    gDH}(\bar a,\nobreak \bar a,\nobreak \bar a;\allowbreak b,\nobreak
  b,\nobreak b; c)$, where
  \begin{displaymath}
    \bar a = a+\frac{(c-3a)(c+a-2b)}{c+3a-3b}.
  \end{displaymath}
  It is assumed that the transformed gDH system is not undefined, i.e.,
  that there is no division by zero.
\label{prop:123}
\end{proposition}
\begin{proof}
  By a lengthy direct computation, i.e.\ by successively applying the
  formulas for $[i_-]$, $i=1,2,3$, given in Theorem~\ref{thm:48}.
\end{proof}

The group $\mathfrak{G}$ can be used to create new solution-preserving maps
between non-isomorphic gDH systems.  If $g_1,g_2\in\mathfrak{G}$ and
$x=\Phi(\tilde x)$ is one of the rational solution-preserving maps of
Theorem~\ref{thm:gDHsystems}, from some gDH system $\dot{\tilde
  x}=\tilde{Q}(\tilde x)$ to another gDH system $\dot x=Q(x)$, then
$\bar\Phi\defeq g_2\circ\Phi \circ g_1^{-1}$ will be solution-preserving
from $\dot{\bar{\tilde{x}}}=\bar{\tilde{Q}}(\bar{\tilde x})$ to
$\dot{\bar{{x}}}={\bar{Q}}(\bar{x})$, where $\bar{\tilde x},\bar x$ come
from $\tilde x,x$ by $\bar{\tilde x}=T_{g_1}\tilde x$ and $\bar x=T_{g_2}
x$.  In this way, up~to $48^2=2304$ rational maps~$\bar\Phi$ can be
produced.  Of~course maps obtained by mere permutation of components
(of~$\bar{\tilde x}$ and/or~$\bar x$) are of relatively little interest.
However, each map~$\Phi$ of Theorem~\ref{thm:gDHsystems} lies on an orbit
of up~to $8^2=64$ significantly different maps~$\bar\Phi$, obtained by
acting on~$\tilde x$ and/or~$x$ by elements of~$\mathfrak{G}$ that are
products of zero or more negatively signed 1\nobreakdash-cycles, such as
the symmetric involution $[1_-][2_-][3_-]$ (the product of all three).

If one focuses on the hypergeometric integration scheme
of~\S\,\ref{subsec:liftings}, one can interpret
$\mathfrak{G}\cong\mathcal{B}_3$ as the automorphism group of the Papperitz
equation~(\ref{eq:PE}), or of its P\nobreakdash-symbol~(\ref{eq:PPsymbol}).
(That group is generated by permutations of the singular points
$t_1,t_2,t_3$ and interchanges $\mu_1\leftrightarrow\mu_1'$,
$\mu_2\leftrightarrow\mu_2'$, $\mu_3\leftrightarrow\mu_3'$ of exponents.)
The group~$\mathfrak{G}$ is therefore closely related to the automorphism
group of the Gauss hypergeometric equation~(\ref{eq:GHE}), or of its
P\nobreakdash-symbol~(\ref{eq:GHPsymbol}).  As usually defined, the latter
group is the order\nobreakdash-24 Coxeter group of \emph{even-signed}
permutations of a 3\nobreakdash-set, which is isomorphic
to~$\mathfrak{S}_4$.  It can be viewed as the origin of Kummer's well-known
set of 24~local solutions of the GHE, each of which is expressed in~terms
of~${}_2F_1$, and four of which are in~fact identical to
${}_2F_1$~\cite{Maier10}.

The many alternative rational solution-preserving maps $\bar\Phi\defeq
g_2\circ\Phi \circ g_1^{-1}$ that can be produced from each~$\Phi$,
i.e.\ from each of the PE\nobreakdash-lifting maps of
Table~\ref{tab:coveringmapsclassical}, have a hypergeometric
interpretation.  Like~$\Phi$, each is associated to a classical ${}_2F_1$
transformation that is based on a covering $\bar R\colon
\mathbb{P}^1_{\tilde t}\to \mathbb{P}^1_{t}$, but $\bar R$~may differ from
the covering~$R$ of the table because of pre- and/or post-composition with
M\"obius transformations that permute the singular points $\tilde
t=0,1,\infty$ and/or $t=0,1,\infty$.  There may also be interchanges of
exponents.  For example, there are quadratic solution-preserving maps that
are associated to the classical quadratic transformations of~${}_2F_1$
other than the canonical one,~(\ref{eq:quadexample}).  But, each of the
other quadratic transformations can be obtained from~(\ref{eq:quadexample})
by P\nobreakdash-symbol manipulations~\cite{Andrews99}.

\section{gDH Systems of Painlev{\'e} Type}
\label{sec:painleve}

A gDH system~(\ref{eq:gDH}) may have the Painlev\'e property (PP),
according to which no local solution has a branch point; at~least, not one
that is movable, with a location depending on the choice of initial
condition $x(\tau_0)=x^0$.  

The (proper) DH case is fully understood.  A proper DH system ${\rm
  DH}(\alpha_1,\alpha_2,\alpha_3)$ has the~PP if and only if
$(\alpha_1,\nobreak\alpha_2,\nobreak\alpha_3)=\allowbreak(\frac1{N_1},\nobreak\frac1{N_2},\nobreak\frac1{N_3})$,
where each~$N_i$ is a nonzero integer or~$\infty$.  The most familiar
subcase is when each~$N_i$ is a \emph{positive} integer or~$\infty$.  If
$\alpha_1+\nobreak\alpha_2+\allowbreak\alpha_3<\nobreak 1$, any
noncoincident solution $x=x(\tau)$ of~(\ref{eq:gDH}) is confined by a
natural barrier (a~`wall of poles') to a half-plane or disk, as mentioned
in the Introduction.  This maximal domain of definition is
solution-dependent, i.e., depends on the initial condition.  If the domain
is ${\rm Im}\,\tau>0$, then $x=x(\tau)$ will be a vector of quasi-modular
forms for a triangle group $\Delta(N_1,N_2,N_3) <{\it PSL}(2,\mathbb{R})$.

Non-DH gDH systems have no modular interpretation: they are
affine-covariant but not projective-covariant.  But one can prove the
following sophisticated classification theorem
(Theorem~\ref{thm:PPclassification}, including
Table~\ref{tab:PPclassification}) on proper non\nobreakdash-DH gDH systems
that have the~PP\null.  Interestingly, many such systems are related by
rational solution-preserving maps $x=\Phi(\tilde x)$ of the types
constructed in~\S\,\ref{sec:transformations}.

The classification is made possible by the proper ${\rm
  gDH}\leftrightarrow{\rm gSE}$ (generalized Schwarzian equation)
correspondence of Theorem~\ref{thm:gDHbasecoin}.  Recall that this
correspondence is based on the maps $t(\cdot)\mapsto x(\cdot)$,
$x(\cdot)\mapsto t(\cdot)$ of
Eqs.\ (\ref{eq:zerothclaim}),(\ref{eq:secondclaim}), the latter reducing to
$t=-(x_2\nobreak -x_3)/\allowbreak(x_1\nobreak -x_2)$ if
$(t_1,t_2,t_3)=\allowbreak(0,1,\infty)$.
Theorem~\ref{thm:PPclassification} is accordingly a corollary of the
classification of the proper gSE's of the form~(\ref{eq:gSE}) that have
the~PP but are not~SE's, which was begun by Garnier~\cite{Garnier12} and
completed by Carton-LeBrun~\cite{CartonLeBrun69b}, using a rigorous version
of Painlev\'e's $\alpha$\nobreakdash-method.  Garnier and Carton-LeBrun
chose $(0,1,\infty)$ as the gSE singular points, with no loss of
generality.

Theorem~\ref{thm:PPclassification} is followed, for purposes of
illustration, by the explicit integration of several proper
non\nobreakdash-DH gDH systems with the~PP\null.  In fact, each of the
systems in Table~\ref{tab:PPclassification} is integrable; this follows
from the integration of all non-SE gSE's with the~PP, also performed by
Garnier and Carton--LeBrun.  (See Theorem~\ref{thm:PPintegration},
including Table~\ref{tab:PPintegration}, and the detailed Examples
\ref{ex:1}--\ref{ex:4}.)  In most cases, the noncoincident gDH solutions
$x=x(\tau)$ are doubly or at~least simply periodic in~$\tau$.  In
Example~\ref{ex:2}, an application to the integration of
Chazy\nobreakdash-XI is indicated.

\afterpage{\clearpage\rm
\begin{landscape}
{
\begin{longtable}{|l||l|l|l|l|l|}
  \caption{All proper non-DH gDH systems with the Painlev{\'e} property,
    up~to isomorphism.  ($r\in\mathbb{Z}\setminus\{0\}$ is arbitrary.)
    Several components~$x_i$ satisfying Chazy equations are indicated in
    brackets.}  \\
  \hline
  system &
  $(a_1,a_2,a_3;b_1,b_2,b_3;c)$ &
  $(\nu_1,\nu_1';\nu_2,\nu_2';\nu_3,\nu_3')$ &
  $(\alpha_1,\alpha_2,\alpha_3)$ & $(r_1,r_1';r_2,r_2';r_3,r_3')$ &
  preimage(s), [notes]  \\ \hline\hline
  \endfirsthead
  \hline
  system &
  $(a_1,a_2,a_3;b_1,b_2,b_3;c)$ &
  $(\nu_1,\nu_1';\nu_2,\nu_2';\nu_3,\nu_3')$ &
  $(\alpha_1,\alpha_2,\alpha_3)$ & $(r_1,r_1';r_2,r_2';r_3,r_3')$ &
  preimage(s), [notes]  \\ \hline\hline
  \endhead
  \hline \multicolumn{6}{r}{\emph{Continued on next page}}
  \endfoot
  \hline 
  \endlastfoot 
  \hline \multicolumn{6}{|l|}{$n$ an
    arbitrary integer ($\neq0,-1,-2$), with $\bar n=(n+1)/n$}\\* \hline
  e.I.1($n$) & $(-2,-3,-1;$ &
  $(-\frac1{3n},-\frac{n+1}{3n};-\frac1{2n},-\frac{n+1}{2n};$ &
  $(-\frac13,-\frac12,-\frac16)$ & $(3n+3,3;2n+2,2;$ & \\* &
  $2n-2,3n,n-4;6n)$ & $-\frac1{6n},-\frac{n+1}{6n})$ & & $6n+6,6)$ & \\*
  e.I.2($n$) & $(-1,-2,-1;$ &
  $(-\frac1{4n},-\frac{n+1}{4n};-\frac1{2n},-\frac{n+1}{2n};$ &
  $(-\frac14,-\frac12,-\frac14)$ & $(4n+4,4;2n+2,2;$ & \\* &
  $n-2,2n,n-2;4n)$ & $-\frac1{4n},-\frac{n+1}{4n})$ & & $4n+4,4)$ & \\*
  e.II($n$) & $(-1,-1,-1;$ &
  $(-\frac1{3n},-\frac{n+1}{3n};-\frac1{3n},-\frac{n+1}{3n};$ &
  $(-\frac13,-\frac13,-\frac13)$ & $(3n+3,3;3n+3,3;$ & e.II($n$) via
  ${\bf3}_{\bf c}$ \\* & $n-1,n-1,n-1;3n)$ & $-\frac1{3n},-\frac{n+1}{3n})$
  & & $3n+3,3)$ & \\* e.IV.4($n$) & $(-1,0,-1;$ &
  $(-\frac1{2n},-\frac{n+1}{2n};0,0;$ & $(-\frac12,0,-\frac12)$ &
  $(2n+2,2;\infty,\infty;$ & e.IV($n$,1,$\infty$) via $\bf2$ \\* &
  $n,-2,n;2n)$ & $-\frac1{2n},-\frac{n+1}{2n})$ & & $2n+2,2)$ & \\*
  e.IV($n$,1,$\infty$) & $(0,-1,0;$ & $(0,0;-\frac1n,-\frac{n+1}{n};$ &
  $(0,-1,0)$ & $(\infty,\infty;n+1,1;$ & \\ & $-1,n+1,-1;n)$ & $0,0)$ & &
  $\infty,\infty)$ & \\ \hline e.IV($n$,1,$r$) & $(n+1,r(n+1),-1;$ & $(
  -\frac{n+1}{nr}, -\frac{1}{nr}; -\frac{n+1}{n},-\frac{1}{n};$ &
  $(\frac1r,1,\frac1r)$ & $(r,r(n+1);1,n+1;$ & \\ & $-1,-1,n+1;n)$ &
  $\frac{1}{nr},\frac{n+1}{nr})$ & & $-r(n+1),-r)$ & \\ \hline \hline
  \multicolumn{6}{|l|}{$n=1$ (i.e., $\bar n=2$)}\\* \hline I &
  $(2,2,1;-1,-1,2;3)$ & $(-2,0;-2,0;-1,2)$ & $(2,2,3)$ &
  $(1,\infty;1,\infty;2,-1)$ & \\ I.1 & $(2,4,1;-3,1,2;3)$ &
  $(-1,-\frac12;-2,0;-\frac12,1)$ & $(\frac12,2,\frac32)$ &
  $(2,4;1,\infty;4,-2)$ & I via ${\bf2}$ \\ II & $(1,1,0;0,0,0;1)$ &
  $(-2,0;-2,0;0,1)$ & $(2,2,1)$ & $(1,\infty;1,\infty;\infty,-2)$ & \\*
  II.1 & $(1,2,0;-1,1,0;1)$ & $(-1,-\frac12;-2,0;0,\frac12)$ &
  $(\frac12,2,\frac12)$ & $(2,4;1,\infty;\infty,-4)$ & II via ${\bf2}$
  \\ III & $(1,1,1;0,-1,1;2)$ & $(-2,1;-2,0;-2,2)$ & $(3,2,4) $ &
  $(1,-2;1,\infty;1,-1)$ & \\ IV & $(2,2,1;1,-1,0;3)$ & $(-2,1;-2,0;-1,1)$
  & $(3,2,2) $ & $(1,-2;1,\infty;2,-2)$ & \\ V & $(1,1,1;0,0,0;2)$ &
  $(-2,1;-2,1;-2,1)$ & $(3,3,3) $ & $(1,-2;1,-2;1,-2)$ & \\* V.1 &
  $(1,2,1;-2,2,0;2)$ & $(-1,-\frac12;-2,1;-1,\frac12)$ &
  $(\frac12,3,\frac32) $ & $(2,4;1,-2;2,-4)$ & V via ${\bf2}$ \\ V.2 &
  $(3,2,3;-2,-2,4;2)$ & $(-1,-\frac12;-\frac23,-\frac13;-1,\frac12)$ &
  $(\frac12,\frac13,\frac32) $ & $(2,4;3,6;2,-4)$ & V.1,3;V via
  ${\bf3},{\bf2};{\bf6}_{\bf c}$ \\* & & & & & [Ch-XI($4$) for~$x_3$] \\*
  V.3 & $(1,3,1;-2,4,-2;2)$ & $(-\frac23,-\frac13;-2,1;-\frac23,-\frac13)$
  & $(\frac13,3,\frac13) $ & $(3,6;1,-2;3,6)$ & V via ${\bf3}_{\bf c}$
  \\ VI & $(1,2,2;-2,1,1;3)$ & $(-1,0;-2,1;-2,1)$ & $(1,3,3) $ &
  $(2,\infty;1,-2;1,-2)$ & \\* VI.1 & $(2,4,1;-3,5,-2;3)$ &
  $(-1,-\frac12;-2,1;-\frac12,0)$ & $(\frac12,3,\frac12) $ &
  $(2,4;1,-2;4,\infty)$ & VI via ${\bf2}$ \\ VII.1 & $(1,2,1;-1,0,1;2)$ &
  $(-1,0;-2,0;-1,1)$ & $(1,2,2) $ & $(2,\infty;1,\infty;2,-2)$ & \\* VII.2
  & $(1,3,1;-2,1,1;2)$ & $(-\frac23,-\frac13;-2,0;-\frac23,\frac23)$ &
  $(\frac13,2,\frac43) $ & $(3,6;1,\infty;3,-3)$ & \\ VIII.1 &
  $(2,6,1;-3,3,0;3)$ & $(-\frac23,-\frac13;-2,0;-\frac13,\frac13)$ &
  $(\frac13,2,\frac23) $ & $(3,6;1,\infty;6,-6)$ & \\* X.1 &
  $(1,4,2;-2,1,1;3)$ & $(-\frac12,0;-2,0;-1,\frac12)$ &
  $(\frac12,2,\frac32) $ & $(4,\infty;1,\infty;2,-4)$ & \\ XI.1 &
  $(2,6,1;-1,3,-2;3)$ & $(-\frac23,0;-2,0;-\frac13,0)$ &
  $(\frac23,2,\frac13) $ & $(3,\infty;1,\infty;6,\infty)$ & \\ XII.1 &
  $(1,4,1;-2,2,0;2)$ & $(-\frac12,-\frac14;-2,0;-\frac12,\frac14)$ &
  $(\frac14,2,\frac34) $ & $(4,8;1,\infty;4,-8)$ & \\ XIII.1 &
  $(1,4,1;-1,2,-1;2)$ & $(-\frac12,0;-2,0;-\frac12,0)$ &
  $(\frac12,2,\frac12) $ & $(4,\infty;1,\infty;4,\infty)$ & \\* XIII.2 &
  $(6,4,3;-1,-1,2;1)$ & $(-1,-\frac12;-\frac23,-\frac13;-\frac12,0)$ &
  $(\frac12,\frac13,\frac12) $ & $(2,4;3,6;4,\infty)$ & XIII.1,3,4,5 \\* &
  & & & & \quad via ${\bf6},{\bf3},{\bf6}_{\bf c},{\bf2}$ \\* XIII.3 &
  $(2,2,1;-1,1,0;1)$ & $(-1,-\frac12;-1,0;-\frac12,0)$ &
  $(\frac12,1,\frac12) $ & $(2,4;2,\infty;4,\infty)$ & XIII.1,4 via ${\bf
    2}$ \\* XIII.4 & $(1,1,1;0,0,0;1)$ & $(-1,0;-1,0;-1,0)$ & $(1,1,1) $ &
  $(2,\infty;2,\infty;2,\infty)$ & \\* XIII.5 & $(2,3,2;-1,2,-1;1)$ &
  $(-\frac23,-\frac13;-1,0;-\frac23,-\frac13)$ & $(\frac13,1,\frac13) $ &
  $(3,6;2,\infty;3,6)$ & XIII.4 via ${\bf 3}_{\bf c}$ \\ \hline
  e.IV(1,2,$r$) & $(2,2r,0;0,-2,2;2)$ & $(-\frac2r,0;-2,-1;0,\frac2r)$ &
  $(\frac2r,1,\frac2r)$ & $(r,\infty;1,2;\infty,-r)$ & \\ e.IV(1,3,$r$) &
  $(2,2r,1;1,-3,2;3)$ & $(-\frac2r,\frac1r;-2,-1;-\frac1r,\frac2r)$ &
  $(\frac3r,1,\frac3r)$ & $(r,-2r;1,2;2r,-r)$ & \\ e.IV(1,4,$r$) &
  $(2,2r,2;2,-4,2;4)$ & $(-\frac2r,\frac2r;-2,-1;-\frac2r,\frac2r)$ &
  $(\frac4r,1,\frac4r)$ & $(r,-r;1,2;r,-r)$ & \\ e.IV.1($r$) &
  $(r,2,r;-2,4,-2;2)$ & $(-1,-\frac12;-\frac2r,\frac2r;-1,-\frac12)$ &
  $(\frac12,\frac4r,\frac12) $ & $(2,4;r,-r;2,4)$ & e.IV(1,4,$r$) via
  $\bf2$\\\newpage
  \multicolumn{6}{|l|}{$n=2$ (i.e., $\bar
    n=3/2$)}\\* \hline I & $(1,1,1;0,0,1;2)$ &
  $(-\frac32,\frac12;-\frac32,\frac12;-\frac32,\frac32)$ & $(2,2,3) $ &
  $(1,-3;1,-3;1,-1)$ & \\* I.1 & $(1,2,1;-1,1,1;2)$ &
  $(-\frac34,-\frac14;-\frac32,\frac12;-\frac34,\frac34)$ &
  $(\frac12,2,\frac32) $ & $(2,6;1,-3;2,-2)$ & I via ${\bf2}$ \\ II &
  $(3,3,1;1,1,0;4)$ &
  $(-\frac32,\frac12;-\frac32,\frac12;-\frac12,\frac12)$ & $(2,2,1) $ &
  $(1,-3;1,-3;3,-3)$ & \\* II.1 & $(3,6,1;-2,4,0;4)$ &
  $(-\frac34,-\frac14;-\frac32,\frac12;-\frac14,\frac14)$ &
  $(\frac12,2,\frac12)$ & $(2,6;1,-3;6,-6)$ & II via ${\bf2}$\\ III.1 &
  $(4, 6, 4; 3, -1, -1; 2)$ &
  $(-\frac12,\frac16;-\frac34,-\frac14;-\frac12,-\frac16)$ &
  $(\frac23,\frac12,\frac13)$ & $(3,-9;2,6;3,9)$ & [Ch-XI(9) for~$x_1$] \\*
  III.2 & $(1, 3, 1; -1, 2, 0; 2)$ &
  $(-\frac12,-\frac16;-\frac32,\frac12;-\frac12,\frac16)$ &
  $(\frac13,2,\frac23)$ & $(3,9;1,-3;3,-9)$ & \\ \hline e.IV(2,2,$r$) &
  $(3,3r,1;1,-2,3;4)$ &
  $(-\frac3{2r},\frac1{2r};-\frac32,-\frac12;-\frac1{2r},\frac3{2r})$ &
  $(\frac2r,1,\frac2r)$ & $(r,-3r;1,3;3r,-r)$ & \\ e.IV(2,3,$r$) &
  $(3,3r,3;3,-3,3;6)$ &
  $(-\frac3{2r},\frac3{2r};-\frac32,-\frac12;-\frac3{2r},\frac3{2r})$ &
  $(\frac3r,1,\frac3r)$ & $(r,-r;1,3;r,-r)$ & \\ e.IV.2($r$) & $(r, 2, r;
  -1, 3, -1; 2)$ &
  $(-\frac34,-\frac14;-\frac3{2r},\frac3{2r};-\frac34,-\frac14)$ &
  $(\frac12,\frac3r,\frac12) $ & $(2,6;r,-r;2,6)$ & e.IV(2,3,$r$) via
  $\bf2$\\ \hline \hline \multicolumn{6}{|l|}{$n=3$ (i.e., $\bar
    n=4/3$)}\\* \hline I & $(2, 2, 1; 1, 1, 0; 3)$ &
  $(-\frac43,\frac23;-\frac43,\frac23;-\frac23,\frac13)$ & $(2,2,1)$ &
  $(1,-2;1,-2;2,-4)$ & \\* I.1 & $(2, 4, 1; -1, 3, 0; 3)$ &
  $(-\frac23,-\frac16;-\frac43,\frac23;-\frac13,\frac16)$ &
  $(\frac12,2,\frac12) $ & $(2,8;1,-2;4,-8)$ & I via ${\bf2}$ \\ \hline
  e.IV(3,2,$r$) & $(4,4r,2;2,-2,4;6)$ &
  $(-\frac4{3r},\frac2{3r};-\frac43,-\frac13;-\frac2{3r},\frac4{3r})$ &
  $(\frac2r,1,\frac2r)$ & $(r,-2r;1,4;2r,-r)$ & \\ \hline \hline
  \multicolumn{6}{|l|}{$n=5$ (i.e., $\bar n=6/5$)}\\* \hline e.IV(5,2,$r$)
  & $(6,6r,4;4,-2,6;10)$ &
  $(-\frac6{5r},\frac4{5r};-\frac65,-\frac15;-\frac4{5r},\frac6{5r})$ &
  $(\frac2r,1,\frac2r)$ & $(r,-\frac{3}{2}r;1,6;\frac{3}{2}r,-r)$ & [N.B.: even $r$ only]\\ \hline \hline
  \multicolumn{6}{|l|}{$n=\infty$ (i.e., $\bar n=1$)}\\* \hline
  e.I.1($\infty$) & $(0,0,0;2,3,1;6)$ &
  $(0,-\frac13;0,-\frac12;0,-\frac16)$ & $(-\frac13,-\frac12,-\frac16)$ &
  $(\infty,3;\infty,2;\infty,6)$ & \\* e.I.2($\infty$) & $(0,0,0;1,2,1;4)$
  & $(0,-\frac14;0,-\frac12;0,-\frac14)$ & $(-\frac14,-\frac12,-\frac14)$ &
  $(\infty,4;\infty,2;\infty,4)$ & \\* e.II($\infty$) & $(0,0,0;1,1,1;3)$ &
  $(0,-\frac13;0,-\frac13;0,-\frac13)$ & $(-\frac13,-\frac13,-\frac13)$ &
  $(\infty,3;\infty,3;\infty,3)$ & e.II($\infty$) via ${\bf3}_{\bf c}$\\*
  e.IV.4($\infty$) & $(0,0,0;1,0,1;2)$ & $(0,-\frac12;0,0;0,-\frac12)$ &
  $(-\frac12,0,-\frac12)$ & $(\infty,2;\infty,\infty;\infty,2)$ &
  e.IV($\infty$,1,$\infty$) via $\bf2$ \\* e.IV($\infty$,1,$\infty$) &
  $(0,0,0;0,1,0;1)$ & $(0,0;0,-1;0,0)$ & $(0,-1,0)$ &
  $(\infty,\infty;\infty,1;\infty,\infty)$ & \\* \hline
  e.IV($\infty$,1,$r$) & $(1,r,0;0,0,1;1)$ & $(-\frac1r,0;-1,0;0,\frac1r)$
  & $(\frac1r,1,\frac1r)$ & $(r,\infty;1,\infty;\infty,-r)$ & \\*
  e.IV($\infty$,2,$r$) & $(1,r,1;1,0,1;2)$ &
  $(-\frac1r,\frac1r;-1,0;-\frac1r,\frac1r)$ & $(\frac2r,1,\frac2r)$ &
  $(r,-r;1,\infty;r,-r)$ & \\* e.IV.3($r$) & $(r, 2, r; 0, 2, 0; 2)$ &
  $(-\frac12,0;-\frac1r,\frac1r;-\frac12,0)$ & $(\frac12,\frac2r,\frac12)$
  & $(2,\infty;r,-r;2,\infty)$ & e.IV($\infty$,2,$r$) via $\bf2$
\label{tab:PPclassification}
\end{longtable}
}
\end{landscape}
}  

\begin{theorem}
  The non\nobreakdash-DH gDH systems with the Painlev\'e property, which
  are proper in the sense of Definition\/ {\rm\ref{def:proper}}, i.e.,
  satisfy\/ {\rm(i)} $c\neq0$, {\rm(ii)} $2c-\nobreak b_1-\nobreak
  b_2-\nobreak b_3\neq0$, and\/ {\rm(iii)} $c-\nobreak a_1-\nobreak
  a_2-\nobreak a_3\neq0$, are listed up to linear equivalence in
  Table\/~{\rm\ref{tab:PPclassification}}.

  On each line a vector\/ $(a_1,\nobreak a_2,\nobreak a_3;\allowbreak
  b_1,\nobreak b_2,\nobreak b_3;c)$ is given, as are the elements of the
  alternative\/ {\rm(}birationally equivalent\/{\rm)} vector\/
  $(\nu_1,\nobreak \nu_1'; \allowbreak \nu_2,\nobreak \nu_2'; \allowbreak
  \nu_3,\nobreak \nu_3';\allowbreak \bar n; c )$, and the angular
  parameters\/ $\alpha_i\defeq\allowbreak\nu_i'-\nobreak\nu_i$.  {\rm(}The
  freedom in the choice of\/~$c$ is exploited to make each of\/
  $a_1,\nobreak a_2,\nobreak a_3; \allowbreak b_1,\nobreak b_2,\nobreak
  b_3$ an integer.{\rm)} In each non\nobreakdash-DH\/ {\rm(}$\bar
  n\neq1/2${\rm)} case with the PP, $\bar n$~equals $(n+1)/n$, with\/
  $n$~an integer\/ {\rm(}$n\neq0,-1,-2${\rm)} or\/~$\infty$.  The table is
  partitioned according to\/~$n$.
  
  The list is complete, up to the multiplication of\/ $(a_1,\nobreak
  a_2,\nobreak a_3;\allowbreak b_1,\nobreak b_2,\nobreak b_3;c)$ by a
  nonzero constant, and up to the action on the gDH parameter space of the
  order\/{\rm-48} isomorphism group\/~$\mathfrak{G}$ of
  Theorem\/~{\rm\ref{thm:48}}, which is generated by permutations of the
  components\/ $i=1,2,3$, and the three interchanges\/
  $\nu_i\leftrightarrow\nu_i'$, $i=1,2,3$.  Thus each line may stand for up
  to\/ {\rm48} distinct\/ $(\nu_1,\nobreak \nu_1'; \allowbreak
  \nu_2,\nobreak \nu_2'; \allowbreak \nu_3,\nobreak \nu_3';\bar n)$; and up
  to\/ {\rm48} distinct\/ $(a_1,\nobreak a_2,\nobreak a_3;\allowbreak
  b_1,\nobreak b_2,\nobreak b_3;c)$, each of which can be multiplied by a
  nonzero constant.
\label{thm:PPclassification}
\end{theorem}

The classification data in Table~\ref{tab:PPclassification} were extracted
from Tables I--\nobreak{VII} of~\cite{CartonLeBrun69b}, with some
labor. The following brief explanation of the results
of~\cite{CartonLeBrun69b}, formal rather than rigorous, may be useful.
Consider a gSE of the form~(\ref{eq:gSE}) with parameters $(\nu_1,\nobreak
\nu_1'; \allowbreak \nu_2,\nobreak \nu_2'; \allowbreak \nu_3,\nobreak
\nu_3';\bar n)$, the parameters associated to $t=t_i$ being $\nu_i,\nu_i'$,
and the Fuchsian condition being $\sum_{i=1}^3(\nu_i+\nobreak\nu'_i)=
\allowbreak 1-\nobreak2\bar n$.  Of~interest is the behavior of a solution
$t=t(\tau)$ of the gSE near any of its singular points, say $\tau=\tau_*$.
For the PP to obtain, the solution must not be branched.

As was mentioned at the end of~\S\,\ref{subsec:fromPEtogDH}, by
substituting into the gSE the formal statement $t=\allowbreak
t(\tau)\sim\allowbreak t_* +\nobreak C(\tau-\nobreak\tau_*)^p$ where
$t_*\neq t_1,t_2,t_3$, it is easy to deduce that $p\in\{\pm1, \pm(n+1) \}$,
with the statement holding as~$\tau\to\tau_*$ if $\mathop{\text{Re}}p>0$
and as~$\tau\to\infty$ if $\mathop{\text{Re}}p<0$.  Here $n\defeq 1/(\bar
n-1)$ and $\bar n=(n+1)/n$.  Formally, this is the source of the
restriction that $n$ be an integer or~$\infty$.

If instead $t_*=t_i$ for one of $i=1,2,3$, one still finds that a leading
exponent~$p$ governing $\tau\to\infty$ behavior (where
$\mathop{\text{Re}}p<0$) must satisfy $p\in\{-1,-n-\nobreak 1\}$; but now,
an exponent~$p$ governing $\tau\to\tau_*$ behavior (where
$\mathop{\text{Re}}p>0$) must satisfy $p\in\{r_i,r_i'\}$, where
$(r_i,r_i')\defeq -\bar n(1/\nu_i, 1/\nu_i')$.  The appearance of new
exponents is a reflection of the following.  In Papperitz-based integration
(with the local function $\tau=\tau(t)$ and its inverse $t=t(\tau)$ defined
as in Theorem~\ref{thm:integration}), the two Frobenius solutions of the~PE
at $t=t_i$, say $f^{[i]},f^{[i']}$, turn~out to yield gSE solutions
$t=t^{[i]}(\tau)$, $t^{[i']}(\tau)$ with formal asymptotic behavior
\begin{subequations}
\label{eq:rsolns}
\begin{align}
  t^{[i]}(\tau) &\sim t_i + \textrm{const}\times (\tau-\tau_*)^{r_i},\label{eq:rsolnsa}\\
  t^{[i']}(\tau) &\sim t_i + \textrm{const}\times (\tau-\tau_*)^{r_i'},\label{eq:rsolnsb}
\end{align}
\end{subequations}
where $(r_i,r_i')\defeq-\bar n(1/\nu_i,1/\nu'_i)$, and lower-order terms
with exponents differing by integers from $r_i$, resp.~$r'_i$, are omitted.
In either approach, one expects that for the gSE to have the~PP, $r_i,r_i'$
must be integers (at~least, if they are finite, with positive real parts).
On each line of Table~\ref{tab:PPclassification}, $(r_1,\nobreak
r_1';\allowbreak r_2,\nobreak r_2';\allowbreak r_3,\nobreak r_3')$ is
given.  The identity
\begin{equation}
\label{eq:rcondition}
-(n+1)^{-1}+\sum_{i=1}^3 (r_i^{-1} + r_i'^{-1}) = 1
\end{equation}
is a restatement of the Fuchsian condition in~terms of these exponents.

The PE having a 2\nobreakdash-dimensional space of solutions, each of its
basis functions $f^{[i]},f^{[i']}$ can be continuously deformed; and one
finds that from any nontrivial mixture comes a gSE solution $t=t(\tau)$
with one of two formal behaviors:
\begin{subequations}
\label{eq:qsolns}
\begin{align}
\bar t^{[i]}(\tau) & \sim t_i + \textrm{const}\times
(\tau-\tau_*)^{r_i}\times \bigl[1+\textrm{const}\times
  (\tau-\tau_*)^{q_i}\bigr],
\label{eq:qsolnsa}
\\ \bar t^{[i']}(\tau) & \sim t_i +
\textrm{const}\times (\tau-\tau_*)^{r_i'}\times
\bigl[1+\textrm{const}\times (\tau-\tau_*)^{q_i'}\bigr],
\label{eq:qsolnsb}
\end{align}
\end{subequations}
where lower-order terms are omitted, and the definition of the auxiliary
exponents is $(q_i,q_i')\defeq \bar
n\left((\nu_i-\nu_i')/\nu_i,(\nu'_i-\nu_i)/\nu'_i\right)$, so that if
defined they satisfy
\begin{equation}
\label{eq:qcondition}
(n+1)^{-1} + q_i^{-1}  + q_i'^{-1} = 1, \qquad i=1,2,3.
\end{equation}
The `mixed' solution (\ref{eq:qsolnsa}), resp.\ (\ref{eq:qsolnsb}), occurs
if $\nu_i<\nu_i'$, resp.\ $\nu_i'<\nobreak \nu_i$.  The conditions that for
$i=1,2,3$, each of $r_i,\nobreak r_i';\allowbreak q_i,\nobreak q_i'$ be an
integer (or~$\infty$ or undefined [i.e.~`$0/0$'], either indicating that
there is no singularity at~all), imposed along with the condition that
$n\defeq1/(\bar n -1)$ be an integer or~$\infty$, and the Fuchsian
condition (\ref{eq:rcondition}), greatly restrict the possible
$(\nu_1,\nobreak \nu_1'; \allowbreak \nu_2,\nobreak \nu_2'; \allowbreak
\nu_3,\nobreak \nu_3';\bar n)$.

Formally, this is the origin of the list of parameter vectors in
Table~\ref{tab:PPclassification}.  The tables of~\cite{CartonLeBrun69b}
give $r_i$ and~$q_i$ (or~rather the equivalent quantity $p_i\defeq q_i-1$)
for $i=1,2,3$, which permits the vectors to be reconstructed.  The labeling
scheme used in Table~\ref{tab:PPclassification} is that
of~\cite{CartonLeBrun69b}, though explicit parametrizations (by $(n)$,
$(r)$ or~$(n,q,r)$, as appropriate) have been added.  For consistency the
family e.III($n$,$r$) of~\cite{CartonLeBrun69b} has been relabeled
e.IV($n$,1,$r$).  In parametrized families the parameter~$q$ is a positive
integer, $r$~is a nonzero integer, and as~stated $n\defeq\allowbreak
1/(\bar n-\nobreak1)$, if finite, can take any integral value (except for
$n=0,-1$, which would imply $\bar n=0,\infty$, which are improper values;
and for $n=-2$, which would yield the DH [and~SE] value $\bar n=1/2$).

Of special note are the five parametrized families e.I.1($n$), e.I.2($n$),
e.II($n$), e.IV.4($n$), and e.IV($n$,1,$\infty$).  Each has been defined so
that it is `pseudo-Euclidean,' in~that the angular parameters
$\alpha_1,\alpha_2,\alpha_3$ satisfy $\alpha_1+\alpha_2+\alpha_3=-1$.
Because of this convention, the last of these, the family
e.IV($n$,1,$\infty$), is not the formal $r\to\infty$ limit of the
separately listed family e.IV($n$,1,$r$).  Its parameters $\nu_2,\nu_2'$
have been transposed, so that $(\alpha_1,\alpha_2,\alpha_3)$ equals
$(0,-1,0)$ rather than $(0,1,0)$.  That is, the element $[2_-]\in\mathfrak
G$ has been applied.  Without this transposition the alternative parameter
vector $(\nu_1,\nobreak \nu_1'; \allowbreak \nu_2,\nobreak \nu_2';
\allowbreak \nu_3,\nobreak \nu_3';\bar n)$ would be improper in the sense
of Definition~\ref{def:proper}, because $\rho^{-1}\defeq
(1-\nobreak\alpha_1-\nobreak\alpha_2-\nobreak\alpha_3)/2$, where
$\alpha_i\defeq \nu_i'-\nobreak\nu_i$, would equal zero; and according to
Eq.~(\ref{eq:zetasntoab}), $(a_1,\nobreak a_2,\nobreak a_3;\allowbreak
b_1,\nobreak b_2,\nobreak b_3; c)$ would not be defined.  (The
gSE~(\ref{eq:gSE}) is unaffected by the transposition, but without~it
Theorem~\ref{thm:gDHbasecoin} would not apply: there would be no ${\rm
  gDH}\leftrightarrow{\rm gSE}$ correspondence.)

The reader may wonder why the HQDS $\dot x_i=x_jx_k$ (known as the
3\nobreakdash-dimensional Nahm system, a variant of the Euler--Poinsot
top), which is ${\rm gDH}(0,\nobreak 0,\nobreak 0;\allowbreak 0,\nobreak
0,\nobreak 0;1)$ and is known to be completely integrable~\cite{Fairlie87}
and to have the Painlev\'e property, is not listed in
Table~\ref{tab:PPclassification}.  In fact it is present in disguise.  It
is one of the systems on the $\mathfrak{G}$\nobreakdash-orbit of the $n=1$
system XIII.4, which is ${\rm gDH}(1,\nobreak 1,\nobreak 1;\allowbreak
0,\nobreak 0,\nobreak 0;1)$.  The two gDH systems are linearly equivalent,
being related by the involution $[1_-][2_-][3_-]\in\mathfrak{G}$, i.e.\ by
$\nu_i\leftrightarrow\nu_i'\ (\forall i)$, as one sees from the formulas of
Proposition~\ref{prop:123}.

\smallskip
For any HQDS, Painlev\'e and integrability properties can be explored with
the aid of Kovalevskaya exponents~\cite{Goriely2001,Yoshida87a}, which
provide information on the linearization of the dynamics specified by $\dot
x=Q(x)$ around each ray solution.  If the non-associative
algebra~$\mathfrak{A}$ for a $d$\nobreakdash-dimensional HQDS has an
idempotent~$p$ with accompanying ray solution
$x(\tau)=\allowbreak-(\tau-\nobreak\tau_*)^{-1}p$, the Kovalevskaya
exponents associated to~$p$ are defined as the eigenvalues of the $d\times
d$ matrix $I-\nobreak[\partial Q_i/\partial x_j]_{ij}(x=p)$.  A~necessary
condition for an HQDS to be algebraically integrable is that each
Kovalevskaya exponent be rational~\cite{Yoshida87a}; and by definition,
each must be an integer if the HQDS is to have the~PP\null.  An easy
calculation applied to the gDH system~(\ref{eq:gDH}) yields the following.

\afterpage{\clearpage\rm
\begin{landscape}
{
\small
\begin{table}[ht]
  \caption{ODEs (i.e., algebraic curves in $t,\dot t$) determining the
    solutions of many gDH systems with the Painlev\'e property.
    (For e.IV($n$,$q$,$r$) systems, $t$~is represented as~$u^r$.)}
  \begin{center}
    \begin{tabular}{|l||l|l|}
    \hline
    system & ODE & type of curve \\*
    \hline
    \hline
    \multicolumn{3}{|l|}{$n$ an arbitrary integer($\neq0,-1,-2$), with $\bar n=(n+1)/n$}\\*
    \hline
    e.IV($n$,1,$r$) & $\dot u^{n+1} = (K_1u - K_2)^n$ & rational\\
    \hline
    \hline
    \multicolumn{3}{|l|}{$n=1$ (i.e., $\bar n=2$)}\\*
    \hline
    I & $\dot t^2 = K_1t^2(2t-3) + K_2$  & elliptic; rational if $K_1K_2(K_1-K_2)=0$ \\
    II & $\dot t^2 = K_1t^2 - K_2(2t-1)$ & rational \\
    III & $\dot t^2 = K_1t^3(3t-4) + K_2$ & elliptic; rational if $K_1K_2(K_1-K_2)=0$ \\
    IV & $\dot t^2 = K_1t^3 - K_2(3t-2)$ & elliptic; rational if $K_1K_2(K_1-K_2)=0$ \\
    V & $\dot t^2 = K_1t^3(t-2) + K_2(2t-1)$ & elliptic ($j=0$); rational if $K_1K_2(K_1-K_2)=0$ \\
    VI & $\dot t^2 = K_1t^2(t^2-3t+3) - K_2t$ & elliptic; rational if $K_1K_2(K_1-K_2)=0$ \\
    \hline
    e.IV(1,2,$r$) & $\dot u^2 = K_1u^2 - K_2$ & rational \\
    e.IV(1,3,$r$) & $\dot u^2 = K_1u^3 - K_2$ & elliptic ($j=0$); rational if $K_1K_2=0$\\
    e.IV(1,4,$r$) & $\dot u^2 = K_1u^4 - K_2$ & elliptic ($j=12^3$); rational if $K_1K_2=0$\\
    \hline
    \hline
    \multicolumn{3}{|l|}{$n=2$ (i.e., $\bar n=3/2$)}\\*
    \hline
    I & $\dot t^3 = \left[K_1t^2(2t-3) + K_2\right]^2$ & elliptic ($j=0$); rational if $K_1K_2(K_1-K_2)=0$ \\
    II & $\dot t^3 = \left[K_1t^2 - K_2(2t-1)\right]^2$ & elliptic ($j=0$); rational if $K_1K_2(K_1-K_2)=0$ \\
    \hline
    e.IV(2,2,$r$) & $\dot u^3 = (K_1u^2 - K_2)^2$ & elliptic ($j=0$); rational if $K_1K_2=0$\\
    e.IV(2,3,$r$) & $\dot u^3 = (K_1u^3 - K_2)^2$ & elliptic ($j=0$); rational if $K_1K_2=0$\\
    \hline
    \hline
    \multicolumn{3}{|l|}{$n=3$ (i.e., $\bar n=4/3$)}\\*
    \hline
    I & $\dot t^4 = \left[K_1t^2 - K_2(2t-1)\right]^3$ & elliptic ($j=12^3$); rational if $K_1K_2(K_1-K_2)=0$\\
    \hline
    e.IV(3,2,$r$) & $\dot u^4 = (K_1u^2 - K_2)^3$ & elliptic ($j=12^3$); rational if $K_1K_2=0$\\
    \hline
    \hline
    \multicolumn{3}{|l|}{$n=5$ (i.e., $\bar n=6/5$)}\\*
    \hline
    e.IV(5,2,$r$) & $\dot u^6 = (K_1u^2 - K_2)^5$ & hyperelliptic, $\textrm{genus}=2$; rational if $K_1K_2=0$\\
    \hline
    \hline
    \multicolumn{3}{|l|}{$n=\infty$ (i.e., $\bar n=1$)}\\*
    \hline
    e.IV($\infty$,1,$r$) & $\dot u = K_1u - K_2$ & rational\\
    e.IV($\infty$,2,$r$) & $\dot u = K_1u^2 - K_2$ & rational\\
    \hline
    \end{tabular}
  \end{center}
\label{tab:PPintegration}
\end{table}
}
\end{landscape}
}  

\begin{lemma}
  The Kovalevskaya exponents associated to the seven canonical
  idempotents\/ $p=p_0,p_1,p_2,p_3,p_1',p_2',p_3'\in\mathfrak{A}$ of any\/
  {\rm(}generic\/{\rm)} gDH system that are defined in
  Proposition\/~{\rm\ref{prop:defineps}}, i.e., associated to the seven ray
  solutions along these idempotents, are respectively
  \begin{alignat*}{3}
  \mathcal{R}_0&=&&\{ -1,\bar n/(\bar n-1),\bar n/(\bar n-1) \} \eqdef \{ -1,n+1,n+1 \}  ,&&    \\
  \mathcal{R}_i&=&&\{ -1,r_i,q_i \}  ,&\qquad& i=1,2,3,   \\
  \mathcal{R}_i'&=&&\{ -1,r_i',q_i' \}  ,&\qquad& i=1,2,3,
  \end{alignat*}
  where $(r_i,r_i')=-\bar n(1/\nu_i,1/\nu'_i)$ and $(q_i,q_i')= \bar
  n\left((\nu_i-\nu_i')/\nu_i,(\nu'_i-\nu_i)/\nu'_i\right)$, as above; so
  that if defined, they satisfy
  Eqs.\ {\rm(\ref{eq:rcondition}),\allowbreak(\ref{eq:qcondition})}.  In
  each of\/ $\mathcal{R}_i$ and\/~$\mathcal{R}_i'$, the eigendirection
  corresponding to the first of the three exponents\/ {\rm(}i.e.,
  $-1${\rm)} lies along\/~$p$, and that of the third\/ {\rm(}i.e., $q_i$,
  resp.~$q'_i${\rm)} lies in the\/ $x_j-\nobreak x_k=0$ subspace.

  If for any of the seven, any Kovalevskaya exponent is formally infinite
  or undefined {\rm(}i.e.\ {\rm`$0/0$'}{\rm)}, it indicates nongenericity:
  the idempotent is absent, and the ray solution degenerates to a ray of
  constant solutions\/ {\rm(}i.e., nilpotents\/{\rm)}.
\end{lemma}

\begin{example*}
  Any proper DH system ${\rm DH}(\alpha_1,\alpha_2,\alpha_3\,|\,c)$, which
  is a (proper) gDH system with $(\nu_i,\nu_i')=\allowbreak
  (-\alpha_i/2,\alpha_i/2)$ and $\bar n=\nobreak 1/2$, has Kovalevskaya
  exponents $\mathcal{R}_0 =\allowbreak \{-1,-1,-1\}$ and $\mathcal{R}_i
  =\allowbreak \{-1,1/\alpha_i,1\}$, $\mathcal{R}_i' =\allowbreak
  \{-1,-1/\alpha_i,1\}$, for $i=1,2,3$.  If $\alpha_i=0$ for any~$i$ then
  $\mathcal{R}_i,\mathcal{R}_i'$ will formally include exponents that are
  infinite or undefined, which simply indicates that
  $\mathcal{R}_i,\mathcal{R}_i'$ should be omitted: though each element
  $e=e_i,e_i'$ satisfies $e*e\propto e$, it is a nilpotent ($e*e=0$), and
  hence there is no idempotent $p\propto e$, and no ray solution
  $x(\tau)=-(\tau-\tau_*)^{-1}p$.

  Thus the original Darboux system ${\rm DH}(0,0,0\,|\,c)$ has as its only
  ray solution $x(\tau)=-(2/c)(\tau-\tau_*)^{-1}(1,1,1)$, i.e.\ the ray
  solution~(\ref{eq:raysoln}) with all components coincident, which lies
  along the multiplicative identity element $p_0=\allowbreak
  (2/c)e_0=\allowbreak(2/c)(e_1+e_2+e_3)$.  Its set of Kovalevskaya
  exponents~$\mathcal{R}_0$ is~$\{-1,-1,-1\}$.
\end{example*}

It follows from the lemma that conditions for the PP to obtain are that
$n\defeq1/(\bar n-1)$ be a nonzero integer or~$\infty$, and that each of
$r_i,\nobreak r_i';\allowbreak q_i,\nobreak q_i'$, $i=1,2,3$, be an integer
or~$\infty$, if defined.  Thus, similar necessary conditions emerge from a
Kovalevskaya--Painlev\'e analysis as from the proper ${\rm
  gDH}\leftrightarrow{\rm gSE}$ correspondence.

\smallskip
The reader may have noticed a small lacuna in the proof of
Theorem~\ref{thm:PPclassification}.  As stated, the theorem is not an
immediate corollary of Carton-LeBrun's classification of non\nobreakdash-SE
gSE's that have the~PP\null.  This is because it does not restrict the
\emph{type} of solution $x=x(\tau)$ that cannot have a (movable) branch
point.  The ${\rm gDH}\leftrightarrow{\rm gSE}$ correspondence of
Theorem~\ref{thm:gDHbasecoin} applies only to \emph{noncoincident}
solutions, with no pair of components coinciding (necessarily, at
all~$\tau$).  To rule~out branch points in any \emph{coincident} solution
of each of the gDH's listed in Table~\ref{tab:PPclassification}, an
auxiliary argument is needed.

Suppose that two components coincide; say, $x_j\equiv x_k$.  By
substituting this into the gDH system~(\ref{eq:gDH}) one obtains a
2\nobreakdash-dimensional HQDS for $(x_i,x_j)$.  Such systems have long
been classified~\cite{Markus60}, and their integrability and Painlev\'e
properties are known.  It is known in particular that for a
2\nobreakdash-dimensional HQDS, integer Kovalevskaya exponents are
sufficient for the~PP\null.  But, the Kovalevskaya exponents of this
reduced 2\nobreakdash-dimensional HQDS must be integers, since they are a
subset of the exponents of the original gDH system: they come from
$p=p_0,p_i,p_i'$, each of which lies in the $x_j-\nobreak x_k=0$ subspace,
and in particular they are the two eigenvalues in each of
$\mathcal{R}_0,\mathcal{R}_i,\mathcal{R}_i'$ with eigendirections in that
subspace.  That is, they are $\{ -1,n+1\}$, $\{ -1,q_i \}$, $\{ -1,q_i'
\}$.  (If~any of these three pairs contains a formally infinite or
undefined exponent, indicating nilpotence rather than idempotence, it must
be omitted.)  This concludes the proof of
Theorem~\ref{thm:PPclassification}.

\smallskip
There is much to say about the non-gDH systems with the PP in
Table~\ref{tab:PPclassification}.  Strikingly, many are related by rational
morphisms, i.e.\ by rational solution-preserving maps $x=\Phi(\tilde x)$.
The maps, which include ${\bf2},{\bf3},{\bf3}_{\bf c},{\bf6},{\bf6}_{\bf
  c}$, are shown in the final column.  (Permutations of components $x_i$
and~$\tilde x_i$, if a part of the morphism, are not shown.)  The labeling
scheme indicates whether any gDH system in the table is obtained in this
way from a `base' system.  For instance, the $n=1$ (i.e., $\bar n=2$)
systems V.1, V.2, V.3 are images of the system~V by ${\bf2},{\bf6}_{\bf
  c},{\bf3}_{\bf c}$.  Moreover V.2~is the image of~V.1 by~$\bf3$, and
of~V.3 by~$\bf2$.  Thus, this quadruple of gDH systems with the~PP
illustrates the ${\bf6}_{\bf c}\sim{\bf3}\circ{\bf2} \sim
{\bf2}\circ{\bf3}_{\bf c}$ `diamond' in Figure~\ref{fig:only}!

For some `image' gDH systems in the table, no base system is given.  (For
instance, the ${n=2}$ systems include III.1 and~III.2, but no~III\null.)
In each such case, the gSE associated to the system is the image under a
rational map $t=R(\tilde t)$ of a nonlinear third-order ODE that is not of
the `triangular' gSE form~(\ref{eq:gSE}), having more than three singular
values on~$\mathbb{P}^1_{\tilde t}$.  Equivalently, the associated PE
on~$\mathbb{P}^1_t$ can be lifted along $t=R(\tilde t)$ to a linear
Fuchsian ODE on~$\mathbb{P}^1_{\tilde t}$ that has more than three singular
points, and is therefore not a~PE\null.  Generalized Schwarzian equations
with more than three singular points can be found in the tables
of~\cite{CartonLeBrun69b}, but are a bit beyond the scope of the present
paper.

Each of the base systems in Table~\ref{tab:PPclassification}, or more
precisely the corresponding gSE's, can be
integrated~\cite{CartonLeBrun69b,Garnier12}.  PE\nobreakdash-based
integration suffices.  In the notation of Theorem~\ref{thm:integration},
let $\dot t=K^2(t)f^{1/\bar n}(t)$, where $\bar n=(n+1)/n$ and $f$~is a
nonzero solution of the associated PE~(\ref{eq:PE}), with exponents
$(\mu_1,\nobreak \mu_1'; \allowbreak \mu_2,\nobreak \mu_2'; \allowbreak
\mu_3,\nobreak \mu_3')$ computed from $(\nu_1,\nobreak \nu_1'; \allowbreak
\nu_2,\nobreak \nu_2'; \allowbreak \nu_3,\nobreak \nu_3')$ and the offset
vector~$\kappa$ by~(\ref{eq:zetas}).  But (see the Remark after that
theorem), $K\equiv\textrm{const}$ when $(t_1,t_2,t_3)=(0,1,\infty)$ and
$\kappa=\allowbreak(0,0,1)$.  In this case one can write $\dot
t^{n+1}\propto f^n(t)$, where $f=\allowbreak K_1f^{(1)}+\nobreak
K_2f^{(2)}$ is the general solution of the~PE\null.  This shows that the
space of possible $t=t(\tau)$, which are solutions of the
gSE~(\ref{eq:gSE}), is closed under affine transformations
$\tau\mapsto\allowbreak A\tau+\nobreak B$, as one expects.  For nearly all
base systems in the table the PE turns~out to have trivial monodromy, which
implies that the PE solutions $f^{(1)}(t),f^{(2)}(t)$ are polynomials.  If
this is the case then $t,\dot t$~are algebraically related: they are
functions on an algebraic curve.

\begin{theorem}
  Of the proper gDH systems with the PP listed in
  Table\/~{\rm\ref{tab:PPclassification}}, the base systems are integrated
  as follows.  For each, a first-order ODE for the gSE solution\/
  $t=t(\tau)$ {\rm(}i.e., an algebraic curve\/ in~$t,\dot t$\,{\rm)},
  where\/ $t=-(x_2\nobreak -x_3)/\allowbreak(x_1\nobreak -x_2)$, is given
  in Table\/~{\rm\ref{tab:PPintegration}}.  The parameters\/
  $K_1,K_2\in\mathbb{C}$ are free, and after\/ $t=t(\tau)$ is obtained, the
  corresponding noncoincident gDH solution\/ $x=x(\tau)$ is computed from
  the\/ $t(\cdot)\mapsto x(\cdot)$ map\/~{\rm(\ref{eq:zerothclaim})}.

  The five pseudo-Euclidean families at the head of
  Table\/~{\rm\ref{tab:PPclassification}}, namely\/ {\rm e.I.1(}$n${\rm)},
  {\rm e.I.2(}$n${\rm)}, {\rm e.II(}$n${\rm)}, {\rm e.IV.4(}$n${\rm)} and
  {\rm e.IV(}$n$,$1$,$\infty${\rm)}, omitted from
  Table\/~{\rm\ref{tab:PPintegration}} because the associated PE's have
  nonpolynomial solutions, are integrated thus.  Let
  \begin{displaymath}
    \sigma_C(\tau)\defeq
    \left\{
    \begin{alignedat}{2}
      &C+\tau^{n+1}, &\qquad& n\neq\infty,\\
      &C+e^\tau, &\qquad& n = \infty,
    \end{alignedat}
    \right.
  \end{displaymath}
  where\/ $C\in\mathbb{C}$ is a free parameter.  Then up to an affine
  transformation\/ $\tau\mapsto\allowbreak A\tau+\nobreak B$, the
  \emph{general} solution\/ $t=t(\tau)$ of the gSE associated to a gDH
  system in any of them is\/ $t=\psi(\sigma_C(\tau))$, where\/
  $\psi=\psi(\sigma)$ is any nonconstant function satisfying
    \begin{displaymath}
      \left\{
      \begin{alignedat}{2}
        & ({\rm d}\psi/{\rm d}\sigma)^6=\psi^4(\psi-1)^3,
        &\qquad&\text{for\/ {\rm e.I.1(}$n${\rm);}}\\ & ({\rm d}\psi/{\rm
          d}\sigma)^4=\psi^3(\psi-1)^2, &\qquad&\text{for\/ {\rm
            e.I.2(}$n${\rm);}}\\ & ({\rm d}\psi/{\rm
          d}\sigma)^3=\psi^2(\psi-1)^2, &\qquad&\text{for\/ {\rm
            e.II(}$n${\rm);}}\\ & ({\rm d}\psi/{\rm
          d}\sigma)^2=\psi^1(\psi-1)^2,&\qquad&\text{for\/ {\rm
            e.IV.4(}$n${\rm)}};\\ & ({\rm d}\psi/{\rm
          d}\sigma)^1=\psi^1(\psi-1)^0,&\qquad&\text{for\/ {\rm
            e.IV(}$n$,$1$,$\infty${\rm)}}.\end{alignedat}\right.\end{displaymath}There is also a \emph{special} solution that can be chosen to be\/
    $t=\psi(\tau)$.
\label{thm:PPintegration}
\end{theorem}

\begin{remark*}
  For each of the five pseudo-Euclidean families, the curve in $\psi,{\rm
    d}\psi/{\rm d}\sigma$ is of the form ${\rm d}\psi/{\rm
    d}\sigma=\allowbreak \psi^{\alpha_1+1}(\psi-\nobreak 1)^{\alpha_2+1}$.
  The family {\rm e.IV.4(}$n${\rm)} is included in this theorem for
  completeness, but is not a family of base systems.  As is indicated in
  Table~\ref{tab:PPclassification}, {\rm e.IV.4(}$n${\rm)} comes from {\rm
    e.IV(}$n$,$1$,$\infty${\rm)} via~$\bf2$.  One can simply choose
  \begin{equation}
    \psi(\sigma)=\left(\frac{e^\sigma+1}{e^\sigma-1}\right)^2,
    \qquad\text{resp.}\qquad e^\sigma,
  \end{equation}
  for {\rm e.IV.4(}$n${\rm)}, resp.\ {\rm e.IV(}$n$,$1$,$\infty${\rm)}.
\end{remark*}

\begin{proof}[Sketch of Proof]
The curves $\dot t^{n+1} = (K_1f^{(1)} + K_2f^{(2)})^n$ in
Table~\ref{tab:PPintegration} are adapted
from~\cite[Tab.~VIII]{CartonLeBrun69b}, with modifications to ensure
consistency.  For each system in Table~\ref{tab:PPintegration},
$K_1f^{(1)}+\nobreak K_2f^{(2)}$ is a solution of the associated PE\null.
When $[K_1:\nobreak K_2]=\allowbreak [1:\nobreak 0],\allowbreak[1:\nobreak
  1],\allowbreak[0:\nobreak 1]$, it is a Frobenius solution corresponding
respectively to $t=t_1,t_2,t_3$, i.e., to $t=0,1,\infty$, in the following
way: it belongs to the characteristic exponent~$\mu'_i$ at the singular
point $t=t_i$, and to the exponents $\mu_j,\mu_k$ at the other two singular
points.  Of the five pseudo-Euclidean families, e.II($n$) and
e.IV($n$,1,$\infty$) are integrated in~\cite[Tab.~VIII]{CartonLeBrun69b},
and the others are integrated similarly.
\end{proof}

\smallskip
For several illustrative base gDH systems selected from
Table~\ref{tab:PPclassification}, explicit integrations are supplied below.
(See Examples \ref{ex:1}--\ref{ex:4}.)  Most come from the ODE's given in
Table~\ref{tab:PPintegration}, via the $t(\cdot)\mapsto x(\cdot)$ map.

The simplest cases in Table~\ref{tab:PPintegration} are the ones in which
the curve in~$t,\dot t$ is rational, i.e.\ of genus~0, so that the
solutions $t=t(\tau)$ can be expressed in~terms of elementary functions.
In the cases flagged as elliptic, the curve is of genus~1 and $t=t(\tau)$
can be expressed in~terms of Weierstrass $\wp$\nobreakdash-functions
$\wp(g_2,g_3;\cdot)$.  In most though not all of these cases the
Klein--Weber invariant $j=g_2^3/(g_2^3-27g_3^2)$ of the curve is
independent of~$K_1,K_2$, and equals either~0 (indicating an equianharmonic
curve, and a triangular period lattice for $t(\cdot)$ and~$x(\cdot)$);
or~$12^3$ (indicating a~lemniscatic curve, and a square period lattice).
Incidentally, the first three curves in $\psi,{\rm d}\psi/{\rm d}\sigma$
appearing in Theorem~\ref{thm:PPintegration} are elliptic, with respective
$j$\nobreakdash-values $0,12^3,0$; the final two are rational.

If an algebraic curve in $t,\dot t$ is elliptic with $j=0$ then $t(\cdot)$
and~$x(\cdot)$ can optionally be expressed in~terms of Dixon's elliptic
functions~$\sm,\cm$, which satisfy $\sm^3+\nobreak\cm^3=1$ and
$(\sm)^{\textbf{.}}=\allowbreak \cm^2$, $(\cm)^{\textbf{.}}=\allowbreak
-\sm^2$.  (See~\cite{Dixon1890}; $\sm,\cm$ respectively equal
$6\wp/(1-3\dot\wp)$ and $(3\dot\wp+1)/(3\dot\wp-1)$, with
$(g_2,g_3)=\allowbreak (0,1/27)$ so that $\ddot\wp=6\wp^2$.)  If $j=12^3$,
they can be expressed in~terms of the lemniscatic sine/cosine functions
$\slemn,\clemn$, which satisfy $\slemn^2+\nobreak\clemn^2=1$ and
$[(\slemn)^{\textbf{.}}]^2 =\allowbreak 1-\nobreak\slemn^4$.  These are
identical to the Jacobi functions $\sn,\cn$ that have modular parameter
$m=\allowbreak k^2=-1$.  Jacobi functions with $m\allowbreak =k^2=1/2$ or
$m\allowbreak =k^2=2$ could also be used.

It is noteworthy that besides expressing $t=t(\tau)$ and $x=x(\tau)$ in
closed form, for any gDH system in Table~\ref{tab:PPintegration} it is
straightforward to construct first integrals that are rational
in~$x_1,x_2,x_3$.  It follows from $\dot t^{\bar n} =\allowbreak K_1f^{(1)}
+\nobreak K_2f^{(2)}$ that
\begin{equation}
  I\defeq \left.\frac{\dot f^{(1)}(t)-\bar n\,(\ddot t/\dot t)f^{(1)}(t)}{\dot f^{(2)}(t)-\bar n\,(\ddot t/\dot t)f^{(2)}(t)}\right|_{t=-(x_2-x_3)/(x_1-x_2)}
\end{equation}
is a constant of the motion ($\dot I=0$), i.e., is a first integral.  (A
useful expression for~$\ddot t/\dot t$ in~terms of $x_1,x_2,x_3$ is
provided by Eq.~(\ref{eq:sixthclaim}) of Lemma~\ref{lem:usedlater2}.)  The
same is true if $f^{(1)},f^{(2)}$ are replaced by any two nonzero solutions
of the~PE, i.e., by any two nonzero combinations of~$f^{(1)},f^{(2)}$.

First integrals can also be constructed as products of powers of Darboux
polynomials (sometimes called `second integrals'), if the space of such
polynomials is sufficiently rich~\cite{Goriely2001,Moulin99}.  By
definition, a Darboux polynomial (DP) of a HQDS is a homogeneous polynomial
$p\in\mathbb{C}[x_1,\dots,x_d]$ satisfying $\dot
p\defeq\allowbreak\sum_{i=1}^d Q_i\partial_ip =\allowbreak
\lambda\cdot\nobreak p$ for some `eigenvalue'
$\lambda\in\mathbb{C}[x_1,\dots,x_d]$, which if nonzero is necessarily a
homogeneous polynomial of degree~1.  For any DP~$p$, the HQDS flow
stabilizes the surface $p(x_1,\dots,x_d)=0$.  The DP's of any gDH
system~(\ref{eq:gDH}), with $d=3$, include $x_2-\nobreak x_3$,
$x_3-\nobreak x_1$, $x_1-\nobreak x_2$, so the gDH flow stabilizes their
zerosets, which are planes.  (Note that by the $x(\cdot)\mapsto t(\cdot)$
map~(\ref{eq:secondclaim}), these `planes of coincidence' correspond to
$t\equiv t_1,t_2,t_3$.)  Also, if $c-\nobreak a_i-\nobreak b_i=0$ for
any~$i$ then $x_i$~is a~DP; i.e.\ the gDH flow also stabilizes the
coordinate plane $x_i=0$.  If $c-\nobreak a_1-\nobreak a_2-\nobreak a_3=0$,
so that the system is improper, then any linear combination of
$x_2-\nobreak x_3$, $x_3-\nobreak x_1$, $x_1-\nobreak x_2$ is a~DP\null.
That is, the flow stabilizes any plane containing the ray
$(x_1,x_2,x_3)\propto\allowbreak(1,1,1)$; hence every solution $x=x(\tau)$
lies in an invariant plane.  All these statements about~DP's can be
rephrased in~terms of proper subalgebras of the non-associative
algebra~$\mathfrak{A}$.

\begin{example} ${\rm
    gDH}(1,1,0;0,0,0;1)$, i.e., the $n=1$ system~II, which has alternative
  parameter vector $(\nu_1,\nobreak \nu_1'; \allowbreak \nu_2,\nobreak
  \nu_2'; \allowbreak \nu_3,\nobreak \nu_3';\bar n) = \allowbreak
  (-2,\nobreak 0;\allowbreak -2,\nobreak0;\allowbreak 0,\nobreak 1;2)$ and
  has $(r_1,\nobreak r_1';\allowbreak r_2,\nobreak r_2';\allowbreak
  r_3,\nobreak r_3')=(1,\nobreak \infty;\allowbreak 1,\nobreak
  \infty;\allowbreak \infty,-2)$.  Explicitly, this system is
  \begin{equation}
    \label{eq:ex1gDH}
    \left\{
  \begin{aligned}
    \dot x_1 &= x_1(x_1-x_2-x_3)    ,\\
    \dot x_2 &= x_2(x_2-x_1-x_3)    ,\\
    \dot x_3 &= -x_1x_2.    
  \end{aligned}
  \right.
  \end{equation}
  This proper gDH system with the PP, invariant under $x_1\leftrightarrow
  x_2$, is listed in Table~\ref{tab:PPintegration}.  The algebraic curve
  in~$t,\dot t$ listed there,
  \begin{equation}
    \label{eq:tcurveII}
    \dot t^2 = K_1\,t^2 - K_2\, (2\,t-1),
  \end{equation}
  is rational.  Up~to the action of an affine transformation that replaces
  $\tau$ by $A\tau+\nobreak B$, solving this ODE for $t=t(\tau)$ yields the
  following general solution of the gSE: 
  \begin{equation}
    \label{eq:newguy}
    t(\tau)= \frac {c_1+\nobreak c_3+\nobreak c_2\sin\tau}{2c_3},
  \end{equation}
  where $[c_1:c_2:c_3]$ is any point on a genus\nobreakdash-0
  parametrization curve~$\mathfrak{C}$, given in projective form (i.e.,
  $\mathfrak{C}\subset\mathbb{P}^2$) as
  \begin{equation}
    \mathfrak{C}\colon\quad
    c_1^2 - c_2^2 - c_3^2 = 0.
  \end{equation}
  Substituting~(\ref{eq:newguy}) into the $t(\cdot)\mapsto x(\cdot)$
  map~(\ref{eq:zerothclaim}) yields
  \begin{equation}
    \label{eq:ex1gensoln}
    x(\tau) = \left(
    \frac{-c_2 -(c_1+c_3)\sin\tau}{ (c_1 +c_3+ c_2\sin\tau)\cos\tau},\,
    \frac{-c_2 -(c_1-c_3)\sin\tau}{(c_1 - c_3 + c_2\sin\tau)\cos\tau},\,
    -\tan\tau
    \right)
  \end{equation}
  as the general solution of the gDH system, up to an affine transformation
  that replaces $\tau$ by $A\tau+\nobreak B$ and scales $x$ by~$A$.

  The points $[c_1:c_2:c_3] = [-1:0:1],[1:0:1],[\pm1:1:0]$ on the
  curve~$\mathfrak{C}$ are special: for each, the
  solution~(\ref{eq:ex1gensoln}) is coincident, because respectively
  $t\equiv0,1,\infty$, i.e., $t\equiv t_1,t_2,t_3$, which correspond to the
  planes $x_2-\nobreak x_3=0$, $x_3-\nobreak x_1=0$, $x_1-\nobreak x_2=0$.
  But these points on~$\mathfrak{C}$ can be shown by a limiting procedure
  to yield respective special solutions $t=t^{[i']}(\tau)$, $i=1,2,3$,
  i.e.,
  \begin{equation}
    \label{eq:specsolnsII}
    t^{[1']}(\tau)=0+e^\tau, \qquad     t^{[2']}(\tau)=1+e^\tau, \qquad 
    t^{[3']}(\tau)=1/2 - \tau^2,
  \end{equation}
  in each of which $\tau\mapsto A\tau+B$ can be taken.  (As an alternative
  to the limiting procedure, in the ODE~(\ref{eq:tcurveII}) for $t=t(\tau)$
  simply set $[K_1:K_2]$ equal to $[1:\nobreak 0],\allowbreak [1:\nobreak
    1],\allowbreak [0,\nobreak 1]$, corresponding to $t=0,1,\infty$.)
  In~fact these special solutions of the gSE come from Frobenius solutions
  $f^{[1']},\allowbreak f^{[2']},\allowbreak f^{[3']}$ of the PE at
  $t=0,1,\infty$.  They are of the type shown in~(\ref{eq:rsolnsb}),
  associated respectively to the exponents $(r_1',r_2',r_3')=\allowbreak
  (\infty,\infty,-2)$; and the first two have no singularities because that
  is what $r_i'=\infty$ implies.  The general solution $t=t(\tau)$ given
  in~(\ref{eq:newguy}) comes from deformations of $f^{[1']},\allowbreak
  f^{[2']},\allowbreak f^{[3']}$.  The (noncoincident) special gDH
  solutions $x=x^{[i']}(\tau)$, $i=1,2,3$, obtained from the gSE
  solutions~(\ref{eq:specsolnsII}) by the $t(\cdot)\mapsto x(\cdot)$
  map~(\ref{eq:zerothclaim}), are
  \begin{subequations}
  \label{eq:ex1specsolns}
  \begin{align}
    x = x^{[1']}(\tau) &= (0,(1-e^\tau)^{-1},1),  \label{eq:ex1specsoln1}\\
    x = x^{[2']}(\tau) &= ((1+e^\tau)^{-1},0,1),  \label{eq:ex1specsoln2}\\
    x = x^{[3']}(\tau) &= \left(\frac{1+2\tau^2}{\tau(1-2\tau^2)},\,
    \frac{1-2\tau^2}{\tau(1+2\tau^2)}, \, \frac1\tau \right).\label{eq:ex1specsoln3}
  \end{align}
  \end{subequations}Again, an affine transformation can be applied.  
  This completes the integration of the gDH system~(\ref{eq:ex1gDH}),
  insofar as noncoincident solutions are concerned.  It should be noted
  that though the general solution~(\ref{eq:ex1gensoln}) and the special
  solutions (\ref{eq:ex1specsoln1}),(\ref{eq:ex1specsoln2}) are simply
  periodic on the $\tau$\nobreakdash-plane, the curious 5\nobreakdash-pole
  special solution~(\ref{eq:ex1specsoln3}) is not.

  Rational first integrals of this gDH system include
  \begin{equation}
    I_1 = x_3^2 - x_1x_2,\qquad
    I_2 = \frac{x_1(x_2-x_3)^2}{x_1-x_2},\qquad
    I_3 = \frac{x_2(x_3-x_1)^2}{x_1-x_2},
  \end{equation}
  which satisfy $I_1 = I_2-I_3$.  In each, each factor in the numerator and
  denominator is a Darboux polynomial.

  As is indicated in Table~\ref{tab:PPclassification}, the system ${\rm
    gDH}(1,2,0;-1,1,0;1)$, i.e., the $n=1$ system~II.1, is the image of
  this gDH system (the `base') under
  $\sigma_{(12)}\circ{\bf2}\circ\sigma_{(23)}$.  This is essentially the
  quadratic map given in~(\ref{eq:quadraticliftmap}), with interchanges of
  components.  As a solution-preserving map, it takes the above general and
  special solutions of the base system to those of the image, which like
  the base has the~PP\null.
\label{ex:1}
\end{example}

\begin{example} ${\rm
    gDH}(1,1,1;0,0,0;2)$, i.e., the $n=1$ system~V, which has alternative
  parameter vector $(\nu_1,\nobreak \nu_1'; \allowbreak \nu_2,\nobreak
  \nu_2'; \allowbreak \nu_3,\nobreak \nu_3';\bar n) = \allowbreak
  (-2,\nobreak 1;\allowbreak -2,\nobreak 1;\allowbreak -2,\nobreak1;2)$ and
  has $(r_1,\nobreak r_1';\allowbreak r_2,\nobreak r_2';\allowbreak
  r_3,\nobreak r_3')=(1,\nobreak -2;\allowbreak 1,\nobreak -2;\allowbreak
  1,-2)$.  Explicitly, this system is
  \begin{equation}
    \label{eq:ex2gDH}
    \left\{
  \begin{aligned}
    \dot x_1 &= x_1^2 - (x_1x_2 + x_2x_3 + x_3x_1),\\
    \dot x_2 &= x_2^2 - (x_1x_2 + x_2x_3 + x_3x_1),\\
    \dot x_3 &= x_3^2 - (x_1x_2 + x_2x_3 + x_3x_1).
  \end{aligned}
  \right.
  \end{equation}
  This proper gDH system with the PP, invariant under arbitrary
  permutations of $x_1,x_2,x_3$ (an~$\mathfrak{S}_3$ symmetry,
  cf.\ Theorem~\ref{thm:full2}), is listed in
  Table~\ref{tab:PPintegration}.  The algebraic curve in~$t,\dot t$ listed
  there, which is taken from~\cite{CartonLeBrun69b}, is generically
  elliptic and equianharmonic ($j=0$).  But it assumes that $t_1,t_2,t_3$
  are $0,1,\infty$.  Because of the $\mathfrak{S}_3$ symmetry and
  in~particular the~$\mathfrak{Z}_3$ (i.e., $x_1\to x_2\to x_3\to x_1$)
  symmetry, it is more convenient to choose $t_1,t_2,t_3$ to be
  $1,\omega,\omega^2$, where $\omega$~is a primitive cube root of unity.
  By examination, the curve in~$t,\dot t$ then becomes
  \begin{equation}
    \label{eq:PEforV}
    \dot t^2 = K_1\,(t-1)^3(t+1) + K_2\,(t-\omega)^3(t+\omega) + K_3\,(t-\omega^2)^3(t+\omega^2),
  \end{equation}
  in which for each~$i$, $K_i$~multiplies a Frobenius solution of the~PE
  at~$t=t_i$; so that one of $K_1,K_2,K_3$ is redundant, the PE solution
  space being 2\nobreakdash-dimensional.
  
  Up~to the action of an affine transformation that replaces $\tau$ by
  $A\tau+\nobreak B$, integrating this ODE yields the following general
  solution $t=t(\tau)$ of the gSE:
  \begin{equation}
    \label{eq:tforII}
    t(\tau)=
    \frac{\pm\sqrt{c_2c_3^3}\,\dot\wp(\tau) - 2\,c_1c_3\,\wp(\tau) + c_2^2}{4\,c_3^2\,\wp^2(\tau) + c_1c_2},
  \end{equation}
  where $[c_1:c_2:c_3]$ is any point on a parametrization
  curve~$\mathfrak{C}$, given in projective form (i.e.,
  $\mathfrak{C}\subset\mathbb{P}^2$) as
  \begin{equation}
    \mathfrak{C}\colon\quad
    c_1^3 - c_2^3 - (4/27)\,c_3^3 = 0,
  \end{equation}
  and the Weierstrass function~$\wp$ is equianharmonic, having parameters
  $(g_2,g_3)=\allowbreak (0,1/27)$.  (The curve~$\mathfrak{C}$ is elliptic,
  i.e.\ is of genus~1, and has $j=0$, i.e., is itself equianharmonic.)  The
  $t(\cdot)\mapsto x(\cdot)$ map~(\ref{eq:zerothclaim}) reduces for this
  gDH system to
  \begin{equation}
    \label{eq:ttoxII}
    x_i(\tau) =
    \frac12 \frac{{\rm d}}{{\rm d}\tau}\log
    \left[
      \frac{\dot t}{(t-\omega^{i-1})^2}
      \right]
    = \frac{\ddot t}{2\,\dot t} - \frac{\dot
      t}{t-\omega^{i-1}},\qquad i=1,2,3,
  \end{equation}
  and substituting (\ref{eq:tforII}) into~(\ref{eq:ttoxII}) yields the
  general gDH solution $x=x(\tau)$, up to an affine transformation that
  replaces $\tau$ by $A\tau+\nobreak B$ and scales $x$ by~$A$.

  The points $[c_1:c_2:c_3] = [1:1:0],[1:\omega:0],[1:\omega^2:0]$ on the
  curve~$\mathfrak{C}$ are special: for each, the resulting solution
  $x=x(\tau)$ is coincident, because respectively $t\equiv
  1,\omega,\omega^2$, i.e., $t\equiv t_1,t_2,t_3$, which correspond to the
  planes $x_2-\nobreak x_3=0$, $x_3-\nobreak x_1=0$, $x_1-\nobreak x_2=0$.
  But these points on~$\mathfrak{C}$ can be shown by a limiting procedure
  to yield respective special solutions $t=t^{[i']}(\tau)$, $i=1,2,3$,
  i.e.,
  \begin{equation}
    t^{[i']}(\tau)=\omega^{i-1}\, \frac{\tau^2+1}{\tau^2-1},\qquad i=1,2,3,
  \end{equation}
  in each of which $\tau\mapsto A\tau+B$ can be taken.  (As an alternative
  to the limiting procedure, in the ODE~(\ref{eq:PEforV}) for $t=t(\tau)$
  simply set $[K_1:K_2:K_3]$ equal to $[1:\nobreak 0:\nobreak
    0],\allowbreak [0:\nobreak 1:\nobreak 0],\allowbreak [0:\nobreak
    0:\nobreak 1]$, corresponding to $t=1,\omega,\omega^2$.)  In fact these
  special solutions of the gSE come from Frobenius solutions
  $f^{[1']},\allowbreak f^{[2']},\allowbreak f^{[3']}$ of the PE at
  $t=1,\omega,\omega^2$.  They are of the type shown in~(\ref{eq:rsolnsb}),
  associated respectively to the exponents $(r_1',r_2',r_3')=\allowbreak
  (-2,-2,-2)$; and the exponents determine their behavior
  as~$\tau\to\infty$.  The (noncoincident) special gDH solutions
  $x=x^{[i']}(\tau)$ obtained from these gSE solutions by the
  $t(\cdot)\mapsto x(\cdot)$ map~(\ref{eq:ttoxII}) include
  \begin{equation}
    \label{eq:ex2specsolns}
    x=x^{[1']}(\tau)=
    \left(\frac1{2\,\tau},\, \frac{1-3\,\tau^2}{2\,\tau(1+\tau^2)},\,
    \frac{1+3\,\tau^2}{2\,\tau(1-\tau^2)}
    \right),
  \end{equation}
  with $x^{[2']},x^{[3']}$ obtained from $x^{[1']}$ by cyclically permuting
  components.  Again, an affine transformation can be applied.  This
  completes the integration of the gDH system~(\ref{eq:ex2gDH}), insofar as
  noncoincident solutions are concerned.  It should be noted that though
  the general solution is doubly periodic in~$\tau$, the three special
  solutions, such as~(\ref{eq:ex2specsolns}), are rational (with 5~poles)
  and are not periodic.

  Rational first integrals of this gDH system include $I_1,I_2,I_3$,
  defined as
  \begin{equation}
    I_i = \frac{(x_i-x_j)^3\,P_k(x_1,x_2,x_3)}{(x_j-x_k)^3\,P_i(x_1,x_2,x_3)},
  \end{equation}
  where $i,j,k$ are a cyclic permutation of $1,2,3$.  In this, for
  $i=1,2,3$, the polynomial $P_i(x_1,x_2,x_3)$ is defined to be $3x_i^2
  -\nobreak x_1x_2 -\nobreak x_2x_3 -\nobreak x_3x_1$.  Each of
  $P_1,P_2,P_3$ is a Darboux polynomial, like $x_1-\nobreak x_2$,
  $x_2-\nobreak x_3$, $x_3-\nobreak x_1$, and the degree\nobreakdash-6
  polynomial $P_1P_2P_3$ is not merely a Darboux polynomial: it is a first
  integral.

  \smallskip
  As mentioned, and as is indicated in Table~\ref{tab:PPclassification},
  the $n=1$ systems V.1,\,V.2,\,V.3 are the images of this gDH system~(V)
  under the rational morphisms ${\bf 2},{\bf6}_{\bf c},{\bf 3}_{\bf c}$.
  (Some interchanges of components are required.)  The most interesting of
  these maps is the $\text{V}\mapsto\text{V.2}$ one, performed by
  $\sigma_{(13)}\circ{\bf6}_{\bf c}$, because it will be shown in Part~II
  that the third component of the $n=1$ system~V.2, i.e., of ${\rm
    gDH}(3,2,3;-2,-2,4;2)$, satisfies the $N=4$ case of the
  Chazy\nobreakdash-XI equation.  The Chazy\nobreakdash-XI equation is a
  third-order scalar ODE with the~PP, which is of the Chazy-class
  form~(\ref{eq:chazyclass}) with $[A:\nobreak B:\allowbreak C:\nobreak
    D]\in\mathbb{P}^1(1,1,2,3)$ parametrized by a positive integer~$N$ and
  equal to
  \begin{equation}
    [N^2-1 : N^2-13 : 12(N^2-1) : -3(N^2-1)^2].
  \end{equation}
  (See~\cite{Cosgrove2000,Sasano2010}; for the PP one needs $N\neq1$ and
  $6\notdivides N$.)  Explicitly, it is
  \begin{sizeequation}{\small}
    \label{eq:chazyxi}
    \dddot u =  \lambda(N^2-1)\,u\ddot u + \lambda(N^2-13)\,\dot u^2 + 12\,\lambda^2(N^2-1)\,u^2\dot u -3\,\lambda^3(N^2-1)^2\,u^4,
  \end{sizeequation}
  where the value of $\lambda$ is a matter of convention; it simply
  scales~$u$.  For the $N=4$ case arising from this gDH system, it happens
  that $\lambda=2/5$ is appropriate.

  According to the formula $x=\Phi(\tilde x)$ for the ${\bf 6}_{\bf c}$~map
  given in~(\ref{eq:6cmap}), the third component of the image system under
  $\sigma_{(13)}\circ{\bf6}_{\bf c}$ is simply $(\tilde x_1+\nobreak \tilde
  x_2+\nobreak \tilde x_3)/3$.  Hence if $x=x(\tau)$ is any of the above
  solutions of the $n=1$ system~V, general or special, then the average of
  components $u\defeq\allowbreak (x_1+\nobreak x_2+\nobreak x_3)/3$ will be
  a solution of Chazy\nobreak-XI(4), i.e., of Eq.~(\ref{eq:chazyxi}) with
  $N=4$ and~$\lambda=2/5$.

  Thus the general solution $x=x(\tau)$ of the gDH
  system~(\ref{eq:ex2gDH}), computed from (\ref{eq:tforII})
  by~(\ref{eq:ttoxII}), yields the general solution $u=u(\tau)$ of
  Chazy\nobreak-XI(4).  Explicitly,
  \begin{equation}
    \label{eq:genchazy4}
    u(\tau)=
    \frac12 \frac{{\rm d}}{{\rm d}\tau}\log
    \left[
      \frac{\dot t}{(t^3-1)^{2/3}}
      \right]
    =\frac{\ddot t}{2\,\dot t} - \frac{t^2\,\dot t}{t^3-1},
  \end{equation}
  where $t=t(\tau)$ is the general gSE solution~(\ref{eq:tforII}).  This
  general solution is rational in the equianharmonic Weierstrass
  functions~$\wp,\dot\wp$, and is therefore doubly periodic with a
  triangular period lattice; and it is algebraically parametrized by the
  point $[c_1:\nobreak c_2:\nobreak c_3]\in\allowbreak
  \mathfrak{C}\subset\nobreak\mathbb{P}^2$.  Similarly, each of the special
  gDH solutions $x=x^{[i']}(\tau)$, such as $x^{[1']}(\tau)$
  of~(\ref{eq:ex2specsolns}), yields the special (i.e.\ rational)
  Chazy\nobreakdash-XI(4) solution
  \begin{equation}
    \label{eq:chazysoln1}
    u = u^{[1']}(\tau) = u^{[2']}(\tau) = u^{[3']}(\tau) = \frac{1-5\,\tau^4}{2\,\tau(1+3\,\tau^4)},
  \end{equation}
  up to an affine transformation.  This too solves (\ref{eq:chazyxi}) with
  $N=4$ and~$\lambda=2/5$.

  Any Chazy solution $u=u(\tau)$ is transformed to another Chazy solution
  by any affine transformation that replaces $\tau$ by $A\tau+\nobreak B$
  and scales $u$ by~$A$.  It should be mentioned that there are additional
  special Chazy-XI(4) solutions,
  \begin{equation}
    \label{eq:chazysoln2}
    u(\tau) = \frac{1}{2\,\tau},\quad -\,\frac{5}{6\,\tau},\quad -\,\frac{1}{3\,\tau},
  \end{equation}
  which fit less well into the framework of PE-based integration.  These
  come via the averaging formula $u=(x_1+x_2+x_3)/3$ from ray solutions
  $x=x(\tau)$ of the gDH system~(\ref{eq:ex2gDH}), each of which has
  coincident components.  As is easily verified, the respective ray
  solutions are the ones along the directions $e_0$, $e_i'$, and~$e_i$.
  (See Proposition~\ref{prop:defineps}; by the $\mathfrak{S}_3$ symmetry,
  $i$~is arbitrary.)

  It is not difficult to generalize heuristically the special Chazy-XI(4)
  solutions (\ref{eq:chazysoln1}),(\ref{eq:chazysoln2}) to the case of
  arbitrary~$N$ (and trivially, to arbitrary~$\lambda$).  The
  generalizations, which solve Eq.~(\ref{eq:chazyxi}) and subsume
  (\ref{eq:chazysoln1}),\allowbreak(\ref{eq:chazysoln2}), are
  \begin{equation}
    \lambda\, u(\tau) =
    \frac{c_+\,\tau^{N/2}+c_-\,\tau^{-N/2}}{\tau\left[(1-N)c_+\,\tau^{N/2}+(1+N)c_-\,\tau^{-N/2}\right]},\quad
    \quad \frac{2}{(1-N^2)\tau},
  \end{equation}
  in the first of which $c_+,c_-\in\mathbb{C}$ are parameters satisfying
  $c_++\nobreak c_-=\nobreak 1$; equivalently, $[c_+:c_-]\in\mathbb{P}^1$.
  To both, a translation $\tau\mapsto\allowbreak \tau+B$ can be applied.
  The above special solutions (with $N=4$) are recovered by setting
  $[c_+:c_-]=\allowbreak [-5:1]$, $[0:1]$, $[1:0]$.

  The \emph{general} solution~(\ref{eq:genchazy4}) of Chazy-XI(4), rational
  in the equianharmonic Weierstrass functions~$\wp,\dot\wp$, cannot readily
  be generalized from $N=4$ to arbitrary~$N$; not even to the integral
  values of~$N$ for which Chazy\nobreakdash-XI has the Painlev\'e property.
  But it will be shown in Part~II that the individual components
  $x_i=x_i(\tau)$, $i=1,2,3$, of the gDH system~(\ref{eq:ex2gDH}) satisfy
  Chazy\nobreakdash-XI(2), with $\lambda=2/3$.  Also, it will be shown that
  the first component of the system ${\rm gDH}(4,6,4;\allowbreak
  3,-1,-1;2)$, i.e., of the $n=2$ system~III.1 listed in
  Table~\ref{tab:PPclassification}, satisfies Chazy\nobreakdash-XI(9), with
  $\lambda=1/10$.  So the general solution of Chazy\nobreakdash-XI($N$)
  with $N=2,4,9$, at~least, can be constructed by PE\nobreakdash-based
  integration, applied to a proper gDH system with the~PP\null.
  Chazy\nobreak-XI($N$) is integrable in a Liouvillian sense for
  any~$N$~\cite{Cosgrove2000}, but constructing algebraically parametrized
  closed-form solutions may be especially easy for certain integer values
  of~$N$.
\label{ex:2}
\end{example}

\begin{example}
The parametrized gDH family {\rm{e.IV(}}$n$,$q$,$r${\rm)}, defined by either of
\begin{subequations}
\begin{sizealign}{\small}
(a_1, a_2, a_3;\, b_1, b_2, b_3;\,c)&\propto \bigl(n+1, r(n+1), qn-n-1;\, qn-n-1,
  -q, n+1 ;\, qn\bigr),\label{eq:origparvec}\\
\bigl(\nu_1, \nu'_1;\, \nu_2, \nu'_2;\, \nu_3, \nu'_3\bigr)&= \bigl(-\tfrac{\bar n}r,
  \tfrac{q-\bar n}r;\, -\bar n, -\bar n+1;\, -\tfrac{q-\bar n}r, \tfrac{\bar
    n}r\bigr),\label{eq:newparvec}
\end{sizealign}
\end{subequations}
where as usual $\bar n\defeq(n+1)/n$, with $n$ restricted by
$n\in\mathbb{Z}\setminus\{0,-1,-2\}$; and where $q$~is a positive integer,
and $r$~a nonzero one.  (By assumption $n\neq\infty$ here, so that $\bar
n\neq1$; the family {\rm{e.IV(}}$\infty$,$q$,$r${\rm)} is best treated
separately.)

One notes that the second angular parameter, $\alpha_2\defeq
\nu_2'-\nobreak\nu_2$, always equals~1.  Also, from the formula
$(r_i,r_i')=-\bar n(1/\nu_i,1/\nu'_i)$, it follows that the vector of
leading exponents of the Frobenius-derived gSE solutions shown
in~(\ref{eq:rsolns}) is
\begin{equation}
(r_1,r_1';\, r_2, r_2';\, r_3, r_3')=
\bigl(r,\tfrac{-(n+ 1)r}{qn-  n-  1};\, 1, n+  1;\, \tfrac{(n+  1)r}{qn-  n- 1},-r\bigr),
\end{equation}
in which the pair $(r_2,r_2')=(1,n+1)$ is what one would expect of an
\emph{ordinary}, nonsingular point.  In fact, in the gSE associated to any
gDH system in the family {\rm{e.IV(}}$n$,$q$,$r${\rm)}, the second
`singular point' $t=t_2$ is not a singular point at~all: it is ordinary.
This is because, according to~(\ref{eq:newparvec}),
\begin{equation}
(\nu_1, \nu'_1;\, \nu_2, \nu'_2;\, \nu_3, \nu'_3)= (\nu,\nu';\, -\bar n,-\bar
  n+1;\, -\nu',-\nu)  
\end{equation}
for certain $\nu,\nu'$.  If the vector of offset exponents is of this
special form, then no poles will appear at~$t=t_2$ in the
gSE~(\ref{eq:gSE}), as one can easily check.

A~related degeneracy is visible in the gDH parameter
vector~(\ref{eq:origparvec}).  It satisfies $c-\nobreak a_1-\nobreak b_1=0$
and $c-\nobreak a_3-\nobreak b_3=0$.  From this, it follows that the gDH
flow stabilizes the coordinate planes $x_1=0$ and $x_3=0$.  Moreover, any
system {\rm{e.IV(}}$n$,$q$,$r${\rm)} is `triangular': $x_1,x_3$ evolve
independently of~$x_2$.  Although the opposite is not true, the
non-associative algebra~$\mathfrak{A}$ associated to the system does have
as a proper subalgebra the span of~$e_2$.

The systems in this family that have the PP are listed in both Tables
\ref{tab:PPclassification} and~\ref{tab:PPintegration}.  For
e.IV($n$,$q$,$r$) and its associated gSE to have the~PP, $(n,q)$ must be
tightly restricted, because of the following considerations.  The auxiliary
exponents $q_i,q_i'$ (see~(\ref{eq:qsolns})) are defined by
$(q_i,q_i')=\allowbreak \bar
n\left((\nu_i-\nobreak\nu_i')/\nu_i,\allowbreak
(\nu'_i-\nobreak\nu_i)/\nu'_i\right)$, so that
\begin{equation}
\label{eq:certaincondition}
(n+1)^{-1} + q_i^{-1}  + q_i'^{-1} = 1, \qquad i=1,2,3,
\end{equation}
as was previously noted.  Moreover, one calculates
from~(\ref{eq:newparvec}) that
\begin{equation}
(q_1,q_1';\,q_2,q_2';\, q_3,q_3')= (q,q';\,1,-(n+1);\,q',q)
\end{equation}
for a certain~$q'$ (namely, $q'=\frac{(n+1)q}{qn-n-1}$.)
Equation~(\ref{eq:certaincondition}) thus requires that $n+\nobreak1$ (an
integer other than $1,0,-1$), $q$~(an integer), and $q'$ (an integer
or~$\infty$) be related by the sum of their reciprocals equaling~1.  Hence,
$\{n+\nobreak1,\nobreak q,\nobreak q'\}$ must be one of
$\{n+\nobreak1,\nobreak 1,\nobreak -(n+\nobreak1)\}$, $\{2,2,\infty\}$,
$\{2,3,6\}$, $\{2,4,4\}$, or $\{3,3,3\}$; so that
\begin{equation}
  \label{eq:manycases}
  (n,q) = (n,1),\, (1,2),\,  (1,3),\,  (1,4),\,  (2,2),\,  (2,3),\, 
 (3,2),\,  (5,2),
\end{equation}
up~to the interchange $q\leftrightarrow q'$, or equivalently up~to
$x_1\leftrightarrow x_3$.  These choices yield the systems
e.IV($n$,$q$,$r$) with the PP that appear in Tables
\ref{tab:PPclassification} and~\ref{tab:PPintegration}.  It should be noted
that the final choice $(n,q)=\allowbreak (5,2)$ is anomalous.  As is
indicated in Table~\ref{tab:PPclassification}, the system
{\rm{e.IV(}}$5$,$2$,$r${\rm)} will have the PP only if $r$~is even.

For any system e.IV($n$,$q$,$r$), having the PP or not, it follows readily
from PE\nobreakdash-based integration that the algebraic curve in~$t,\dot
t$ that determines the gSE and gDH solutions (i.e., the ODE satisfied by
$t=t(\tau)$) is
\begin{equation}
\label{eq:newcurve}
\dot u^{n+1} =\allowbreak (K_1u^q -\nobreak K_2)^n, \qquad t\defeq u^r,
\end{equation}
as is indicated in Table~\ref{tab:PPintegration}.  This is a statement that
the curve in~$t,\dot t$ is the image under a cyclic covering $t=R(u)\defeq
u^r$ of a curve in~$u,\dot u$; and associated to this covering there must
necessarily be a rational morphism (i.e.\ solution-preserving map)
$x=\Phi(\tilde x)$, from {\rm{e.IV(}}$n$,$q$,$1${\rm)} to
{\rm{e.IV(}}$n$,$q$,$r${\rm)}.  The latter system is therefore an image
system, and its integration is facilitated by the existence of the map.  It
follows from the useful formula~(\ref{eq:fund}) that the map is
  \begin{equation}
    \label{eq:rmorphism}
    x_1=\tilde x_1,\qquad
    x_2= \frac{\tilde x_1(\tilde x_3-\tilde x_2)^r - \tilde x_3(\tilde
      x_1-\tilde x_2)^r}{(\tilde x_3-\tilde x_2)^r - (\tilde x_1-\tilde
      x_2)^r},\qquad
    x_3=\tilde x_3,
  \end{equation}
which is of a new type.  It did not arise in~\S\,\ref{sec:transformations}
because the PE associated to any {\rm{e.IV(}}$n$,$q$,$r${\rm)}, having only
two singular points on~$\mathbb{P}^1_t$ as explained above (i.e., $t_1,t_3$
but not~$t_2$), is of a degenerate type that was not considered there.

Written explicitly, even the base system e.IV($n$,$q$,1) is disconcertingly
complicated:
\begin{equation}
  \label{eq:ex3gDH}
  \left\{
  \begin{aligned}
    \dot {\tilde x}_1 &= {\tilde x}_1\bigl[(n+1){\tilde x}_1 - (n+q+1){\tilde x}_3 \bigr]    ,\\
    \dot {\tilde x}_2 &= (n+1){\tilde x}_2^2 + {\tilde x}_3 \bigl[(qn-2n-2){\tilde x}_2 - (q-1)(n+1){\tilde x}_1 \bigr]    ,\\
    \dot {\tilde x}_3 &= {\tilde x}_3\bigl[(qn-n-1){\tilde x}_3 -
      (q-1)(n+1){\tilde x}_1       \bigr]  .
  \end{aligned}
  \right.
\end{equation}
But the algebraic curve in~$u,\dot u$ (i.e.,~$\tilde t,\dot{\tilde t}$)
used in its integration will simply be the curve~(\ref{eq:newcurve}), which
is easy to integrate.  Once the gSE associated to the base system
e.IV($n$,$q$,1) has been integrated, i.e., its general and special
solutions $t=t(\tau)$ have been found, the gDH solutions $\tilde x=\tilde
x(\tau)$ are obtained by applying the $t(\cdot)\mapsto x(\cdot)$ map.  From
these, explicit solutions of the image system e.IV($n$,$q$,$r$), for $r$
any nonzero integer, can be obtained by applying the solution-preserving
map~(\ref{eq:rmorphism}).

As Table~\ref{tab:PPintegration} indicates, the curve~(\ref{eq:newcurve})
in~$u,\dot u$ is rational for $(n,q)$ equal to $(n,1)$ or $(1,2)$;
generically elliptic (with $j=0$) for $(n,q)$ equal to $(1,3)$, $(2,2)$,
or~$(2,3)$; and generically elliptic (with $j=12^3$) for $(n,q)$ equal to
$(1,4)$ or~$(3,2)$.  In the rational cases the construction of explicit gSE
solutions $t=t(\tau)$ proceeds similarly to Example~\ref{ex:1}; and in the
elliptic cases, to Example~\ref{ex:2}.  (Details are left to the reader.)
Only the case $(n,q)=(5,2)$ requires special comment.  As is indicated in
Table~\ref{tab:PPintegration}, for this choice the
curve~(\ref{eq:newcurve}) is generally hyperelliptic.  Its general integral
$u=u(\tau)$ is not single-valued on the complex $\tau$\nobreakdash-plane;
but if $r$~is even then $t\defeq u^r$ will be single-valued.  In~fact, it
will be an equianharmonic elliptic function of~$\tau$.  Another way of
saying this is that if $v\defeq u^2$ then the resulting curve in~$v,\dot
v$, which is
\begin{equation}
\dot v^6 = 64\, v^3(K_1\,v - K_2)^5,
\end{equation}
is elliptic (with $j=0$), provided that $K_1K_2\neq0$.  This completes our
summary of the integration of the cases~(\ref{eq:manycases}) of the gDH
system~(\ref{eq:ex3gDH}) that have the~PP, insofar as noncoincident
solutions are concerned.

Rational first integrals of the base gDH system e.IV($n$,$q$,1) include
\begin{equation}
  I_1 = {\tilde x}_1^{(q-1)n-1}\, {\tilde x_3}^{n+1} (\tilde x_3-\tilde x_1)^q, \qquad
  I_2 = \frac{{\tilde{x}}_1({\tilde{x}}_3-{\tilde{x}}_2)^q}{{\tilde{x}}_3({\tilde{x}}_1-{\tilde{x}}_2)^q}.
\label{eq:r1pair}
\end{equation}
In each, each factor is a Darboux polynomial.  For any gDH system
e.IV($n$,$q$,$r$) with $r$ a nonzero integer other than~$\pm1$, the
corresponding first integrals are algebraic rather than rational.  They can
be derived from~(\ref{eq:r1pair}) by applying a irrational but algebraic
morphism that is the inverse of the morphism~(\ref{eq:rmorphism}).

It must be stressed that $I_1,I_2$ of~(\ref{eq:r1pair}) are first integrals
of the system~(\ref{eq:ex3gDH}), i.e., of
\begin{displaymath}
\text{e.IV}(n,q,1) = {\rm gDH}(n+1, n+1, qn-n-1;\, qn-n-1, -q, n+1 ;\, qn),
\end{displaymath}
whether or not it appears in Table~\ref{tab:PPintegration}, i.e., whether
or not the pair $(n,q)$ takes on one of the relatively few values that
cause it to have the Painlev\'e property.  The parameters $n,q$ do not even
need to take~on integral values.
\label{ex:3}
\end{example}

\begin{example}
    ${\rm gDH}(0,0,0;0,1,0;1)$, i.e., the $n=\infty$ system {\rm
    e.IV(}$\infty$,$1$,$\infty${\rm)}, which has alternative parameter
  vector $(\nu_1,\nobreak \nu_1'; \allowbreak \nu_2,\nobreak \nu_2'; \allowbreak \nu_3,\nobreak \nu_3';\bar n) = \allowbreak (0,\nobreak 0;\allowbreak
  0,\nobreak -1;\allowbreak 0,\nobreak0;1)$ and has $(r_1,\nobreak
  r_1';\allowbreak r_2,\nobreak r_2';\allowbreak r_3,\nobreak
  r_3')=(\infty,\nobreak \infty;\allowbreak \infty,\nobreak 1;\allowbreak
  \infty,\infty)$.  Explicitly, this system is
  \begin{equation}
    \label{eq:ex4gDH}
    \left\{
  \begin{aligned}
    \dot x_1 &= x_3(x_1-x_2),\\
    \dot x_2 &= 0,\\
    \dot x_3 &= x_1(x_3-x_2).    
  \end{aligned}
  \right.
  \end{equation}
  This proper but rather degenerate gDH system with the PP, invariant under
  $x_1\leftrightarrow x_3$, does not appear in
  Table~\ref{tab:PPintegration}: it is the $n=\infty$ member of {\rm
    e.IV(}$\infty$,$1$,$n${\rm)}, the fifth pseudo-Euclidean family treated
  in Theorem~\ref{thm:PPclassification}.  The curve in~$t,\dot t$ is not
  algebraic but transcendental.  It was not given in the theorem but is
  \begin{equation}
    \label{eq:finalode}
    \dot t = K_1(t\log t) + K_2\,t,
  \end{equation}
  the right side being the solution $K_1f^{(1)}(t)+K_2f^{(2)}(t)$ of
  the~PE\null.  As in Example~\ref{ex:3} above, the second `singular point'
  $t=t_2$ of the~PE (i.e., $t=1$) is ordinary rather than singular: the
  only singular points are $t=t_1,t_3$, i.e., $t=0,\infty$.

  Up~to an affine transformation that replaces $\tau$ by $A\tau+\nobreak
  B$, the general solution of~(\ref{eq:finalode}) is 
  \begin{equation}
    \label{eq:finalgsesoln}
    t(\tau)=\allowbreak \exp(C+\nobreak e^\tau)\eqdef \bar C\exp(e^\tau),
  \end{equation}
  as stated in the theorem.  Substituting the gSE
  solution~(\ref{eq:finalgsesoln}) into the $t(\cdot)\mapsto x(\cdot)$
  map~(\ref{eq:zerothclaim}) yields
  \begin{equation}
    \label{eq:genfinal}
    x=x(\tau) = \left(
    \frac{1+e^\tau - \bar C\exp(e^\tau)}{1-\bar C\exp(e^\tau)}
    ,\,1,\,
    \frac{1-\bar C(1-e^\tau)\exp(e^\tau)}{1-\bar C\exp(e^\tau)}
    \right)
  \end{equation}
  as the general solution of the gDH system, up to an affine transformation
  that replaces $\tau$ by $A\tau+\nobreak B$ and scales $x$ by~$A$.  Here,
  the free parameter $\bar C\in\mathbb{C}$ can be viewed as a point on a
  genus\nobreakdash-0 parametrization curve
  $\mathfrak{C}=\mathbb{P}^1=\mathbb{C}\cup\{\infty\}$.

  The points $\bar C=0,\infty$ on~$\mathfrak{C}$ are special: for each, the
  solution~(\ref{eq:genfinal}) is coincident, because respectively $t\equiv
  0,\infty$, i.e., $t\equiv t_1,t_3$, which correspond to the planes
  $x_2-\nobreak x_3=0$, $x_1-\nobreak x_2=0$.  But these points
  on~$\mathfrak{C}$ can be shown by a limiting argument to yield special
  gSE solutions
  \begin{equation}
    t^{[1']}(\tau) = \exp(e^\tau),\qquad     t^{[3']}(\tau) = e^{\tau},
  \end{equation}
  in each of which $\tau\mapsto A\tau+B$ can be taken.  (Alternatively,
  simply set $[K_1:\nobreak K_2] =\allowbreak [0:\nobreak 1]$,
  $[1:\nobreak0]$ in~(\ref{eq:finalode}).)  The first of these is merely an
  instance of the general solution.  But the solution $t=t^{[3']}(\tau) =
  e^{\tau}$, which was mentioned in Theorem~\ref{thm:PPintegration}, comes
  from a Frobenius solution~$f^{[3']}$ of the~PE at $t=t_3=\infty$; and it
  is of the type shown in~(\ref{eq:rsolnsb}), though it has no singularity
  because $r_3'=\infty$.  It is deformations of~$f^{[3']}$ that yield the
  above general solution, $t(\tau)=\allowbreak\bar C \exp(\nobreak
  e^\tau)$.  The special gDH solution obtained from $t=t^{[3']}(\tau)$ by
  the $t(\cdot)\mapsto x(\cdot)$ map is
  \begin{equation}
    x = x^{[3']}(\tau) = \left(\frac{1}{1-e^\tau},\,0,\,\frac{e^\tau}{1-e^\tau}
    \right).
  \end{equation}
  Again, an affine transformation can be applied (e.g., the one based on a
  replacement of $\tau$ by~$-\tau$ will interchange $x_1,x_3$).  This
  completes the integration of the gDH system~(\ref{eq:ex4gDH}), insofar as
  noncoincident solutions are concerned.

  Because of the double exponentials in~(\ref{eq:genfinal}), the poles of
  the general solution $x=x(\tau)$ in the complex plane form a lattice that
  is not regular, being exponentially stretched.  Such stretching,
  polynomial for $n$~finite and exponential for $n=\infty$, is
  characteristic of gDH systems in the pseudo-Euclidean families
  e.I.1($n$), e.I.2($n$), e.II($n$), e.IV.4($n$), and
  e.IV($n$,$1$,$\infty$), as Theorem~\ref{thm:PPintegration} makes clear.

  First integrals of this gDH system include $I_1,I_2$, defined as
  \begin{equation}
    I_1 = x_2\,\log\left( \frac{x_3-x_2}{x_1-x_2}\right) + (x_3 -
    x_1),\qquad
    I_2 = x_2.
  \end{equation}
  The transcendentality is apparent.
\label{ex:4}
\end{example}

In each of the preceding examples of explicit (Papperitz-based) integration
of non-DH gDH systems with the~PP, there was a \emph{general} noncoincident
solution $x=x(\tau)$ with three free parameters, and one or more
\emph{special} noncoincident solutions having only two.  (For both types of
solution, one must include in the count the parameters
$(A,B)\in(\mathbb{C}\setminus\{0\})\times \mathbb{C}$ of the affine
transformation $\tau\mapsto\allowbreak A\tau+\nobreak B$, which can always
be applied.)  In the integration of differential equations and systems,
solutions without a full complement of free parameters are traditionally
called `singular' solutions~\cite{Davis62}.  This is what the above special
solutions are.

Generically, special gDH and gSE solutions originate in the following way.
Suppose that $\nu_i<\nu_i'$ for any of $i=1,2,3$.  There will be a gDH
solution $x=x^{[i']}(\tau)$, derived from a gSE solution $t=t^{[i']}(\tau)$
of the type shown in~(\ref{eq:rsolnsb}), with leading exponent~$r_i'$; the
gSE solution comes from a Frobenius solution of the~PE at the singular
point $t=t_i$, with leading exponent~$\mu_i'$.  There will also be a gDH
solution $x=x^{[i]}(\tau)$, derived from a gSE solution $t=t^{[i]}(\tau)$
of the type shown in~(\ref{eq:rsolnsa}), with leading exponent~$r_i$; the
gSE solution comes from a Frobenius solution of the~PE at $t=t_i$, with
exponent~$\mu_i$.  But $t=t^{[i']}(\tau)$ will be \emph{special}, with only
two free parameters, and $t=t^{[i]}(\tau)$ will lie on the
three-dimensional general-solution manifold.  This is because the dominance
assumption $\nu_i<\nu_i'$ ensures that each instance of the general
solution will have leading exponent $r_i$, not~$r_i'$.  By definition, each
such instance is based on a nontrivial mixture of the two Frobenius
solutions at~$t=t_i$; and it will be of the type shown
in~(\ref{eq:qsolnsa}), with a leading exponent ($r_i$) different from that
of the special solution~($r_i'$), which is `recessive.'

However, special solutions can be obtained from the general solution by a
limiting procedure, as has been mentioned.  It must also be mentioned that
solutions with coincident components of any gDH system are `singular' in
the traditional sense; but they do not arise from PE\nobreakdash-based
integration of the sort considered here.

\smallskip
We close by briefly discussing possible extensions of the preceding
classification of proper gDH systems with the~PP\null.  Any solution
$x=x(\tau)$ of such a system is meromorphic on its maximal domain of
definition in~$\mathbb{C}$ (which in the ${\rm
  DH}(\frac1{N_1},\frac1{N_2},\frac1{N_3})$ case is generically a
subdomain: a~disk or half-plane).  An extension would be a classification
of (proper) gDH systems having the \emph{weak} Painlev\'e property
(wPP)\null.  Solutions of such systems are allowed to be finitely branched,
i.e., to be meromorphic on a finite cover of a domain in~$\mathbb{C}$.  The
classification of proper gDH systems with the wPP will presumably require a
classification of gSE's with the wPP, going beyond the work of Garnier and
Carton-LeBrun.  Such gSE's may be numerous.

Another extension would be a fuller study of the integrability properties
of gDH systems.  Under what circumstances are there two (algebraically
independent) first integrals, say rational or algebraic in~$x_1,x_2,x_3$?
The integrability properties of HQDS's of Lotka--Volterra type have been
exhaustively investigated, as was mentioned in the Introduction.  But the
class of gDH systems, which is nearly disjoint, remains to be treated.  It
is clear that on account of Papperitz-based integration, proper gDH systems
may be not~only Liouvillian-integrable, but also integrable in the
preceding sense; even if they lack the~PP\null.  (See, e.g., the remarks at
the end of Example~\ref{ex:3} above.)

As a further example, let the parameters $(a_1,a_2,a_3;\allowbreak
b_1,b_2,b_3;c)$ equal $(a,a,a;\allowbreak b,b,b;c)$.  The resulting
$\mathfrak{S}_3$\nobreakdash-invariant gDH system can be denoted by ${\rm
  gDH}(a;b;c)$.  If $c-\nobreak a-\nobreak b=0$, the gDH flow stabilizes
each coordinate plane $x_i=0$, and the first integrals include
\begin{equation}
I_i = \frac{x_i(x_j-x_k)}{x_j(x_k-x_i)},\qquad i=1,2,3,
\end{equation}
where $i,j,k$ is a cyclic permutation of $1,2,3$.  (These satisfy
$I_1I_2I_3=1$.)  In each of these, each factor is a Darboux polynomial.
The integrability of ${\rm gDH}(a;b;c)$ when $c-\nobreak a-\nobreak b=0$
has been noticed elsewhere~\cite[Ex.~2.5]{Kinyon97}.  The special case
${\rm gDH}(0;1;1)$ is a Lotka--Volterra model that has a
bi\nobreakdash-Hamiltonian structure, and its integration can be
carried~out in a manner that respects the
structure~\cite{Grammaticos89,Gumral93}.

This paper has said little about \emph{improper} gDH systems, because they
are not amenable to Papperitz-based integration; and because if the
parameter vector $(a_1,a_2,a_3;\allowbreak b_1,b_2,b_3;c)$ is not proper,
the ${\rm gDH}\leftrightarrow{\rm gSE}$ correspondence of
Theorem~\ref{thm:gDHbasecoin} breaks down.  However, many improper gDH
systems are rationally integrable.  If $c-\nobreak a_1-\nobreak
a_2-\nobreak a_3=0$ then the system ${\rm gDH}(a_1,a_2,a_3;\allowbreak
b_1,b_2,b_3;c)$, improper by definition, stabilizes any plane containing
the ray $(x_1,x_2,x_3)\propto\allowbreak(1,1,1)$, and has first integrals
\begin{equation}
I_i = \frac{x_i-x_j}{x_k-x_i},\qquad i=1,2,3.
\end{equation}
This makes possible the integration of any ${\rm DH}(a_1,a_2,a_3;c)$ system
that is improper on account of $c-\nobreak a_1-\nobreak a_2-\nobreak a_3$
equaling zero~\cite[\S\,4.3]{Maciejewski98}, even though no \emph{proper}
${\rm DH}(a_1,a_2,a_3;c)$, i.e.\ no~system ${\rm
  DH}(\alpha_1,\alpha_2,\alpha_3\,|\,c)$, is even algebraically
integrable~\cite{Maciejewski95,Valls2006b}.  This fact also lies behind
several integrations of improper non\nobreakdash-DH gDH systems that have
appeared in the literature.  One is the Kasner system ${\rm
  gDH}(1,1,1;\allowbreak 1,1,1;3)$, which is of historic interest.
(See~\cite{Kasner25,Kinyon97} and~\cite[\S\,5.3]{Walcher91}.)

Any complete Painlev\'e analysis of improper gDH systems must examine all
systems that are classified as improper because $c=0$ or $2c-b_1-\nobreak
b_2-\nobreak b_3=0$, as~well as those with $c-a_1-\nobreak a_2-\nobreak
a_3=0$.  Owing to the absence of the ${\rm gDH}\leftrightarrow{\rm gSE}$
correspondence, a classification of the improper gDH systems that have
the~PP will not be attempted here.








\end{document}